\documentclass[twoside,10pt]{article}
\usepackage{amsmath,amsthm,amssymb,mathrsfs}
\usepackage{paralist,graphicx,psfrag}
\usepackage[hidelinks]{hyperref}
\allowdisplaybreaks[2]
\numberwithin{equation}{section}

\def\t{\theta}
\def\T{\Theta}
\def\k{\kappa}

\def\g{\gamma}
\def\d{\delta}

\def\z{\zeta}

\def\di{\displaystyle}

\newtheorem{theorem}{Theorem}[section]
\newtheorem{lemma}[theorem]{Lemma}
\newtheorem{proposition}{Proposition}[section]

\newtheorem{remark}{Remark}
\title{Global stability    of  superposition of viscous contact wave and rarefaction waves for compressible Navier-Stokes system  with temperature-dependent transport coefficients and large data}
\author{Hongyu Wang$^a$, Rong Zhang$^b$ \thanks{
		Email addresses: {hyuwang\_a@163.com} (H. Y. Wang),   rzhang0921@gmail.com (R.
		Zhang).} \\[3mm]  a. School of Mathematics and Computer Science,\\ Nanchang University,  Nanchang 330031, P. R. China; \\
	b. School of Mathematics and Computer Science, \\Nanchang University,  Nanchang 330031, P. R. China.}
\date{}
\begin{document}
\maketitle
\begin{abstract}
We study the Cauchy problem for the one-dimensional full compressible
Navier--Stokes equations with viscosity
$\mu(\theta)=\tilde\mu\theta^\alpha$ and heat conductivity
$\kappa(\theta)=\tilde\kappa\theta^\beta$. For each fixed $\beta\ge0$,  we prove   the global stability of  the combination of a viscous contact wave with rarefaction waves, provided that the
viscosity exponent $\alpha$ and the total wave strength are sufficiently
small. The initial perturbation may be large in $H^1$, and
the specific volume and temperature are assumed to have positive
initial lower bounds. The solution  remains uniformly bounded in $H^1$ relative to
the wave, and converges uniformly to that wave as time tends to
infinity. We show that both the specific
volume and the temperature admit  time-uniform lower and upper bound.
  Estimates for the logarithmic volume derivative
and the temperature gradient then close the $H^1$ estimates without
requiring a second derivative of the initial specific volume.
\end{abstract}

\section{Introduction}
\renewcommand{\theequation}{\arabic{section}.\arabic{equation}}
\setcounter{equation}{0}

The one-dimensional compressible Navier--Stokes system for a viscous heat-conducting perfect polytropic gas takes the following form in Lagrangian coordinates; see \cite{Batchelor,Serrin}:
\begin{equation}\label{ns}
\left\{
\begin{array}{ll}
\di v_t-u_x=0,\\
\di u_t+p_x=\left(\mu\frac{u_x}{v}\right)_x,\\
\di \left(e+\frac{u^2}{2}\right)_t+\left(p u\right)_x=\left(\kappa \frac{\t_x}{v}+\mu\frac{uu_x}{v}\right)_x
\end{array}
\right.
\end{equation}
for $x\in\mathbb{R}=(-\infty,+\infty)$, $t>0$, where $v(x,t)>0$, $u(x,t)$, $\t(x,t)>0$, $e(x,t)>0$ and $p(x,t)$
are the specific volume, fluid velocity, absolute temperature, internal energy
and pressure, respectively, while the positive functions $\mu$ and $\k$ denote
the viscosity and heat conduction coefficients, respectively. Here we study the ideal
polytropic fluids so that $p$ and $e$ are given by the state equations
\begin{equation}
	\label{p-e}
\di p=\frac{R\t}{v}=Av^{-\gamma}\exp\left(\frac{\gamma-1}{R}s\right),\quad e=c_{\nu}\t+\mathrm{const},
\end{equation}
where $s$ is the entropy, $\gamma>1$ is the adiabatic exponent,  $c_{\nu}=\frac{R}{\gamma-1}$ is
the specific heat, and $A$ and $R$ are positive constants. We consider
viscosity and heat conductivity proportional to possibly different
powers of the temperature:
\begin{equation}
	\di\mu =\mu(\theta)= \tilde{\mu}\theta^\alpha,\quad \kappa = \kappa(\theta)=\tilde {\kappa}\theta^\beta,
	\label{mu-k}
\end{equation}
where $\tilde{\mu},\tilde{\kappa}>0$ and $\alpha, \beta \geq 0$ are constants.

We supplement \eqref{ns} with the initial and far-field conditions
\begin{equation}\label{initial}
\left\{
\begin{array}{ll}
\di (v,u,\t)(x,0)=(v_0,u_0,\t_0)(x), &\di x\in \mathbb{R},\\
\di (v,u,\t)(\pm\infty,t)=(v_{\pm},u_{\pm},\t_{\pm}), &\di t>0,
\end{array}
\right.
\end{equation}
where $v_{\pm}(>0)$, $u_{\pm}$ and $\t_{\pm}(>0)$ are given constants, and we assume
$\inf_{\mathbb{R}}v_0>0$, $\inf_{\mathbb{R}}\t_0>0$, and $(v_0,u_0,\t_0)(\pm\infty)=(v_{\pm},u_{\pm},\t_{\pm})$
as compatibility conditions. 

For constant transport coefficients and equal far-field states, global
existence and large-time behavior have been studied extensively; see
\cite{kazhi-she,kazhikhov,jiang,jiang-2,L-L} and the references therein.
In particular, Jiang \cite{jiang,jiang-2} derived time-uniform bounds
for the specific volume in unbounded domains, and Li and Liang
\cite{L-L} obtained uniform temperature estimates and the corresponding
large-time behavior for large data.

For different end states, the large-time behavior is governed by the
Riemann solution of the associated Euler system, obtained from
\eqref{ns} by setting $\mu=\kappa=0$. Its elementary waves are shock
waves, rarefaction waves and contact discontinuities, together with
their superpositions; see \cite{Smoller}. For constant transport
coefficients, the stability of viscous shock and rarefaction waves is
well understood; see \cite{Huang-matsumura,KM,MN-85} and
\cite{LX-88,MN-86,MN-92,nishi-yang-zhao}, respectively. Viscous
contact waves were studied in
\cite{huang-ma-shi,Huang-Matsumura-Xin,Huang-Xin-Yang,Huang-Yang,huang-zhao}.
Huang, Li and Matsumura \cite{Huang-Li-matsumura} proved the asymptotic
stability of a contact--rarefaction composite wave, and Huang and Wang
\cite{H-W} later allowed large initial perturbations without imposing a
smallness condition on $\gamma-1$. More recently, Peng, Shi and Wu
\cite{Peng-Shi-Wu} sharpened the dependence of the wave-strength
threshold on the initial perturbation for this \emph{constant-coefficient}
problem.

The temperature-dependent case is substantially more delicate. The
Chapman--Enskog expansion motivates transport coefficients depending on
the temperature \cite{Cercignani,Chapman}; when $\alpha$ or $\beta$ is
positive, the diffusion may degenerate at low temperature and the
equations contain additional nonlinear terms. For equal far-field
states with constant viscosity and temperature-dependent heat
conductivity, Jenssen and Karper \cite{Jenssen} obtained weak solutions
on bounded domains, Pan and Zhang \cite{Pan-Zhang} proved global strong
solvability, and Li, Shu and Xu \cite{Li-Shu-Xu} treated the Cauchy
problem in unbounded domains. Further large-time results are given in
\cite{Huang-Shi,Li-Xu}. Results for density--temperature-dependent
transport coefficients and boundary-value problems, which are different
from \eqref{mu-k}, can be found in \cite{Duan-Guo-Zhu,Huang-Shi-Sun}.

There are also several directly relevant stability results for waves with
nonconstant transport coefficients. Huang and Liao \cite{Huang-Liao} studied a
contact--rarefaction wave under density--temperature-dependent
coefficients. For general positive functions $\mu(\theta)$ and
$\kappa(\theta)$, Dong and Guo \cite{Dong-Guo-composite} proved the
stability of an $R_1$--contact--$R_3$ composite wave with large data
through a Nishida--Smoller-type assumption that $\gamma-1$ is
sufficiently small; their theorem uses $H^2$-level initial regularity
and imposes wave-size conditions tied to $\gamma-1$. Their later work
\cite{Dong-Guo} gives the corresponding single-contact result under the
same small-$\gamma-1$ mechanism. In a complementary equal-far-field
setting, Dong and Guo \cite{Dong-Guo-Cauchy} considered the power laws
in \eqref{mu-k} for arbitrary fixed $\gamma>1$, assuming a small
viscosity exponent $\alpha$.

The present paper treats the nonconstant equilibrium  generated by different
far-field states for the power-law model \eqref{mu-k}. For every fixed
$\beta\geq0$, $\gamma>1$, and every prescribed finite $H^1$ size of the
initial perturbation, we prove stability when $\alpha$ and the total
wave strength are sufficiently small, with thresholds allowed to depend
on the prescribed data. Thus ``large data'' means that this $H^1$ size
may be chosen arbitrarily large in advance; it does not mean that the
smallness thresholds are independent of the data. The first theorem
gives global stability of a single viscous contact wave. The second
gives global stability of the $R_1$--contact--$R_3$ composite wave under
smallness of the \emph{total} strength only, without a comparability
assumption on the three constituent strengths. In particular, our
result replaces the small-$\gamma-1$ mechanism of
\cite{Dong-Guo-composite,Dong-Guo} by a small-$\alpha$ condition for
the power-law coefficients, and extends the equal-far-field result of
\cite{Dong-Guo-Cauchy} to the wave patterns considered here. The present
result does not transfer the more flexible constant-coefficient data condition in
\cite{Peng-Shi-Wu}. A key estimate controls the logarithmic volume
derivative together with the temperature gradient, so that the $H^1$
estimates close without requiring a second derivative of the initial
specific volume.

\subsection{Main results}

We first recall the viscous contact wave $(V,U,\Theta)$ for
\eqref{ns}; see \cite{Huang-Xin-Yang}. For the Euler system
\begin{equation}\label{euler}
	\left\{
	\begin{array}{ll}
		\di v_t-u_x=0,\\
		\di u_t+p_x=0,\\
		\di \left(e+\frac{u^2}{2}\right)_t+\left(p u\right)_x=0,
	\end{array}
	\right.
\end{equation}
with Riemann initial data
\begin{equation}\label{Riemann}
	(v,u,\t)(x,0)=\left\{
	\begin{array}{ll}
		\di (v_-,u_-,\t_-),&\di x<0,\\
		\di (v_+,u_+,\t_+),&\di x>0,
	\end{array}
	\right.
\end{equation}
the contact discontinuity takes the form
\begin{equation}
(\tilde V,\tilde U,\tilde\Theta)(x,t)=\left\{
\begin{array}{ll}
\di (v_-,u_-,\t_-),&\di x<0,t>0,\\
\di (v_+,u_+,\t_+),&\di x>0,t>0,
\end{array}
\right.
\end{equation}
provided that
\begin{equation}\label{RH}
\di u_-=u_+,\quad p_-\triangleq\frac{R\theta_-}{v_-}=p_+\triangleq\frac{R\t_+}{v_+}.
\end{equation}
We assume that $u_-=u_+=0$ without loss of generality.
The corresponding viscous contact wave $(V,U,\Theta)$ is a smooth
diffusion wave. To construct it, we set
$$
\frac{R\Theta}{V}=p_+,
$$
and take the leading part of the energy equation $\eqref{ns}_3$ as
\begin{equation}\label{Theta t}
\di c_{\nu}\Theta_t+p_+U_x=\left(\k_1\frac{\Theta_x}{V}\right)_x,
\end{equation}
where $\k_1=\tilde\kappa\Theta^\beta$. Equations \eqref{Theta t}
and $V_t=U_x$ give the nonlinear diffusion equation
\begin{equation}\label{1.8}
\di \Theta_t=\frac{\tilde\kappa(\gamma-1)p_+}{\gamma R^2}\left(\Theta^{\beta-1}\Theta_x\right)_x,\quad
\Theta(\pm\infty,t)=\theta_{\pm}.
\end{equation}
It has a unique self-similar solution $\Theta(x,t)=\Theta(\xi)$,
$\xi=x/\sqrt{1+t}$; see \cite{xiao-liu}. The function $\Theta(\xi)$
is increasing if $\theta_+>\theta_-$ and decreasing if
$\theta_+<\theta_-$. With $\delta=|\theta_+-\theta_-|$, it satisfies
\begin{equation}\label{1.9}
\di (1+t)|\Theta_{xx}|+(1+t)^{\frac{1}{2}}|\Theta_x|+|\Theta-\theta_{\pm}|
=O(1)\delta e^{-\frac{c_1x^2}{1+t}} \quad \mathrm{as}~ |x|\rightarrow\infty,
\end{equation}
where $c_1>0$ depends on the fixed coefficients and end temperatures.
We define the contact wave by
\begin{equation}\label{contact}
\di V=\frac{R\T}{p_+},\quad U=\frac{\k_1(\gamma-1)\Theta_x}{\gamma R\Theta},
\quad \Theta=\Theta.
\end{equation}
Then $(V,U,\Theta)$ satisfies
\begin{equation}\label{contact equation}
\left\{
\begin{array}{ll}
\di V_t-U_x=0,\\
\di U_t+\left(\frac{R\Theta}{V}\right)_x=\left(\mu_1\frac{U_x}{V}\right)_x+\tilde{R}_1,\\
\di c_{\nu}\Theta_t+p(V,\Theta)U_x=\left(\kappa_1 \frac{\Theta_x}{V}\right)_x+\mu_1\frac{U^2_x}{V}+\tilde{R}_2,
\end{array}
\right.
\end{equation}
where
\begin{equation}
\di  
\mu_1 = \tilde{\mu}\Theta^\alpha,\quad 
 \widetilde R_1=U_t-\left(\mu_1\frac{U_x}{V}\right)_x,\quad \widetilde R_2=-\mu_1\frac{U_x^2}{V}.
\label{r1r2}
\end{equation}

For the single viscous contact wave, define the perturbation by
\begin{equation}\label{perturbation-def}
(\phi,\psi,\zeta)(x,t)=(v-V,u-U,\t-\Theta)(x,t).
\end{equation}

Here $\Phi(z)=z-1-\log z$ for $z>0$.  We now state the first result.
\begin{theorem}[Viscous contact wave]\label{theorem1}
Fix $\beta\geq0$, $\gamma>1$, positive physical coefficients,
$N_0>0$, $\underline v_0>0$, $\underline\theta_0>0$, and a compact
set $\mathcal K\Subset(0,\infty)\times\mathbb R\times(0,\infty)$.
Let $(v_\pm,u_\pm,\theta_\pm)\in\mathcal K$ satisfy \eqref{RH}, let
$(V,U,\Theta)$ be the viscous contact wave in \eqref{contact}, and set
$\delta=|\theta_+-\theta_-|$.  There are positive constants
$\varepsilon_0$ and $\delta_0$, depending only on the prescribed data,
such that if
\begin{equation}\label{data}
\begin{aligned}
&0\leq\alpha<\varepsilon_0,\qquad 0\leq\delta\leq\delta_0,\\
&(\phi_0,\psi_0,\zeta_0)\in H^1(\mathbb R)^3,\qquad
  \|\left(\phi_0,\psi_0,\zeta_0\right)\|_{H^1}\leq N_0,\\
&v_0\geq\underline v_0,\qquad \theta_0\geq\underline\theta_0,
\end{aligned}
\end{equation}
then the Cauchy problem \eqref{ns}, \eqref{initial} admits a unique
global positive strong solution, and
\[
(\phi,\psi,\zeta)\in C\bigl([0,\infty);H^1(\mathbb R)^3\bigr).
\]
Moreover,
\[
0<\inf_{(x,t)\in\mathbb R\times[0,\infty)}v(x,t)
 \leq\sup_{(x,t)\in\mathbb R\times[0,\infty)}v(x,t)<\infty,
\]
\[
0<\inf_{(x,t)\in\mathbb R\times[0,\infty)}\theta(x,t)
 \leq\sup_{(x,t)\in\mathbb R\times[0,\infty)}\theta(x,t)<\infty,
\]
and
\begin{equation}\label{main-estimate}
\begin{aligned}
&\sup_{t\ge0}\|(\phi,\psi,\zeta)(t)\|_{H^1}^2\\
&\quad+\int_0^\infty\left[
 \|\phi_x\|_2^2+\|(\psi_x,\zeta_x)\|_{H^1}^2
               +\|(\psi_t,\zeta_t)\|_2^2\right]dt\le C,
\end{aligned}
\end{equation}
and
\begin{equation}\label{behavior-1}
\lim_{t\to\infty}\|(\phi,\psi,\zeta)(t)\|_\infty=0.
\end{equation}
\end{theorem}
\begin{remark}
The number $N_0$ may be arbitrarily large, but the thresholds may
depend on $N_0$, $\beta$, $\gamma$, the positive initial lower bounds,
and the compact set of end states.  No smallness of $\gamma-1$ is
assumed, and the case $\alpha=0$ is included.  The constants in the
estimate do not depend on $v_{0xx}$ or on higher norms of smooth
approximations.
\end{remark}

When the end states are connected by a 1-rarefaction, a contact
discontinuity and a 3-rarefaction, we write
\begin{equation}\label{zheng}
(v_+,u_+,\theta_+)\in R_1CR_3(v_-,u_-,\theta_-).
\end{equation}
The entropy and the two characteristic speeds are
\[
s=\frac{R}{\gamma-1}\log\frac{R\theta}{A}+R\log v,\qquad
s_\pm=\frac{R}{\gamma-1}\log\frac{R\theta_\pm}{A}+R\log v_\pm,
\]
\[
\lambda_\pm(v,s)=\pm\left(A\gamma v^{-\gamma-1}
                  e^{(\gamma-1)s/R}\right)^{1/2}.
\]
By the standard argument, there exists a unique pair of points $(v_-^m, u^m, \t_-^m)$
and $(v_+^m, u^m, \t_+^m)$ satisfying
$$
\frac{R\t_-^m}{v_-^m}=\frac{R\t_+^m}{v_+^m}\triangleq p^m,
$$
such that the points $(v_-^m,u^m,\t_-^m)$ and $(v_+^m,u^m,\t_+^m)$ belong to the 1-rarefaction wave
curve $R_-(v_-,u_-,\t_-)$ and the 3-rarefaction wave curve $R_+(v_+,u_+,\t_+)$, respectively, where
$$
R_{\pm}(v_{\pm},u_{\pm},\t_{\pm})=\left\{(v,u,\t):v>0,\ \theta>0,\quad s=s_{\pm},
u=u_{\pm}-\int_{v_{\pm}}^v\lambda_{\pm}(\eta,s_{\pm})d\eta,v\ge v_{\pm}
\right\}.
$$

We assume $u^m=0$ in what follows without loss of generality.
The 1-rarefaction wave $(v_-^r,u_-^r,\t_-^r)(\frac{x}{t})$ (respectively the 3-rarefaction wave $(v_+^r,u_+^r,\t_+^r)(\frac{x}{t})$)
connecting $(v_-,u_-,\t_-)$ and $(v_-^m,0,\t_-^m)$ (respectively $(v_+^m,0,\t_+^m)$ and $(v_+,u_+,\t_+)$) is the
weak solution of the Riemann problem of the Euler system \eqref{euler} with the following initial Riemann data
\begin{equation}\label{Riemann2}
\di(v_{\pm},u_{\pm},\t_{\pm})(x,0)=\left\{
\begin{array}{ll}
(v_{\pm}^m,0,\t_{\pm}^m), &\quad\pm x<0,\\
(v_{\pm},u_{\pm},\t_{\pm}), &\quad\pm x>0.
\end{array}
\right.
\end{equation}
Following \cite{MN-86}, we approximate the rarefaction waves by smooth
solutions $(V_\pm^r,U_\pm^r,\Theta_\pm^r)$ of \eqref{euler}, defined by
\begin{equation}\label{appro-rare}
\di\left\{
\begin{array}{ll}
\di\lambda_{\pm}(V_{\pm}^r(x,t),s_{\pm})=w_{\pm}(x,t),\\
\di U_{\pm}^r=u_{\pm}-\int_{v_{\pm}}^{V_{\pm}^r(x,t)}\lambda_{\pm}(\eta,s_{\pm})d\eta,\\
\T_{\pm}^r=\t_{\pm}(v_{\pm})^{\g-1}(V_{\pm}^r)^{1-\g},
\end{array}
\right.
\end{equation}
 where $w_-$ (respectively $w_+$) is the solution of the initial problem for the typical Burgers equation:
\begin{equation}\label{burgers}
\di\left\{
\begin{array}{ll}
\di w_t+ww_x=0,\quad(x,t)\in\mathbb{R}\times(0,\infty),\\
\di w(x,0)=\frac{w_r+w_l}{2}+\frac{w_r-w_l}{2}\tanh x,
\end{array}
\right.
\end{equation}
with $w_l=\lambda_-(v_-,s_-)$, $w_r=\lambda_-(v_-^m,s_-)$ (respectively $w_l=\lambda_+(v_+^m,s_+)$, $w_r=\lambda_+(v_+,s_+)$).

Let $(V^{cd},U^{cd},\T^{cd})(x,t)$ be the viscous contact wave constructed in \eqref{1.8} and \eqref{contact}
with $(v_{\pm},u_{\pm},\t_{\pm})$ replaced by $(v_{\pm}^m,0,\t_{\pm}^m)$, respectively.

The wave strengths are
\begin{equation*}
\begin{array}{ll}
\di\d^{r_1}=|v_-^m-v_-|+|0-u_-|+|\t_-^m-\t_-|, \quad\\[3mm] \di\d^{cd}=|\t_+^m-\t_-^m|,\\[3mm]
\di\d^{r_3}=|v_+^m-v_+|+|0-u_+|+|\t_+^m-\t_+|
\end{array}
\end{equation*}
and $\delta=\delta^{r_1}+\delta^{cd}+\delta^{r_3}$.
A constituent with zero strength is its constant end state. We define
\begin{equation}\label{ansatz}
\left(
\begin{array}{l}
\di V\\
\di U \\
\di \T
\end{array}
\right)(x,t) = \left(
\begin{array}{l}
\di V^{cd}+V_-^r+V_+^r\\
\di U^{cd}+U_-^r+U_+^r\\
\di \T^{cd}+\T_-^r+\T_+^r
\end{array}\right)(x,t)-\left(
\begin{array}{l}
\di v_-^m+v_+^m\\
\di \quad~~ 0\\
\di \t_-^m+\t_+^m
\end{array}\right),
\end{equation}
and
$$
(\phi,\psi,\zeta)(x,t)=(v-V,u-U,\t-\T)(x,t).
$$

For the composite wave $(V,U,\Theta)$ defined by \eqref{ansatz}, set
$P=R\Theta/V$, $\mu_1=\tilde\mu\Theta^\alpha$, and
$\kappa_1=\tilde\kappa\Theta^\beta$. Its residuals are
\begin{equation}\label{profile-residuals}
\begin{aligned}
\widetilde R_1&=U_t+P_x-\left(\frac{\mu_1U_x}{V}\right)_x,\\
\widetilde R_2&=c_\nu\Theta_t+PU_x
-\left(\frac{\kappa_1\Theta_x}{V}\right)_x-\frac{\mu_1U_x^2}{V}.
\end{aligned}
\end{equation}
We now state the composite-wave result.
\begin{theorem}[Composite waves]\label{theorem2}
Fix $\beta\geq0$, $\gamma>1$, positive physical coefficients,
$N_0>0$, $\underline v_0>0$, $\underline\theta_0>0$, and a compact
set $\mathcal K\Subset(0,\infty)\times\mathbb R\times(0,\infty)$.
Let $(v_\pm,u_\pm,\theta_\pm)\in\mathcal K$ satisfy
\eqref{zheng}, and let $(V,U,\Theta)$ be the composite wave defined by
\eqref{ansatz}.  Set
$\delta=\delta^{r_1}+\delta^{cd}+\delta^{r_3}$.  There are positive
constants $\varepsilon_0$ and $\delta_0$, depending only on the
prescribed data, such that if
\begin{equation}\label{composite-data}
\begin{aligned}
&0\leq\alpha<\varepsilon_0,\qquad 0\leq\delta\leq\delta_0,\\
&(\phi_0,\psi_0,\zeta_0)\in H^1(\mathbb R)^3,\qquad
  \|\left(\phi_0,\psi_0,\zeta_0\right)\|_{H^1}\leq N_0,\\
&v_0\geq\underline v_0,\qquad \theta_0\geq\underline\theta_0,
\end{aligned}
\end{equation}
then the Cauchy problem \eqref{ns}, \eqref{initial} admits a unique
global positive strong solution, and
\[
(\phi,\psi,\zeta)\in C\bigl([0,\infty);H^1(\mathbb R)^3\bigr).
\]
Furthermore,
\[
0<\inf_{(x,t)\in\mathbb R\times[0,\infty)}v(x,t)
 \leq\sup_{(x,t)\in\mathbb R\times[0,\infty)}v(x,t)<\infty,
\]
\[
0<\inf_{(x,t)\in\mathbb R\times[0,\infty)}\theta(x,t)
 \leq\sup_{(x,t)\in\mathbb R\times[0,\infty)}\theta(x,t)<\infty,
\]
and the estimates \eqref{main-estimate} and \eqref{behavior-1} hold
for $(v-V,u-U,\theta-\Theta)$.  In addition,
\begin{equation}\label{main-rarefaction-dissipation}
\int_0^\infty\!\int
  P\left[\Phi\!\left(\frac{\theta V}{\Theta v}\right)
                   +\gamma\Phi\!\left(\frac vV\right)\right]
  \left((U_-^r)_x+(U_+^r)_x\right)dxdt\le C.
\end{equation}
\end{theorem}

\begin{remark}
Only the total strength $\delta$ is required to be small; no
comparison among $\delta^{r_1}$, $\delta^{cd}$, and $\delta^{r_3}$ is
assumed.  By part~(4) of Lemma~\ref{rare-pro},
\eqref{behavior-1} yields
\[
\lim_{t\to\infty}
\left\|(v,u,\theta)-
\left(v_-^r+V^{cd}+v_+^r-v_-^m-v_+^m,
u_-^r+U^{cd}+u_+^r,
\theta_-^r+\Theta^{cd}+\theta_+^r-\theta_-^m-\theta_+^m\right)
\right\|_{L^\infty(\mathbb R)}=0.
\]
\end{remark}

\medskip
\noindent\emph{Notation.} Throughout the paper, $c$ and $C$ denote
positive generic constants.  We use the standard spaces $L^p(\Omega)$
and $H^k(\Omega)$, and omit the domain when no confusion can arise.

\medskip

We now make some comments on the analysis of this paper.
First, we consider the stability of the viscous contact wave.  We assume
the a priori estimate \eqref{space} and first establish the basic
relative-entropy estimate \eqref{basic}.  The key step is to obtain
time-uniform positive bounds for the specific volume and temperature.
Since the viscosity depends on the temperature, the equation for the
logarithmic volume gradient contains the new terms
\[
 \frac{\mu_t}{\mu}=\alpha\frac{\theta_t}{\theta},\qquad
 \frac{\mu_x}{\mu}=\alpha\frac{\theta_x}{\theta}.
\]
We derive \eqref{phi-x3} and use the small factor $\alpha$ to absorb
the two terms $\mu_t/\mu$ and $\mu_x/\mu$.  Next, adapting the localized representation
argument for the specific volume, we introduce a special cut-off function
and derive \eqref{v-repre}.  The key point is that the
$\alpha$-dependent remainder in \eqref{v-repre} is controlled while
its time-decay factor is retained; \eqref{v-repre} then yields the
time-uniform lower and upper bounds \eqref{v} for $v$.

Next, we multiply the temperature and momentum equations by
$(\zeta-\Theta)_+$ and $2\psi(\zeta-\Theta)_+$, respectively, and add
the resulting identities.  The cut-off argument together with the
high- and low-temperature tests yields the weighted temperature and
velocity dissipation needed for every fixed $\beta\geq0$, and hence the
uniform positive bounds for $\theta$.  Applying \eqref{scalar-gaussian}
with the weight $w$ in \eqref{w} controls the terms carrying the
contact-wave factor $e^{-c_1x^2/(1+t)}$.  Finally, the higher-order estimates close the
assumed a priori bound at the $H^1$ level; in particular, the argument
does not require an $L^2$ bound for $v_{0xx}$.

Second, we consider the $R_1$--contact--$R_3$ composite wave.  The key
point is to retain the rarefaction structure in the relative-entropy
identity, which yields the nonnegative dissipation in
\eqref{main-rarefaction-dissipation}, instead of treating the
rarefaction derivatives as error terms.  We combine the dissipation in
\eqref{main-rarefaction-dissipation} with
the Gaussian-weighted estimate for the contact layer and the decay
estimates for the residual terms $F$ and $G$, which enter through the
energy estimates and Sobolev inequalities.  Consequently, the terms
containing $(V^{cd}_x,U^{cd}_x,\Theta^{cd}_x)$,
$(V_\pm^r,U_\pm^r,\Theta_\pm^r)_x$, or $(\widetilde R_1,\widetilde R_2)$
are controlled by the total strength
$\delta=\delta^{r_1}+\delta^{cd}+\delta^{r_3}$, and the continuation
criterion associated with \eqref{local0:class} yields the global $H^1$ estimate
without comparing the three component strengths.

Section~2 records the decay and interaction estimates for
$(V^{cd},U^{cd},\Theta^{cd})$ and
$(V_\pm^r,U_\pm^r,\Theta_\pm^r)$. Sections~3 and 4 then prove
Theorems~\ref{theorem1} and \ref{theorem2}, respectively.

\section{Preliminaries}
\setcounter{equation}{0}

The next lemma records the pointwise decay of the viscous contact wave.
\begin{lemma}\label{decay}
Assume that $\delta=|\t_+-\t_-|\leq \d_0$ for a small positive constant $\d_0$. Then the viscous contact wave $(V,U,\Theta)$ satisfies
$$
|V-v_{\pm}|+|\Theta-\t_{\pm}|\leq O(1)\d e^{-\frac{c_1x^2}{1+t}},\qquad \pm x\ge0,
$$
$$
|\partial^k_x V|+|\partial^{k-1}_x U|+|\partial^k_x\Theta|\leq
O(1)\delta(1+t)^{-\frac{k}{2}}e^{-\frac{c_1x^2}{1+t}}, \quad k\geq 1.
$$

\end{lemma}
By \eqref{r1r2},
\begin{equation}
\di \widetilde R_1=O(1)\d(1+t)^{-\frac{3}{2}}e^{-\frac{c_1x^2}{1+t}},
\quad \widetilde R_2=O(1)\d(1+t)^{-2}e^{-\frac{c_1x^2}{1+t}}.
\end{equation}

\begin{lemma}[\cite{Huang-Li-matsumura}]\label{hlm}
Let
\begin{equation}\label{w}
w(x,t)=(1+t)^{-1/2}e^{-c_1x^2/(8(1+t))},\qquad
g(x,t)=\int_{-\infty}^xw(y,t)\,dy.
\end{equation}
If $h\in L^\infty(0,T;L^2)$, $h_x\in L^2(0,T;L^2)$ and
$h_t\in L^2(0,T;H^{-1})$, then
\begin{equation}\label{scalar-gaussian}
\int_0^T\!\int h^2w^2\,dxdt
\le C\|h(0)\|_2^2+C\int_0^T\|h_x\|_2^2dt
+c_1\int_0^T\langle h_t,hg^2\rangle\,dt.
\end{equation}
\end{lemma}

\begin{lemma}\label{lemma5}
Suppose that $v,\theta$ have fixed positive lower and finite upper bounds,
and that
\[
\sup_{0\le t\le T}\|(\phi,\psi,\zeta)(t)\|_{H^1}^2
+\int_0^T\|(\phi_x,\psi_x,\zeta_x)(t)\|_2^2dt\le C_0.
\]
For sufficiently small $\delta$,
\begin{equation}\label{important}
\int_0^T\!\int(\phi^2+\psi^2+\zeta^2)w^2\,dxdt\le C,
\end{equation}
where $C$ depends on $C_0$, the positive lower and finite upper bounds for
$v$ and $\theta$, and the fixed data.
\end{lemma}
The proof is given in Section~3. Under \eqref{space}, its constant is
$C(M)$; after the a priori estimates, it depends only on the data.
The additional composite-wave terms are treated in Section~4.

The next lemma records the properties of the Burgers solution; see \cite{MN-86}.

\begin{lemma}\label{bur-pro}
Let $w_l<w_r$ and $\tilde w=w_r-w_l\le\bar w$ for a fixed $\bar w>0$.
Then the problem \eqref{burgers} has a unique smooth global solution in time satisfying:
\begin{itemize}
\item[(1)] $w_l<w(x,t)<w_r$, $w_x>0$ $(x\in\mathbb{R},t>0)$.
\item[(2)] For $p\in[1,\infty]$, there exists some positive constant $C=C(p,w_l,\bar w)$ such that
for $\tilde w\geqq0$ and $t\geqq0$,
$$
\|w_x(t)\|_{L^p}\leq C\min\{\tilde w,\tilde w^{1/p}t^{-1+1/p}\},\quad
\|w_{xx}(t)\|_{L^p}\leq C\min\{\tilde w,t^{-1}\}.
$$
\item[(3)] If $w_l>0$, for any $(x,t)\in(-\infty,0]\times[0,\infty)$,
$$
|w(x,t)-w_l|\leq\tilde we^{-2(|x|+w_lt)},\quad
|w_x(x,t)|\leq2\tilde we^{-2(|x|+w_lt)}.
$$
\item[(4)] If $w_r<0$, for any $(x,t)\in[0,\infty)\times[0,\infty)$,
$$
|w(x,t)-w_r|\leq\tilde we^{-2(x+|w_r|t)},\quad
|w_x(x,t)|\leq2\tilde we^{-2(x+|w_r|t)}.
$$
\item[(5)] For the Riemann solution $w^r(x/t)$ of the scalar equation \eqref{burgers} with the Riemann
initial data
\begin{equation*}
w(x,0)=\left\{
\begin{array}{ll}
w_l,&\quad x<0,\\
w_r,&\quad x>0,
\end{array}
\right.
\end{equation*}
we have
$$
\lim_{t\rightarrow+\infty}\sup_{x\in\mathbb{R}}|w(x,t)-w^r(x/t)|=0.
$$
\end{itemize}

\end{lemma}

We divide $\mathbb{R}\times(0,\infty)$ into the three regions
$$
\Omega_{-}=\big\{(x,t)\big|2x<\lambda_{-}(v_{-}^m,s_{-})t\big\},
$$
$$
\Omega_{c}=\big\{(x,t)\big|\lambda_-(v_-^m,s_-)t\leq2x\leq\lambda_{+}(v_{+}^m,s_{+})t\big\},
$$
$$
\Omega_{+}=\big\{(x,t)\big|2x>\lambda_{+}(v_{+}^m,s_{+})t\big\}.
$$
We have
\begin{lemma}[\cite{H-W}]\label{rare-pro}
For any given left end state $(v_-,u_-,\t_-)$, suppose that \eqref{zheng} holds. Then the smooth rarefaction waves $(V_{\pm}^r,U_{\pm}^r,\T_{\pm}^r)$
constructed in \eqref{appro-rare} and the viscous contact wave $(V^{cd},U^{cd},\T^{cd})$ constructed in \eqref{contact} satisfy:
\begin{itemize}
\item[(1)] $(U_{\pm}^r)_x\geq0$,~$(x\in\mathbb{R},t>0)$.
\item[(2)] For $p\in[1,\infty]$, there exists a positive constant $C$ depending on the fixed positive end-state range
such that
$$
\|\big((V_{\pm}^r)_x,(U_{\pm}^r)_x,(\T_{\pm}^r)_x\big)(t)\|_{L^p}
\leq C\min\Big\{\d,\d^{1/p}t^{-1+1/p}\Big\}
$$
and
$$
\|\big((V_{\pm}^r)_{xx},(U_{\pm}^r)_{xx},(\T_{\pm}^r)_{xx}\big)(t)\|_{L^p}
\leq C\min\Big\{\d,t^{-1}\Big\}.
$$
\item[(3)] There exists a positive constant $C$ depending on the fixed positive end-state range
such that for
$$
c_0=\frac{1}{10}\min\Big\{|\lambda_-(v_-^m,s_-)|,\lambda_+(v_+^m,s_+),c_1\lambda_-^2(v_-^m,s_-),
c_1\lambda_+^2(v_+^m,s_+),1\Big\},
$$
we have 
\begin{equation*}
	(U_{\pm}^r)_x+|(V_{\pm}^r)_x|+|V_{\pm}^r-v_{\pm}^m|+|(\T_{\pm}^r)_x|+|\T_{\pm}^r-\t_{\pm}^m|
	\leq C\d e^{-c_0(|x|+t)}, \quad \text{in }\Omega_c,
\end{equation*}
\begin{equation*}
	\left\{
	\begin{array}{ll}
		|V^{cd}-v_{\mp}^m|+|V^{cd}_x|+|\T^{cd}-\t_{\mp}^m|+|U^{cd}_x|+|\T^{cd}_x|
		\leq C\d e^{-c_0(|x|+t)},\\
		(U_{\pm}^r)_x+|(V_{\pm}^r)_x|+|V_{\pm}^r-v_{\pm}^m|+|(\T_{\pm}^r)_x|+|\T_{\pm}^r-\t_{\pm}^m|
		\leq C\d e^{-c_0(|x|+t)},
	\end{array}
	\quad \text{in }\Omega_{\mp}.
	\right.
\end{equation*}
\item[(4)] For the rarefaction waves $(v_{\pm}^r,u_{\pm}^r,\t_{\pm}^r)(x/t)$ determined by
\eqref{euler} and \eqref{Riemann2}, it holds that 
$$
\lim_{t\rightarrow+\infty}\sup_{x\in\mathbb{R}}
\big|(V_{\pm}^r,U_{\pm}^r,\T_{\pm}^r)(x,t)-(v_{\pm}^r,u_{\pm}^r,\t_{\pm}^r)(x/t)\big|=0.
$$
\end{itemize}

\end{lemma}

\begin{lemma}\label{lem:contact-and-profile-estimates}
The viscous contact wave satisfies $V_t=U_x$, $P=p_+$ and
$c_\nu\Theta_t+p_+U_x=(\kappa(\Theta)\Theta_x/V)_x$.
For the solution $w$ of \eqref{burgers},
\begin{align*}
\|w_x(t)\|_\infty+\|w_{xx}(t)\|_1
 &\le C\min\{w_r-w_l,(1+t)^{-1}\},\\
\|w_x(t)^2\|_1
 &\le C(w_r-w_l)\min\{w_r-w_l,(1+t)^{-1}\}.
\end{align*}
The relations in \eqref{appro-rare} transfer the displayed bounds for
$w_x$ and $w_{xx}$ to $(V_\pm^r,U_\pm^r,\Theta_\pm^r)$. For the viscous contact wave,
$P=p_+$. In the single-contact and composite cases,
\begin{equation}\label{profile:residual-integral}
\int_0^\infty
 (\|\widetilde R_1(t)\|_1+\|\widetilde R_2(t)\|_1)^{4/3}dt
 \le C\delta^{1/3},\qquad 0\le\delta\le1.
\end{equation}
\end{lemma}
\begin{proof}
By \eqref{1.8},
\begin{align*}
V_t=\frac R{p_+}\Theta_t
 &=\frac{\tilde\kappa(\gamma-1)}{\gamma R}
                    (\Theta^{\beta-1}\Theta_x)_x=U_x.
\end{align*}
Since $P=p_+$, it follows that
\[
c_\nu\Theta_t+p_+U_x
 =\frac{\gamma R}{\gamma-1}\Theta_t
 =\left(\frac{\kappa(\Theta)\Theta_x}V\right)_x.
\]
Differentiating $x=y+tw(y,0)$ and $w(x,t)=w(y,0)$, we obtain
\[
w_x(x,t)=\frac{w_x(y,0)}{1+tw_x(y,0)},
\qquad
w_{xx}(x,t)=\frac{w_{xx}(y,0)}{(1+tw_x(y,0))^3}.
\]
The function $w_x(y,0)$ increases from zero to $(w_r-w_l)/2$ and
then decreases to zero. Changing variables with
$dx=(1+tw_x(y,0))\,dy$ yields
\[
\|w_{xx}(t)\|_1
=\int\frac{|w_{xx}(y,0)|}{(1+tw_x(y,0))^2}\,dy
=\frac{w_r-w_l}{1+t(w_r-w_l)/2}.
\]
Moreover, $0\le w_x\le(w_r-w_l)/(2+t(w_r-w_l))$ and
$\int w_x\,dx=w_r-w_l$, so
\begin{align*}
\int w_x^2dx&\le\|w_x\|_\infty\int w_xdx
 \le C(w_r-w_l)\min\{w_r-w_l,(1+t)^{-1}\}.
\end{align*}
Differentiating the relations in \eqref{appro-rare} gives
\begin{align*}
|\partial_x(V_\pm^r,U_\pm^r,\Theta_\pm^r)|
&\le C|(w_\pm)_x|,\\
|\partial_x^2(V_\pm^r,U_\pm^r,\Theta_\pm^r)|
&\le C\bigl(|(w_\pm)_{xx}|+|(w_\pm)_x|^2\bigr).
\end{align*}

For the contact wave,
\[
\|\widetilde R_1(t)\|_1\le C\delta(1+t)^{-1},\qquad
\|\widetilde R_2(t)\|_1\le C\delta^2(1+t)^{-3/2}.
\]
Each rarefaction satisfies the Euler equations
\[
\begin{cases}
(U_\pm^r)_t+\left(R\Theta_\pm^r/V_\pm^r\right)_x=0,\\
c_\nu(\Theta_\pm^r)_t
              +(R\Theta_\pm^r/V_\pm^r)(U_\pm^r)_x=0.
\end{cases}
\]
Thus its residual terms satisfy
\begin{align*}
&\left\|\left(\frac{\mu(\Theta_\pm^r)(U_\pm^r)_x}{V_\pm^r}\right)_x\right\|_1
 +\left\|\left(\frac{\kappa(\Theta_\pm^r)(\Theta_\pm^r)_x}{V_\pm^r}\right)_x\right\|_1
 +\left\|\frac{\mu(\Theta_\pm^r)|(U_\pm^r)_x|^2}{V_\pm^r}\right\|_1\\
&\qquad\le C\bigl(\|(w_\pm)_{xx}\|_1+\|(w_\pm)_x\|_2^2\bigr)
 \le C\min\{\delta,(1+t)^{-1}\}.
\end{align*}

Choose $c_0>0$ with
$2c_0<\inf|\lambda_\pm(V_\pm^r,s_\pm)|$.
The characteristic formulas for $w_\pm$ and the Gaussian decay of
$(V^{cd},U^{cd},\Theta^{cd})$ imply,
for $t\ge1$,
\begin{align*}
&\sum_{\pm}|(V_\pm^r,U_\pm^r,\Theta_\pm^r)
                         -(v_\pm^m,u^m,\theta_\pm^m)|
 \le C\delta e^{-c(|x|+t)},\qquad |x|\le c_0t/2,\\
&|(V^{cd},U^{cd},\Theta^{cd})-(v_\pm^m,u^m,\theta_\pm^m)|\\
&\quad+|(V_\mp^r,U_\mp^r,\Theta_\mp^r)-(v_\mp^m,u^m,\theta_\mp^m)|
 \le C\delta e^{-c(|x|+t)},\qquad \pm x\ge c_0t/2.
\end{align*}
Here we used $x^2/(1+t)\ge c(|x|+t)$ for $|x|\ge c_0t/2$.
Subtracting the equations for $(V^{cd},U^{cd},\Theta^{cd})$,
$(V_-^r,U_-^r,\Theta_-^r)$, and $(V_+^r,U_+^r,\Theta_+^r)$, and using
the displayed exponential bounds in $\Omega_c$, $\Omega_-$, and $\Omega_+$, we obtain,
for $0<\delta\le1$,
\begin{align*}
&\int_0^\infty
  (\|\widetilde R_1(t)\|_1+\|\widetilde R_2(t)\|_1)^{4/3}dt\\
&\quad\le C\int_0^\infty\left[
 \delta^{4/3}(1+t)^{-4/3}+\delta^{8/3}(1+t)^{-2}
       +\min\{\delta,(1+t)^{-1}\}^{4/3}+\delta^{4/3}e^{-ct}\right]dt\\
&\quad\le C\delta^{4/3}+C\delta^{8/3}
       +C\int_0^{\delta^{-1}}\delta^{4/3}dt
       +C\int_{\delta^{-1}}^\infty(1+t)^{-4/3}dt
 \le C\delta^{1/3}.
\end{align*}
When $\delta=0$, all three wave functions are constant; hence their
spatial derivatives and the residuals $\widetilde R_1$, $\widetilde R_2$
vanish.
\end{proof}

\begin{lemma}\label{lem:remaining-profile-integrals}
For some fixed $r>0$, the functions $(V,U,\Theta)$ defined by
\eqref{contact} in the single-contact case and by \eqref{ansatz} in the
composite case satisfy
\begin{align}
&\int_0^\infty\bigl(\|U_x\|_\infty^2
       +\|(\Theta_-^r)_x\|_\infty^2
       +\|(\Theta_+^r)_x\|_\infty^2\bigr)dt\nonumber\\
&\quad+\int_0^\infty\!\int
 (U_{xx}^2+\Theta_{xx}^2+V_x^4+\Theta_x^4
               +\widetilde R_1^2+\widetilde R_2^2)dxdt
 \le C\delta^r,\label{profile:integrated-derivatives}\\
&\|U_t+P_x\|_{L^1(0,\infty;L^2)}
 +\|U_t+P_x\|_{L^2((0,\infty)\times\mathbb R)}
 \le C\delta^r,\label{profile:momentum-defect}\\
&\sup_{x,t}|(V_x,\Theta_x,U_x,\Theta_t)|\le C\delta.
\label{profile:first-supremum}
\end{align}
For a single contact wave the rarefaction derivatives are zero.
Furthermore,
\begin{align*}
&\left\|\left(\frac{\kappa(\Theta)\Theta_x}{V}\right)_x
                                             \right\|_\infty
 \le C(1+t)^{-1}.
\end{align*}
The time integrals satisfy
\begin{align*}
&\int_0^\infty
 \bigl(\|\Theta_t\|_\infty^2+\|U_x\Theta_x\|_\infty\bigr)dt\\
&\quad+\int_0^\infty\left[
 \left\|\left(\frac{U_x}{V}\right)_x\right\|_\infty^2
 +\left\|\frac{U_x\Theta_x}{V}\right\|_\infty^2\right]dt\le C.
\end{align*}
For a sufficiently broad contact Gaussian,
\begin{align*}
|V_x|+|\Theta_x|+|P_x|
&\le C\bigl((U_-^r)_x+(U_+^r)_x\bigr)
 +\frac{C\delta}{\sqrt{1+t}}e^{-cx^2/(1+t)},\\
|U_x|+|\Theta_t|
&\le C\bigl((U_-^r)_x+(U_+^r)_x\bigr)
 +\frac{C\delta}{1+t}e^{-cx^2/(1+t)}.
\end{align*}
All constants are uniform for sufficiently small total strength.
\end{lemma}
\begin{proof}
The characteristic formulas and $|w_{xx}(y,0)|\le2w_x(y,0)$ give
\begin{align*}
&\|w_{xx}(t)\|_\infty
 \le2\sup_y\frac{w_x(y,0)}{(1+tw_x(y,0))^3}
 \le C\min\{w_r-w_l,(1+t)^{-1}\}.
\end{align*}
The identity $\|w_{xx}\|_2^2\le\|w_{xx}\|_\infty\|w_{xx}\|_1$, the
bound $\|w_{xx}(t)\|_\infty\le C\min\{w_r-w_l,(1+t)^{-1}\}$, and the
$L^1$ bound for $w_{xx}$ in Lemma~\ref{lem:contact-and-profile-estimates} give
\begin{align*}
&\int_0^\infty\|w_{xx}\|_2^2dt
 \le\int_0^\infty\|w_{xx}\|_\infty\|w_{xx}\|_1dt\\
&\qquad\le C\int_0^\infty\min\{w_r-w_l,(1+t)^{-1}\}^2dt
 \le C(w_r-w_l).
\end{align*}
Since $\int w_x\,dx=w_r-w_l$,
\begin{align*}
&\int_0^\infty\!\int w_x^4dxdt
 \le(w_r-w_l)\int_0^\infty\|w_x\|_\infty^3dt
 \le C(w_r-w_l)^3.
\end{align*}
The derivative relations in \eqref{appro-rare}, the $L^2_{t,x}$ bound for
$(w_\pm)_{xx}$, and the $L^4_{t,x}$ bound for $(w_\pm)_x$ yield
\begin{align*}
&\int_0^\infty\!\int\bigl(|\partial_x^2(V_\pm^r,U_\pm^r,\Theta_\pm^r)|^2
                         +|\partial_x(V_\pm^r,U_\pm^r,\Theta_\pm^r)|^4\bigr)dxdt\\
&\qquad\le C\int_0^\infty\!\int\bigl(|(w_\pm)_{xx}|^2+|(w_\pm)_x|^4\bigr)dxdt
 \le C\delta.
\end{align*}
For the viscous contact wave $(V^{cd},U^{cd},\Theta^{cd})$,
\begin{align*}
&\|\Theta_x^{cd}(t)\|_4^4\le C\delta^4(1+t)^{-3/2},
 \qquad\|\Theta_{xx}^{cd}(t)\|_2^2\le C\delta^2(1+t)^{-3/2}.
\end{align*}
The $L^4$ estimate for $\Theta_x^{cd}$, the $L^2$ estimate for
$\Theta_{xx}^{cd}$, and the interaction estimates in
Lemma~\ref{lem:contact-and-profile-estimates} give
\begin{align*}
&\int_0^\infty\!\int(\widetilde R_1^2+\widetilde R_2^2)dxdt\\
&\qquad\le C\int_0^\infty\!\int
                 (U_{xx}^2+\Theta_{xx}^2+V_x^4+U_x^4+\Theta_x^4)dxdt
                 +C\delta^2\int_0^\infty e^{-ct}dt
 \le C\delta^r.
\end{align*}
The rarefaction equations give
\[
U_t+P_x
 =(U^{cd})_t+
 \left(P-\frac{R\Theta_-^r}{V_-^r}-\frac{R\Theta_+^r}{V_+^r}\right)_x.
\]
Hence
\begin{align*}
\|U_t+P_x\|_2
 &\le C\delta(1+t)^{-5/4}+C\delta e^{-ct}.
\end{align*}
Integrating in time, we obtain
\begin{align*}
&\int_0^\infty\|U_t+P_x\|_2dt
 +\left(\int_0^\infty\|U_t+P_x\|_2^2dt\right)^{1/2}\\
 &\le C\delta\int_0^\infty[(1+t)^{-5/4}+e^{-ct}]dt\\
 &\quad+C\delta\left(\int_0^\infty[(1+t)^{-5/2}+e^{-ct}]dt\right)^{1/2}
 \le C\delta.
\end{align*}
By \eqref{appro-rare},
\begin{align*}
|(V_\pm^r)_x|
 &=\frac{(U_\pm^r)_x}{|\lambda_\pm(V_\pm^r,s_\pm)|}
 \le C(U_\pm^r)_x,
 \qquad |(\Theta_\pm^r)_x|\le C|(V_\pm^r)_x|.
\end{align*}
Combining \eqref{appro-rare}, which bounds $(V_\pm^r)_x$ and
$(\Theta_\pm^r)_x$ by $(U_\pm^r)_x$, with
the Gaussian decay of $(V^{cd},U^{cd},\Theta^{cd})$ and the identity
$P_x=R\Theta_x/V-R\Theta V_x/V^2$ gives the pointwise inequalities in
the statement. Also,
\begin{align*}
\left|\left(\frac{\kappa(\Theta)\Theta_x}{V}\right)_x\right|
 &\le C(|\Theta_{xx}|+\Theta_x^2+|V_x\Theta_x|)
 \le C(1+t)^{-1}.
\end{align*}
Finally, using $(U_x/V)_x=U_{xx}/V-U_xV_x/V^2$, we find
\begin{align*}
&\int_0^\infty\!\left[
 \|\Theta_t\|_\infty^2+\|U_x\Theta_x\|_\infty
 +\left\|\left(\frac{U_x}{V}\right)_x\right\|_\infty^2
 +\left\|\frac{U_x\Theta_x}{V}\right\|_\infty^2\right]dt\\
&\quad\le C\int_0^\infty\left[
 \min\{\delta,(1+t)^{-1}\}^2
 +\delta(1+t)^{-1/2}\min\{\delta,(1+t)^{-1}\}
 +\delta^2(1+t)^{-3/2}\right]dt\\
&\quad\le C\delta+C\delta\int_0^\infty(1+t)^{-3/2}dt
 \le C.
\end{align*}
\end{proof}

\section{Proof of Theorem \ref{theorem1}}
\setcounter{equation}{0}
Subtracting \eqref{contact equation} from \eqref{ns}, we obtain
\begin{equation}\label{perturb}
\left\{
\begin{array}{ll}
\di \phi_t-\psi_x=0,\\
\di \psi_t+\left(p-p_+\right)_x=\tilde{\mu}\left(\frac{\t^\alpha u_x}{v}-\frac{\T^\alpha U_x}{V}\right)_x-\widetilde R_1,\\
\di c_{\nu}\zeta_t+pu_x-p_+U_x=\tilde{\kappa}\left(\frac{\t^{\beta}\t_x}{v}-\frac{\Theta^{\beta}\Theta_x}{V}\right)_x
+\tilde{\mu}\left(\frac{\t^\alpha u_x^2}{v}-\frac{\T^\alpha U^2_x}{V}\right)-\widetilde R_2,\\
\di (\phi,\psi,\zeta)(x,0)=(\phi_0,\psi_0,\zeta_0)(x),\quad x\in \mathbb{R}.
\end{array}
\right.
\end{equation}
The local Cauchy problem has a unique positive strong solution on
$[0,T]$. The perturbation $(\phi,\psi,\zeta)$ is continuous in time
with values in $H^1$, and
\begin{equation}\label{local0:class}
\begin{aligned}
&\sup_{0\le t\le T}\|(\phi,\psi,\zeta)(t)\|_{H^1}^2\\
&\quad+\int_0^T\left(
 \|\phi_t\|_{H^1}^2+\|(\psi,\zeta)\|_{H^2}^2
 +\|(\psi_t,\zeta_t)\|_2^2\right)dt\le C(T).
\end{aligned}
\end{equation}
Local linearization and $H^1$ energy estimates give local existence and
\eqref{local0:class}; cf.\ \cite{huang-ma-shi}. The existence time and the
local bound depend on the initial $H^1$ norm and positive lower bounds,
uniformly for $0\le\alpha\le1$, and continuation is possible as long as
$\| (\phi,\psi,\zeta)(t)\|_{H^1}$ remains finite and $v$, $\theta$ retain
positive lower bounds. For the a priori estimates, assume
\begin{equation}\label{space}
\begin{aligned}
&M^{-1}\le v(x,t),\theta(x,t)\le M,\\
&\sup_{0\le t\le T}\|(\phi,\psi,\zeta)(t)\|_{H^1}^2
+\int_0^T\left(\|\phi_x\|_2^2+\|(\psi_x,\zeta_x)\|_{H^1}^2
+\|(\psi_t,\zeta_t)\|_2^2\right)dt\le M.
\end{aligned}
\end{equation}
Here $M$ is fixed after the estimates. The constants $c,C,C_0$ depend
only on the data in Theorem~\ref{theorem1}; $C(M)$ also depends on $M$.
All space-time integrals below are over $\mathbb R\times(0,t)$,
$0\le t\le T$, unless otherwise indicated. To extend the local solution
globally, it suffices to establish the following a priori estimate.

\begin{proposition}[A priori estimates]\label{prop}
Under the assumptions of Theorem~\ref{theorem1}, there exist positive
constants $\varepsilon_0$ and $\delta_0$, depending only on the prescribed
data, such that, if $0\leq\alpha<\varepsilon_0$ and
$0\leq\delta\leq\delta_0$, then
\begin{equation}\label{p-1}
\begin{aligned}
&C_0^{-1}\le v(x,t),\theta(x,t)\le C_0,\\
&\sup_{0\le t\le T}\|(\phi,\psi,\zeta)(t)\|_{H^1}^2
+\int_0^T\left(\|\phi_x\|_2^2+\|(\psi_x,\zeta_x)\|_{H^1}^2
+\|(\psi_t,\zeta_t)\|_2^2\right)dt\\
&\qquad+\int_0^T\!\int(\phi^2+\psi^2+\zeta^2)w^2\,dxdt\le C_0.
\end{aligned}
\end{equation}
The constant $C_0$ is independent of $M,T$ and higher initial norms.
\end{proposition}

We establish Proposition~\ref{prop} through the following lemmas.

\begin{lemma}\label{basic-lemma}
For sufficiently small $\alpha,\delta$,
\begin{equation}\label{basic}
\begin{aligned}
&\sup_{0\le s\le t}\int\left\{\frac{\psi^2}{2}
+R\Theta\Phi\left(\frac vV\right)
+c_\nu\Theta\Phi\left(\frac\theta\Theta\right)\right\}(x,s)\,dx\\
&\qquad+\int_0^t\!\int\left(
\frac{\mu\Theta}{v\theta}\psi_x^2
+\frac{\kappa\Theta}{v\theta^2}\zeta_x^2\right)dxds\le C_0.
\end{aligned}
\end{equation}
\end{lemma}
\begin{proof}
Multiplying the mass, momentum and temperature equations by
$-R\Theta(v^{-1}-V^{-1})$, $\psi$ and $\zeta/\theta$, respectively, gives
\begin{equation}\label{ba1}
\left(\frac{\psi^2}{2}+R\Theta\Phi\left(\frac vV\right)
+c_\nu\Theta\Phi\left(\frac\theta\Theta\right)\right)_t
+\frac{\mu\Theta}{v\theta}\psi_x^2
+\frac{\kappa\Theta}{v\theta^2}\zeta_x^2+H_x+Q
=-\widetilde R_1\psi-\widetilde R_2\frac\zeta\theta,
\end{equation}
where, as in the perturbation equations,
\begin{equation}\label{H}
H=(p-p_+)\psi-\left(\frac{\mu u_x}{v}-\frac{\mu_1U_x}{V}\right)\psi
-\frac\zeta\theta\left(\frac{\kappa\theta_x}{v}
-\frac{\kappa_1\Theta_x}{V}\right),
\end{equation}
\begin{align*}
Q={}&p_+\Phi\left(\frac Vv\right)U_x
+\frac{p_+}{\gamma-1}\Phi\left(\frac\Theta\theta\right)U_x
+(p-p_+)\frac\zeta\theta U_x
+\left(\frac\mu v-\frac{\mu_1}V\right)\psi_xU_x\\
&-\frac{2\mu}{v\theta}\psi_x\zeta U_x
-\left(\frac\mu v-\frac{\mu_1}V\right)\frac\zeta\theta U_x^2
-\frac{\kappa\zeta}{v\theta^2}\Theta_x\zeta_x
+\frac\Theta{\theta^2}\left(\frac\kappa v-\frac{\kappa_1}V\right)\Theta_x\zeta_x\\
&-\left(\frac\kappa v-\frac{\kappa_1}V\right)
                        \frac{\zeta\Theta_x^2}{\theta^2}.
\end{align*}
By \eqref{space},
\[
\left|\frac\mu v-\frac{\mu_1}V\right|
+\left|\frac\kappa v-\frac{\kappa_1}V\right|
\le C(M)(|\phi|+|\zeta|).
\]
Hence
\[
|Q|\le\frac{\mu\Theta}{4v\theta}\psi_x^2
+\frac{\kappa\Theta}{4v\theta^2}\zeta_x^2
+C(M)(\phi^2+\zeta^2)(|U_x|+\Theta_x^2).
\]
Using Lemmas~\ref{decay} and \ref{lemma5}, we have
\begin{equation}\label{Q}
\int_0^t\!\int(\phi^2+\zeta^2)(|U_x|+\Theta_x^2)\,dxds
\le C\delta\int_0^t\!\int(\phi^2+\zeta^2)w^2\,dxds
\le C(M)\delta.
\end{equation}
Also,
\begin{align}
\left|\int_0^t\!\int\widetilde R_1\psi\,dxds\right|
&\le\sup_{0\le s\le t}\|\psi(s)\|_2
\int_0^t\|\widetilde R_1(s)\|_2\,ds
\le C(M)\delta\int_0^\infty(1+s)^{-5/4}ds
\le C(M)\delta,\label{R1phi}\\
\left|\int_0^t\!\int\widetilde R_2\frac\zeta\theta\,dxds\right|
&\le M\sup_{0\le s\le t}\|\zeta(s)\|_2
\int_0^t\|\widetilde R_2(s)\|_2\,ds
\le C(M)\delta^2\int_0^\infty(1+s)^{-7/4}ds
\le C(M)\delta^2.\label{R2zeta}
\end{align}
Integration of \eqref{ba1} gives
\begin{equation}\label{ba2}
\begin{aligned}
&\sup_{0\le s\le t}\int\left(\frac{\psi^2}{2}
+R\Theta\Phi\left(\frac vV\right)
+c_\nu\Theta\Phi\left(\frac\theta\Theta\right)\right)(x,s)dx\\
&\qquad+\frac34\int_0^t\!\int
\left(\frac{\mu\Theta}{v\theta}\psi_x^2
+\frac{\kappa\Theta}{v\theta^2}\zeta_x^2\right)dxds
\le C+C(M)\delta.
\end{aligned}
\end{equation}
Taking $C(M)\delta\le1$ proves \eqref{basic}.

Put $\tilde v=v/V$. Since
\[
\left(\frac{\tilde v_x}{\tilde v}\right)_t
 =\left(\frac{u_x}{v}-\frac{U_x}{V}\right)_x,
\]
the momentum equation gives
\begin{equation}\label{st2}
\mu\left(\frac{\tilde v_x}{\tilde v}\right)_t-\psi_t-(p-p_+)_x
 =-\mu_x\left(\frac{u_x}{v}-\frac{U_x}{V}\right)
 +\widetilde R_1-\left(\frac{(\mu-\mu_1)U_x}{V}\right)_x.
\end{equation}
Using
\[
(p-p_+)_x=-p\frac{\tilde v_x}{\tilde v}
             +\frac{R\zeta_x}{v}-\frac{R\zeta\Theta_x}{v\Theta},
\]
we multiply \eqref{st2} by $\tilde v_x/(\mu\tilde v)$ to obtain
\begin{equation}\label{phi-x1}
\begin{aligned}
&\partial_t\left\{\frac12\bigl[\frac{\tilde v_x}{\tilde v}\bigr]^2
 -\frac{\psi\frac{\tilde v_x}{\tilde v}}{\mu}\right\}
 +\partial_x\left\{\frac{\psi}{\mu}
 \left(\frac{u_x}{v}-\frac{U_x}{V}\right)\right\}
 +\frac{p}{\mu}\bigl[\frac{\tilde v_x}{\tilde v}\bigr]^2\\
={}&\left(\frac{\psi}{\mu}\right)_x
 \left(\frac{u_x}{v}-\frac{U_x}{V}\right)
 +\frac{\psi\mu_t}{\mu^2}\frac{\tilde v_x}{\tilde v}
 +\frac{R\zeta_x}{\mu v}\frac{\tilde v_x}{\tilde v}\\
&-\frac{R\zeta\Theta_x}{\mu v\Theta}\frac{\tilde v_x}{\tilde v}
 -\frac{\mu_x}{\mu}
 \left(\frac{u_x}{v}-\frac{U_x}{V}\right)\frac{\tilde v_x}{\tilde v}\\
&+\frac{\frac{\tilde v_x}{\tilde v}}{\mu}
 \left\{\widetilde R_1-
 \left[\frac{(\mu-\mu_1)U_x}{V}\right]_x\right\}.
\end{aligned}
\end{equation}

For $\alpha\log M\le\log2$, we have $\tilde\mu/2\le\mu\le2\tilde\mu$.
In view of \eqref{space} and
\[
\frac{\tilde v_x}{\tilde v}=\frac{\phi_x}{v}-\frac{\phi V_x}{vV},
\]
the identities
\[
\frac{\mu_t}{\mu}=\alpha\frac{\zeta_t+\Theta_t}{\theta},\quad
\frac{\mu_x}{\mu}=\alpha\frac{\zeta_x+\Theta_x}{\theta}
\]
give
\begin{align*}
&\left|\iint\frac{\psi\mu_t}{\mu^2}
                                      \frac{\tilde v_x}{\tilde v}\right|\\
&\quad\le\alpha C(M)\iint|\psi|(|\zeta_t|+|\Theta_t|)
                                      |\frac{\tilde v_x}{\tilde v}|\\
&\quad\le\alpha C(M)\left(
 \sup_{0\le s\le t}\|\psi(s)\|_\infty^2\iint\zeta_t^2
 +\sup_{0\le s\le t}\|\psi(s)\|_2^2\int_0^t\|\Theta_t\|_\infty^2ds
 +\iint\left|\frac{\tilde v_x}{\tilde v}\right|^2\right)
 \le\alpha C(M).
\end{align*}
Similarly, the spatial derivative of $\mu$ satisfies
\begin{align*}
&\left|\iint\frac{\mu_x}{\mu}
 \left(\frac{u_x}{v}-\frac{U_x}{V}\right)\frac{\tilde v_x}{\tilde v}\right|\\
&\quad\le\alpha C(M)\iint(|\zeta_x|+|\Theta_x|)
                 (|\psi_x|+|\phi U_x|)|\frac{\tilde v_x}{\tilde v}|\\
&\quad\le\alpha C(M)\left[
 \sup_{0\le s\le t}\|\zeta_x(s)\|_2
 \left(\int_0^t\|\psi_x\|_\infty^2ds\right)^{1/2}
 \left(\iint\left|\frac{\tilde v_x}{\tilde v}\right|^2\right)^{1/2}
 +C(M)\right]\\
&\quad\le\alpha C(M)\left[
 \sup_{0\le s\le t}\|\zeta_x(s)\|_2
 \left(2\int_0^t\|\psi_x\|_2\|\psi_{xx}\|_2ds\right)^{1/2}
 \left(\iint\left|\frac{\tilde v_x}{\tilde v}\right|^2\right)^{1/2}
 +C(M)\right]
 \le\alpha C(M).
\end{align*}
The other products with $\theta_x$ are estimated by
\begin{align*}
\iint|\psi\zeta_x\psi_x|
&\le\sup_{0\le s\le t}\|\psi(s)\|_\infty
 \left(\iint\zeta_x^2\right)^{1/2}\left(\iint\psi_x^2\right)^{1/2}
 \le C(M),\\
\iint|\psi\Theta_x\psi_x|
&\le\left(\iint\psi^2\Theta_x^2\right)^{1/2}
 \left(\iint\psi_x^2\right)^{1/2}\\
&\le C\delta\left(\iint w^2\psi^2\right)^{1/2}
 \left(\iint\psi_x^2\right)^{1/2}
 \le C(M).
\end{align*}
Moreover,
\begin{align*}
&\left[\frac{(\mu-\mu_1)U_x}{V}\right]_x\\
&\quad=\alpha\mu\frac{\zeta_x}{\theta}\frac{U_x}{V}
 +\alpha\left(\frac{\mu}\theta-
              \frac{\mu_1}\Theta\right)\Theta_x\frac{U_x}{V}
 +(\mu-\mu_1)\left(\frac{U_x}{V}\right)_x.
\end{align*}
The bounds for $U_x$, $(U_x/V)_x$, and $U_x\Theta_x$ in
Lemma~\ref{lem:remaining-profile-integrals}, together with \eqref{space}, give
\begin{align*}
&\left|\iint\frac{\frac{\tilde v_x}{\tilde v}}{\mu}
 \left[\frac{(\mu-\mu_1)U_x}{V}\right]_x\right|\\
&\quad\le\alpha C(M)\iint|\frac{\tilde v_x}{\tilde v}|
 \left\{|U_x\zeta_x|
 +|\zeta|\left[\left|\left(\frac{U_x}{V}\right)_x\right|
                      +|U_x\Theta_x|\right]\right\}\\
&\quad\le\alpha C(M)\left(\iint\left|\frac{\tilde v_x}{\tilde v}\right|^2\right)^{1/2}
 \left[\sup_{0\le s\le t}\|U_x(s)\|_\infty
                       \left(\iint\zeta_x^2\right)^{1/2}\right.\\
&\qquad\left.
 +\sup_{0\le s\le t}\|\zeta(s)\|_\infty
       \left(\iint\left\{\left|\left(\frac{U_x}{V}\right)_x\right|^2
                      +U_x^2\Theta_x^2\right\}\right)^{1/2}\right]
 \le\alpha C(M).
\end{align*}
For the remaining terms in \eqref{phi-x1},
\begin{align*}
&\iint\left(\frac\psi{\mu}\right)_x
                        \left(\frac{u_x}{v}-\frac{U_x}{V}\right)\\
&\quad=\iint\frac{\psi_x^2}{\mu v}
 -\iint\frac{\phi\psi_xU_x}{\mu vV}
 -\alpha\iint\frac{\psi\theta_x}{\mu\theta}
                          \left(\frac{\psi_x}{v}-\frac{\phi U_x}{vV}\right)\\
&\quad\le C\iint\frac{\psi_x^2}{v}
          +C(M)\iint\phi^2U_x^2+\alpha C(M).
\end{align*}
Young's inequality gives
\begin{align*}
&\left|\iint\frac{R\zeta_x}{\mu v}\frac{\tilde v_x}{\tilde v}\right|
 \le\frac18\iint\frac p{\mu}[\frac{\tilde v_x}{\tilde v}]^2
       +C\iint\frac{\zeta_x^2}{v\theta},\\
&\left|\iint\left[-\frac{R\zeta\Theta_x}{\mu v\Theta}
 +\frac{\widetilde R_1}{\mu}\right]\frac{\tilde v_x}{\tilde v}\right|\\
&\quad\le\frac18\iint\frac p{\mu}[\frac{\tilde v_x}{\tilde v}]^2
 +C(M)\iint\left[(\phi^2+\zeta^2)(\Theta_x^2)
                                       +\widetilde R_1^2\right].
\end{align*}
Since
\[
\int\left\{\frac12[\frac{\tilde v_x}{\tilde v}]^2
                 -\frac{\psi\frac{\tilde v_x}{\tilde v}}{\mu}\right\}
 \ge\frac14\|\frac{\tilde v_x}{\tilde v}\|_2^2-C\|\psi\|_2^2,
\]
integration of \eqref{phi-x1} gives
\begin{equation}\label{phi-x3}
\begin{aligned}
&\sup_{0\leq t\leq T}\|\frac{\tilde v_x}{\tilde v}(t)\|_2^2
 +\int_0^T\!\int\frac{\theta}{v}[\frac{\tilde v_x}{\tilde v}]^2\\
&\leq C\left\{\|(\log(v_0/V_0))_x\|_2^2
 +\sup_{0\leq t\leq T}\|\psi(t)\|_2^2
 +\int_0^T\!\int\left(\frac{\psi_x^2}{v}+
 \frac{\zeta_x^2}{v\theta}\right)\right\}\\
&\quad+C(M)\int_0^T\!\int
 \left[(\phi^2+\zeta^2)(U_x^2+\Theta_x^2)
 +\widetilde R_1^2\right]+\alpha C(M).
\end{aligned}
\end{equation}

\end{proof}
\begin{lemma}\label{huang-lemma}
Let $\tilde\theta=\theta/\Theta$ and let $\alpha_1,\alpha_2$ be the two
positive roots of $y-\log y-1=C$, with $C$ fixed by \eqref{basic}.
For $k\in\mathbb Z$,
\begin{equation}\label{root}
\alpha_1\le\int_k^{k+1}\tilde v(x,t)\,dx\le\alpha_2,
\qquad
\alpha_1\le\int_k^{k+1}\tilde\theta(x,t)\,dx\le\alpha_2.
\end{equation}
There are $a_k(t),b_k(t)\in[k,k+1]$ such that
\begin{equation}\label{root2}
\alpha_1\le\tilde v(a_k(t),t),\tilde\theta(b_k(t),t)\le\alpha_2.
\end{equation}
\end{lemma}
\begin{proof}
By Jensen's inequality,
\begin{align*}
\Phi\left(\int_k^{k+1}\tilde v\,dx\right)
+\Phi\left(\int_k^{k+1}\tilde\theta\,dx\right)
&\le\int_k^{k+1}\bigl[\Phi(\tilde v)+\Phi(\tilde\theta)\bigr]dx
\le C.
\end{align*}
The mean value theorem gives \eqref{root2}.
\end{proof}

\begin{lemma}\label{jiang-lemma}
For $x\in[k,k+1]$ and $0\le s\le t\le T$,
\begin{equation}\label{v-repre}
v(x,t)=v(x,s)\frac{B(x,t)Y(t)}{B(x,s)Y(s)}
+\int_s^t\frac{R\theta(x,r)}{\mu(\theta(x,r))}
              \frac{B(x,t)Y(t)}{B(x,r)Y(r)}\,dr,
\end{equation}
where
\begin{equation}\label{B}
\begin{aligned}
B(x,t)=v_0(x)\exp\Bigg\{
&\int_x^\infty\left(\frac{u_0(y)}{\mu(\theta_0(y))}
                         -\frac{u(y,t)}{\mu(\theta(y,t))}\right)D(y)\,dy\\
&-\alpha\int_0^t\!\int_x^\infty
\left(\frac{u\theta_t}{\mu\theta}+\frac{R\theta_y}{\mu v}
                        -\frac{\theta_yu_y}{\theta v}\right)D(y)\,dydr
\Bigg\},
\end{aligned}
\end{equation}
\begin{equation}\label{Y}
Y(t)=\exp\left(\int_0^t\!\int_{k+1}^{k+2}\frac\sigma\mu(y,r)\,dydr\right),
\qquad \sigma=\frac{\mu u_x}{v}-\frac{R\theta}{v},
\end{equation}
\begin{equation}\label{beta}
D(x)=\begin{cases}
1,&x\le k+1,\\
k+2-x,&k+1\le x\le k+2,\\
0,&x\ge k+2.
\end{cases}
\end{equation}
\end{lemma}
\begin{proof}
The momentum equation gives
\begin{equation}\label{volume0:divided-momentum}
\left(\frac\sigma\mu\right)_x
=\left(\frac u\mu\right)_t
+\alpha\left(\frac{u\theta_t}{\mu\theta}
+\frac{R\theta_x}{\mu v}-\frac{\theta_xu_x}{\theta v}\right).
\end{equation}
Multiplying by $D$ and integrating from $x$ to $+\infty$, we obtain
\begin{align}
\frac\sigma\mu(x,t)
&=\int_{k+1}^{k+2}\frac\sigma\mu(y,t)\,dy
-\frac d{dt}\int_x^\infty\frac{u(y,t)}{\mu(\theta(y,t))}D(y)\,dy\nonumber\\
&\quad-\alpha\int_x^\infty
 \left(\frac{u\theta_t}{\mu\theta}
+\frac{R\theta_y}{\mu v}-\frac{\theta_yu_y}{\theta v}\right)D(y)\,dy.
\label{volume0:kernel-formula}
\end{align}
Integrating in time and using \eqref{B} and \eqref{Y} gives
\begin{equation}\label{volume0:kernel-ode}
\frac{B(x,t)Y(t)}{B(x,s)Y(s)}
=\exp\left(\int_s^t\frac\sigma\mu(x,r)\,dr\right).
\end{equation}
Since $v_t=u_x$,
\begin{align*}
\frac d{dt}\left(\frac{v(x,t)}{B(x,t)Y(t)}\right)
&=\frac{v_t-v\sigma/\mu}{B(x,t)Y(t)}
=\frac{R\theta(x,t)}{\mu(\theta(x,t))B(x,t)Y(t)}.
\end{align*}
Integration in time proves \eqref{v-repre}.
\end{proof}

\begin{lemma}\label{v-bound}
There are positive constants $\underline v(C_0),\bar v(C_0)$ such that
\begin{equation}\label{v}
\underline v(C_0)\le v(x,t)\le\bar v(C_0),\qquad
x\in\mathbb R,\quad 0\le t\le T.
\end{equation}
\end{lemma}
\begin{proof}
Choose $\alpha\log M\le\log2$. Then
$\tilde\mu/2\le\mu\le2\tilde\mu$. By \eqref{basic},
\begin{align*}
&\left|\int_x^\infty\left(\frac{u(y,t)}{\mu(\theta(y,t))}
                      -\frac{u(y,s)}{\mu(\theta(y,s))}\right)D(y)\,dy\right|\\
&\qquad\le C\bigl(\|\psi(t)\|_2+\|\psi(s)\|_2+\|U\|_\infty\bigr)\le C.
\end{align*}
The time integral in \eqref{B} satisfies
\begin{align*}
&\int_s^t\!\int_k^{k+2}|u\theta_t|\,dydr
\le C(M)\sqrt{t-s}\left[\int_s^t
   \bigl(\|\zeta_t\|_2^2+\|\Theta_t\|_\infty^2\bigr)dr\right]^{1/2}
\le C(M)\sqrt{t-s},\\
&\int_s^t\!\int_k^{k+2}|\theta_x|\,dydr
\le C\sqrt{t-s}\left(\int_s^t\|\zeta_x\|_2^2dr\right)^{1/2}
+C\delta\int_s^t(1+r)^{-1/2}dr
\le C(M)\sqrt{t-s}.
\end{align*}
For the product of the spatial derivatives, we have
\begin{align*}
&\int_s^t\!\int_k^{k+2}|\theta_xu_x|\,dydr\\
&\qquad\le C\left(\int_s^t\|\zeta_x\|_2^2dr\right)^{1/2}
       \left[\int_s^t\bigl(\|\psi_x\|_2^2+C\|U_x\|_\infty^2\bigr)dr\right]^{1/2}\\
&\qquad\quad+C\delta\left(\int_s^t\|\psi_x\|_2^2dr\right)^{1/2}
                    \left(\int_s^t(1+r)^{-1}dr\right)^{1/2}
             +C\int_s^t\|\Theta_xU_x\|_\infty dr\\
&\qquad\le C(M)\bigl(1+\sqrt{\log(1+t-s)}\bigr).
\end{align*}
Consequently,
\begin{equation}\label{b-bound}
\left|\log\frac{B(x,t)}{B(x,s)}\right|
\le C+\alpha C(M)(1+\sqrt{t-s}).
\end{equation}
By Lemma~\ref{huang-lemma} and Jensen's inequality,
\begin{equation}\label{int-t/v}
\begin{aligned}
-\int_s^t\!\int_{k+1}^{k+2}\frac\theta v\,dxd\tau
&\le-c\int_s^t\exp\left(\int_{k+1}^{k+2}
                 \log\frac{\tilde\theta}{\tilde v}\,dx\right)d\tau\\
&\le-c\int_s^t\frac{
 \exp\left(\int_{k+1}^{k+2}
       [\tilde\theta-1-\Phi(\tilde\theta)]\,dx\right)}
 {\displaystyle\int_{k+1}^{k+2}\tilde v\,dx}\,d\tau
\le-c(t-s).
\end{aligned}
\end{equation}
Since $\Theta(x,t)=\Theta(x/\sqrt{1+t},0)$ is monotone and
\[
U_x(x,t)=\frac R{p_+}\Theta_t(x,t)
        =-\frac{Rx}{2p_+(1+t)}\Theta_x(x,t),
\]
we have
\begin{align*}
\int_s^t|U_x(x,\tau)|\,d\tau
&=\frac R{p_+}|\Theta(x,t)-\Theta(x,s)|\le C\delta.
\end{align*}
Integration of \eqref{Y} and \eqref{int-t/v} gives
\begin{equation}\label{int-sigma}
\begin{aligned}
\int_s^t\!\int_{k+1}^{k+2}\frac\sigma\mu\,dxd\tau
&=\int_s^t\!\int_{k+1}^{k+2}
 \left(\frac{\psi_x}{v}+\frac{U_x}{v}
                         -\frac{R\theta}{\mu v}\right)dxd\tau\\
&\le C\int_s^t\!\int\frac{\psi_x^2}{v\theta}\,dxd\tau
 -c\int_s^t\!\int_{k+1}^{k+2}\frac\theta v\,dxd\tau
 +M\int_{k+1}^{k+2}\!\int_s^t|U_x|\,d\tau dx\\
&\le C+C(M)\delta-c(t-s).
\end{aligned}
\end{equation}
For $C(M)\delta\le1$ and $\alpha C(M)\le1$,
\begin{align}
\log\frac{B(x,t)Y(t)}{B(x,s)Y(s)}
&\le C+\alpha C(M)(1+\sqrt{t-s})
+C(M)\delta-c(t-s)
\nonumber\\
&\le C+\frac{[\alpha C(M)]^2}{c}-\frac c2(t-s).\nonumber
\end{align}
Combining \eqref{b-bound} with \eqref{int-sigma} yields
\begin{equation}\label{Y-e}
\frac{B(x,t)Y(t)}{B(x,s)Y(s)}
+\frac{Y(t)}{Y(s)}\le Ce^{-c(t-s)}.
\end{equation}
By \eqref{basic}, \eqref{root2}, and the monotonicity of the powers,
\begin{align*}
&\left|\tilde\theta^{1/2}(x,t)-\tilde\theta^{1/2}(b_k(t),t)\right|\\
&\quad\le C\left|\tilde\theta^{(\beta+1)/2}(x,t)
                    -\tilde\theta^{(\beta+1)/2}(b_k(t),t)\right|\\
&\quad\le C\int_k^{k+1}\tilde\theta^{(\beta-1)/2}
       \left|\frac{\zeta_x}{\Theta}-\frac{\zeta\Theta_x}{\Theta^2}\right|dx\\
&\quad\le C\left(\int_k^{k+1}\frac{\theta^{\beta-2}\zeta_x^2}{v}\,dx\right)^{1/2}
             \left(\int_k^{k+1}\theta v\,dx\right)^{1/2}
       +C(M)\left(\int\zeta^2\Theta_x^2\,dx\right)^{1/2}\\
&\quad\le C\left[\int\frac{\theta^{\beta-2}\zeta_x^2}{v}(y,t)\,dy
                         \sup_yv(y,t)\right]^{1/2}
       +C(M)\left(\int\zeta^2\Theta_x^2\,dx\right)^{1/2}.
\end{align*}
Thus
\begin{equation}\label{v-theta}
\begin{aligned}
\theta(x,t)&\le C+C\int\frac{\theta^{\beta-2}\zeta_x^2}{v}(y,t)\,dy\sup_yv(y,t)
                         +C(M)\int\zeta^2\Theta_x^2\,dx,\\
\theta(x,t)&\ge c-C\int\frac{\theta^{\beta-2}\zeta_x^2}{v}(y,t)\,dy\sup_yv(y,t)
                         -C(M)\int\zeta^2\Theta_x^2\,dx.
\end{aligned}
\end{equation}
Since $C(M)\iint\zeta^2\Theta_x^2\le\delta^2C(M)\le1$, we have
\begin{align}
\sup_xv(x,t)
&\le C+C\int_0^t e^{-c(t-s)}\sup_x\theta(x,s)\,ds\nonumber\\
&\le C+C\int_0^t\left(\int\frac{\theta^{\beta-2}\zeta_x^2}{v}(x,s)\,dx\right)
                                          \sup_xv(x,s)\,ds\nonumber\\
&\le C\exp\left(C\int_0^t\!\int\frac{\theta^{\beta-2}\zeta_x^2}{v}\,dxds\right)
\le C.\label{v-upper}
\end{align}
For $0\le t-s\le L$, \eqref{b-bound} gives
\[
c_L\le\frac{B(x,t)}{B(x,s)}\le C_L.
\]
Dividing \eqref{v-repre} by $Y(t)/Y(s)$ and integrating on $[k,k+1]$,
\begin{align*}
\frac{Y(s)}{Y(t)}\int_k^{k+1}v(x,t)\,dx
&\le C_L\int_k^{k+1}v(x,s)\,dx
+C_L\int_s^t\frac{Y(s)}{Y(r)}\int_k^{k+1}\theta(x,r)\,dxdr.
\end{align*}
By \eqref{root} and Gronwall's inequality,
\begin{align*}
\frac{Y(s)}{Y(t)}
&\le C_L+C_L\int_s^t\frac{Y(s)}{Y(r)}\,dr
\le C_Le^{C_L(t-s)}\le C_Le^{C_LL}.
\end{align*}
Consequently,
\begin{equation}\label{volume0:kernel-lower}
\frac{B(x,t)Y(t)}{B(x,s)Y(s)}\ge c_L>0,
\qquad 0\le t-s\le L.
\end{equation}
By \eqref{basic} and \eqref{v-upper}, choose $L\ge1$, depending only
on the data, such that
\[
C\int_0^T\!\int\frac{\theta^{\beta-2}\zeta_x^2}{v}\,dxdt+1\le\frac{cL}{2}.
\]
For $t\ge L$, \eqref{v-repre}, \eqref{v-theta} and
$C(M)\iint\zeta^2\Theta_x^2\le1$ yield
\begin{align}
v(x,t)
&\ge c_L\int_{t-L}^t\theta(x,s)\,ds\nonumber\\
&\ge c_L\left[cL-C\int_{t-L}^t\!\int\frac{\theta^{\beta-2}\zeta_x^2}{v}\,dxds
-C(M)\int_{t-L}^t\!\int\zeta^2\Theta_x^2\,dxds\right]
\ge\frac{c_LcL}{2}.\label{wt-1}
\end{align}
For $0\le t\le L$, we obtain directly from \eqref{v-repre} that
\begin{equation}\label{v-ka}
v(x,t)\ge v_0(x)\frac{B(x,t)Y(t)}{B(x,0)Y(0)}
\ge c_L\underline v_0.
\end{equation}
The upper bound \eqref{v-upper} and the lower bounds \eqref{wt-1} and
\eqref{v-ka} establish \eqref{v}.
\end{proof}

\begin{lemma}\label{lemma-ll-1}
For $0\leq\beta<1$, one has
\begin{equation}\label{new-nergy}
\begin{aligned}
&\sup_{0\leq s\leq T}\int(\zeta^2+\psi^4)(x,s)\,dx
 +\int_0^T\!\int\bigl[(1+\theta+\psi^2)\psi_x^2
              +(\theta^{-1}+\theta^\beta)\zeta_x^2\bigr]dxds\\
&\qquad+\int_0^T\left\|\left(\zeta-\frac12\Theta\right)_+\right\|_\infty^2ds
 \leq C+\delta^r C(M).
\end{aligned}
\end{equation}
\end{lemma}
\begin{proof}
For $a>1$ and every $\beta\geq0$, let
\begin{equation}\label{omega_a}
\Omega_a(t)=\{x\in\mathbb R:\theta(x,t)>a\Theta(x,t)\}.
\end{equation}
By \eqref{basic} and \eqref{v},
\begin{equation}\label{omega-bound}
\sup_{0\leq t\leq T}\int_{\Omega_a(t)}\theta\,dx
 \leq C(a)\sup_{0\leq t\leq T}\int\Phi(\theta/\Theta)\,dx\leq C(a).
\end{equation}
The same entropy bound gives
\begin{equation}\label{cedu}
\sup_{0\leq t\leq T}\left(
 |\Omega_a(t)|+|\{\theta<a^{-1}\Theta\}(t)|
 +\int_{\{\theta\leq a\Theta\}}\zeta^2dx\right)\leq C(a).
\end{equation}
Multiplying \eqref{perturb}$_3$ by $(\zeta-\Theta)_+$ gives
\begin{align}
&\frac{c_\nu}{2}\int(\zeta-\Theta)_+^2dx
 +\tilde\kappa\int_0^t\!\int_{\Omega_2}\frac{\theta^\beta\zeta_x^2}{v}\,dxds
 =\frac{c_\nu}{2}\int(\zeta_0-\Theta(\cdot,0))_+^2dx\notag\\
&\quad-\int_0^t\!\int\frac{R\zeta+R\Theta}{v}\psi_x(\zeta-\Theta)_+dxds
 -\int_0^t\!\int\frac{R\zeta-p_+\phi}{v}U_x(\zeta-\Theta)_+dxds\notag\\
&\quad+\tilde\kappa\int_0^t\!\int_{\Omega_2}
       \frac{\Theta^\beta\zeta_x\Theta_x}{V}\,dxds
 -\tilde\kappa\int_0^t\!\int_{\Omega_2}
       \frac{(\Theta^\beta v-\theta^\beta V)\Theta_x^2}{vV}\,dxds\notag\\
&\quad+\int_0^t\!\int\frac{\mu }v\psi_x^2(\zeta-\Theta)_+dxds
 +2\int_0^t\!\int\frac{\mu }v\psi_xU_x(\zeta-\Theta)_+dxds\notag\\
&\quad+\int_0^t\!\int
 \left(\frac{\mu }v-\frac{\mu_1 }V\right)U_x^2(\zeta-\Theta)_+dxds
 -\int_0^t\!\int\widetilde R_2(\zeta-\Theta)_+dxds
 -c_\nu\int_0^t\!\int\Theta_t(\zeta-\Theta)_+dxds.
 \label{5.1}
\end{align}
Multiplying \eqref{perturb}$_2$ by $2\psi(\zeta-\Theta)_+$, we obtain
\begin{align}
&\int\psi^2(\zeta-\Theta)_+dx
 +2\int_0^t\!\int\frac{\mu }v\psi_x^2(\zeta-\Theta)_+dxds
 =\int\psi_0^2(\zeta_0-\Theta(\cdot,0))_+dx\notag\\
&\quad+2\int_0^t\!\int\frac{R\zeta-p_+\phi}{v}\psi_x(\zeta-\Theta)_+dxds
 +2\int_0^t\!\int_{\Omega_2}\frac{R\zeta-p_+\phi}{v}\psi\zeta_xdxds\notag\\
&\quad-2\int_0^t\!\int_{\Omega_2}\frac{R\zeta-p_+\phi}{v}\psi\Theta_xdxds
 -2\int_0^t\!\int
 \left(\frac{\mu }v-\frac{\mu_1 }V\right)U_x\psi_x(\zeta-\Theta)_+dxds\notag\\
&\quad-2\int_0^t\!\int_{\Omega_2}\frac{\mu }v\psi\psi_x\zeta_xdxds
 -2\int_0^t\!\int_{\Omega_2}
 \left(\frac{\mu }v-\frac{\mu_1 }V\right)\psi U_x\zeta_xdxds\notag\\
&\quad+2\int_0^t\!\int_{\Omega_2}\frac{\mu }v\psi\psi_x\Theta_xdxds
 +2\int_0^t\!\int_{\Omega_2}
 \left(\frac{\mu }v-\frac{\mu_1 }V\right)\psi U_x\Theta_xdxds\notag\\
&\quad-2\int_0^t\!\int\psi\widetilde R_1(\zeta-\Theta)_+dxds
 +\int_0^t\!\int_{\Omega_2}\psi^2\zeta_tdxds
 -\int_0^t\!\int_{\Omega_2}\psi^2\Theta_tdxds.
 \label{5.2}
\end{align}
Adding \eqref{5.1} and \eqref{5.2}, and substituting \eqref{perturb}$_3$, gives
\begin{align}
&\int\left\{\frac{c_\nu}{2}(\zeta-\Theta)_+^2+\psi^2(\zeta-\Theta)_+\right\}dx
 +\int_0^t\!\int\frac{\mu }v\psi_x^2(\zeta-\Theta)_+dxds
 +\tilde\kappa\int_0^t\!\int_{\Omega_2}\frac{\theta^\beta\zeta_x^2}{v}\,dxds\notag\\
&=\int\left\{\frac{c_\nu}{2}(\zeta_0-\Theta(\cdot,0))_+^2
                    +\psi_0^2(\zeta_0-\Theta(\cdot,0))_+\right\}dx\notag\\
&\quad+\int_0^t\!\int\frac{R\zeta-2p_+\phi-R\Theta}{v}\psi_x(\zeta-\Theta)_+dxds
 -\int_0^t\!\int\frac{R\zeta-p_+\phi}{v}U_x(\zeta-\Theta)_+dxds\notag\\
&\quad+\tilde\kappa\int_0^t\!\int_{\Omega_2}\frac{\Theta^\beta\zeta_x\Theta_x}{V}\,dxds
 -\tilde\kappa\int_0^t\!\int_{\Omega_2}
                \frac{(\Theta^\beta v-\theta^\beta V)\Theta_x^2}{vV}\,dxds\notag\\
&\quad+2\int_0^t\!\int\frac{\mu_1 }V\psi_xU_x(\zeta-\Theta)_+dxds
 +\int_0^t\!\int
 \left(\frac{\mu }v-\frac{\mu_1 }V\right)U_x^2(\zeta-\Theta)_+dxds\notag\\
&\quad+2\int_0^t\!\int_{\Omega_2}\frac{R\zeta-p_+\phi}{v}\psi\zeta_xdxds
 -2\int_0^t\!\int_{\Omega_2}\frac{R\zeta-p_+\phi}{v}\psi\Theta_xdxds\notag\\
&\quad-2\int_0^t\!\int_{\Omega_2}\frac{\mu }v\psi\psi_x\zeta_xdxds
 -2\int_0^t\!\int_{\Omega_2}
 \left(\frac{\mu }v-\frac{\mu_1 }V\right)\psi U_x\zeta_xdxds\notag\\
&\quad+2\int_0^t\!\int_{\Omega_2}\frac{\mu }v\psi\psi_x\Theta_xdxds
 +2\int_0^t\!\int_{\Omega_2}
 \left(\frac{\mu }v-\frac{\mu_1 }V\right)\psi U_x\Theta_xdxds\notag\\
&\quad-2\int_0^t\!\int\psi\widetilde R_1(\zeta-\Theta)_+dxds
 -\int_0^t\!\int\widetilde R_2(\zeta-\Theta)_+dxds
 -c_\nu\int_0^t\!\int\Theta_t(\zeta-\Theta)_+dxds
 -\int_0^t\!\int_{\Omega_2}\psi^2\Theta_tdxds\notag\\
&\quad+\frac1{c_\nu}\int_0^t\!\int_{\Omega_2}\psi^2
 \left\{\frac{\mu }v(\psi_x^2+2\psi_xU_x)
       +\left(\frac{\mu }v-\frac{\mu_1 }V\right)U_x^2\right\}dxds
 -\frac1{c_\nu}\int_0^t\!\int_{\Omega_2}\psi^2\widetilde R_2dxds\notag\\
&\quad-\frac1{c_\nu}\int_0^t\!\int_{\Omega_2}\psi^2
       \left(\frac{R\theta}{v}\psi_x+\frac{R\zeta-p_+\phi}{v}U_x\right)dxds
 +\frac{\tilde\kappa}{c_\nu}\int_0^t\!\int_{\Omega_2}\psi^2
       \left(\frac{\theta^\beta\theta_x}v-\frac{\Theta^\beta\Theta_x}V\right)_xdxds\notag\\
&\equiv\int\left\{\frac{c_\nu}{2}(\zeta_0-\Theta(\cdot,0))_+^2
                    +\psi_0^2(\zeta_0-\Theta(\cdot,0))_+\right\}dx
       +\sum_{i=1}^{20}I_i.
 \label{long}
\end{align}
Here $c\leq\mu \leq C$ follows from $\alpha\log M\leq1$.
By \eqref{basic}, \eqref{v} and \eqref{omega-bound},
\begin{align}
|I_1|&\leq\frac14\int_0^t\!\int\frac{\mu }v\psi_x^2(\zeta-\Theta)_+dxds
 +C\int_0^t\!\int_{\Omega_2}(\zeta^2+\phi^2+1)(\zeta-\Theta)_+dxds\notag\\
&\leq\frac14\int_0^t\!\int\frac{\mu }v\psi_x^2(\zeta-\Theta)_+dxds
 +C\int_0^t\left\|\left(\zeta-\frac12\Theta\right)_+\right\|_\infty^2
                      \int_{\Omega_{3/2}}(\theta+\phi^2)dxds\notag\\
&\leq\frac14\int_0^t\!\int\frac{\mu }v\psi_x^2(\zeta-\Theta)_+dxds
 +C\int_0^t\left\|\left(\zeta-\frac12\Theta\right)_+\right\|_\infty^2ds.
 \label{i1}
\end{align}
The Gaussian decay $|U_x|+\Theta_x^2\le C\delta w^2$ from
Lemma~\ref{decay} and the Gaussian estimate \eqref{important} imply,
for any $\varepsilon>0$,
\begin{align}
&|I_2|+|I_3|+|I_4|\notag\\
&\leq C(M)\int_0^t\!\int(\phi^2+\zeta^2)(|U_x|+\Theta_x^2)dxds
 +\varepsilon\int_0^t\!\int_{\Omega_2}\frac{\theta^\beta\zeta_x^2}v\,dxds\notag\\
&\leq\varepsilon\int_0^t\!\int_{\Omega_2}\frac{\theta^\beta\zeta_x^2}v\,dxds
                  +\delta^r C_\varepsilon(M),\label{i3}\\
&|I_5|+|I_6|+|I_{10}|+|I_{11}|+|I_{12}|\notag\\
&\leq C(M)\left(\int_0^t\!\int(\psi_x^2+\zeta_x^2)dxds\right)^{1/2}
 \left(\int_0^t\!\int(\psi^2+\zeta^2)(U_x^2+\Theta_x^2)dxds\right)^{1/2}\notag\\
&\quad+C(M)\int_0^t\!\int(\phi^2+\psi^2+\zeta^2)(U_x^2+\Theta_x^2)dxds
 \leq\delta^r C(M),\label{i4}\\
&|I_{13}|+|I_{14}|+|I_{15}|+|I_{16}|+|I_{18}|\notag\\
&\leq C(M)\int_0^t(\|\widetilde R_1\|_2+\|\widetilde R_2\|_2)ds
 +C(M)\int_0^t\!\int|\Theta_t|(\psi^2+\zeta^2)dxds
 \leq\delta^r C(M).
 \label{i10}
\end{align}
The remaining pressure and viscosity terms satisfy
\begin{align}
|I_7|+|I_8|
&\leq\varepsilon\int_0^t\!\int_{\Omega_2}\theta^\beta\zeta_x^2dxds
 +C_\varepsilon\int_0^t\!\int_{\Omega_2}(\zeta^2+\phi^2)\psi^2dxds
 +C\int_0^t\!\int_{\Omega_2}\Theta_x^2dxds\notag\\
&\leq\varepsilon\int_0^t\!\int\theta^\beta\zeta_x^2dxds
 +C_\varepsilon\int_0^t\left(
       \left\|\left(\zeta-\frac12\Theta\right)_+\right\|_\infty^2
                           +\|\psi\|_\infty^4\right)ds+\delta^r C(M),\label{i7}\\
|I_9|&\leq\varepsilon\int_0^t\!\int_{\Omega_2}\theta^\beta\zeta_x^2dxds
 +C_\varepsilon\int_0^t\!\int\psi^2\psi_x^2dxds,\label{i9}\\
|I_{17}|+|I_{19}|
&\leq C\int_0^t\!\int\psi^2\psi_x^2dxds
 +C\int_0^t\!\int_{\Omega_2}\theta\psi^2|\psi_x|dxds+\delta^r C(M)\notag\\
&\leq C\int_0^t\!\int\psi^2\psi_x^2dxds
 +C\int_0^t\left\|\left(\zeta-\frac12\Theta\right)_+\right\|_\infty^2ds
 +\delta^r C(M).
 \label{i17}
\end{align}
With the nondecreasing cutoff
\begin{equation}\label{cut-off}
\varphi_\eta(z)=\begin{cases}0,&z\leq0,\\z/\eta,&0<z<\eta,\\1,&z\geq\eta,\end{cases}
\end{equation}
integration by parts gives
\begin{align}
I_{20}
&=\frac{\tilde\kappa}{c_\nu}\lim_{\eta\downarrow0}
 \int_0^t\!\int\varphi_\eta(\zeta-\Theta)\psi^2
             \left(\frac{\theta^\beta(\zeta_x-\Theta_x)}v\right)_xdxds\notag\\
&\quad+\frac{\tilde\kappa}{c_\nu}\int_0^t\!\int_{\Omega_2}\psi^2
             \left(\frac{2\theta^\beta\Theta_x}v
                          -\frac{\Theta^\beta\Theta_x}V\right)_xdxds
 \equiv I_{20}^1+I_{20}^2.\label{i20}
\end{align}
Since $\varphi_\eta'\geq0$,
\begin{align}
I_{20}^1
&=-\frac{2\tilde\kappa}{c_\nu}\lim_{\eta\downarrow0}
 \int_0^t\!\int\varphi_\eta(\zeta-\Theta)
             \frac{\theta^\beta\psi\psi_x(\zeta_x-\Theta_x)}v\,dxds\notag\\
&\quad-\frac{\tilde\kappa}{c_\nu}\lim_{\eta\downarrow0}
 \int_0^t\!\int\varphi_\eta'(\zeta-\Theta)
             \frac{\theta^\beta\psi^2(\zeta_x-\Theta_x)^2}v\,dxds\notag\\
&\leq\varepsilon\int_0^t\!\int_{\Omega_2}\theta^\beta\zeta_x^2dxds
 +C_\varepsilon\int_0^t\!\int_{\Omega_2}\theta^\beta\psi^2\psi_x^2dxds
 +\delta^r C_\varepsilon(M)\notag\\
&\leq\varepsilon\int_0^t\!\int(\theta^\beta\zeta_x^2+\theta\psi_x^2)dxds
 +C_\varepsilon\int_0^t\!\int|\psi|^{2/(1-\beta)}\psi_x^2dxds
 +\delta^r C_\varepsilon(M).
 \label{i201}
\end{align}
The last inequality uses
$\theta^\beta\psi^2\leq\varepsilon\theta+C_\varepsilon|\psi|^{2/(1-\beta)}$
for $0<\beta<1$; for $\beta=0$ the last factor is $\psi^2$.
Moreover,
\begin{align}
I_{20}^2
&=\frac{\tilde\kappa}{c_\nu}\int_0^t\!\int_{\Omega_2}\psi^2
 \left\{\frac{2\beta\theta^{\beta-1}}v\zeta_x\Theta_x
 +\left(\frac{2\beta\theta^{\beta-1}}v
                    -\frac{\beta\Theta^{\beta-1}}V\right)\Theta_x^2
 +\left(\frac{2\theta^\beta}v-\frac{\Theta^\beta}V\right)\Theta_{xx}\right.\notag\\
&\hspace{48mm}\left.
 -\frac{2\theta^\beta}{v^2}\phi_x\Theta_x
 +\left(\frac{\Theta^\beta}{V^2}-\frac{2\theta^\beta}{v^2}\right)V_x\Theta_x
 \right\}dxds\notag\\
&\leq\frac{\tilde\kappa}{8}\int_0^t\!\int_{\Omega_2}
                       \frac{\theta^\beta\zeta_x^2}v\,dxds
 +C\int_0^t\!\int\psi^4(\Theta_x^2+|\Theta_{xx}|+V_x^2)dxds\notag\\
&\quad+C(M)\int_0^t\!\int\zeta^2(\Theta_x^2+|\Theta_{xx}|+V_x^2)dxds
 +C(M)\left(\int_0^t\!\int\theta\phi_x^2dxds\right)^{1/2}
                 \left(\int_0^t\!\int\psi^2\Theta_x^2dxds\right)^{1/2}\notag\\
&\leq\frac{\tilde\kappa}{8}\int_0^t\!\int_{\Omega_2}
                       \frac{\theta^\beta\zeta_x^2}v\,dxds
 +C\int_0^t\|\psi\|_\infty^4
                       \int(\Theta_x^2+|\Theta_{xx}|+V_x^2)dxds
 +\delta^r C(M)\notag\\
&\leq\frac{\tilde\kappa}{8}\int_0^t\!\int
                       \frac{\theta^\beta\zeta_x^2}v\,dxds
 +C\int_0^t\|\psi\|_\infty^4ds+\delta^r C(M).
 \label{i202}
\end{align}
By \eqref{basic} and \eqref{omega-bound},
\begin{align}
\int_0^t\!\int(\theta\psi_x^2+\theta^\beta\zeta_x^2)dxds
&\leq C\int_0^t\!\int_{\Omega_2}
       \left(\frac{\mu }v\psi_x^2(\zeta-\Theta)_+
                       +\frac{\theta^\beta\zeta_x^2}v\right)dxds+C.\label{5.4}
\end{align}
The Sobolev and Young inequalities give
\begin{align*}
\int_0^t\|\psi\|_\infty^4ds
&\leq4\sup_{0\leq s\leq t}\|\psi(s)\|_2^2
                      \int_0^t\|\psi_x\|_2^2ds
 \leq\varepsilon\int_0^t\!\int\theta\psi_x^2dxds+C_\varepsilon.
\end{align*}
Also,
\begin{align*}
\int_0^t\!\int\psi^2\psi_x^2dxds
 &\leq\int_0^t\!\int(1+|\psi|^{2/(1-\beta)})\psi_x^2dxds\\
&\leq\varepsilon\int_0^t\!\int\theta\psi_x^2dxds+C_\varepsilon
       +\int_0^t\!\int|\psi|^{2/(1-\beta)}\psi_x^2dxds.
\end{align*}
Substitution in \eqref{long} yields
\begin{align}
&\int\bigl[(\zeta-\Theta)_+^2+\psi^2(\zeta-\Theta)_+\bigr]dx
 +c\int_0^t\!\int(\theta\psi_x^2+\theta^\beta\zeta_x^2)dxds\notag\\
&\quad\leq C+C\int_0^t\left\|\left(\zeta-\frac12\Theta\right)_+\right\|_\infty^2ds
 +C\int_0^t\!\int|\psi|^{2/(1-\beta)}\psi_x^2dxds+\delta^r C(M).
 \label{5.3}
\end{align}
Set $b=2/(1-\beta)$. Multiplication of
\eqref{perturb}$_2$ by $|\psi|^b\psi$ gives
\begin{align}
&\frac1{b+2}\int|\psi|^{b+2}dx
 +(b+1)\int_0^t\!\int\frac{\mu }v|\psi|^b\psi_x^2dxds
 =\frac1{b+2}\int|\psi_0|^{b+2}dx\notag\\
&\quad+(b+1)\int_0^t\!\int_{\{\theta\leq3\Theta\}}
           \frac{R\zeta-p_+\phi}{v}|\psi|^b\psi_xdxds
 +(b+1)\int_0^t\!\int_{\Omega_3}
           \frac{R\zeta-p_+\phi}{v}|\psi|^b\psi_xdxds\notag\\
&\quad-(b+1)\int_0^t\!\int
 \left(\frac{\mu }v-\frac{\mu_1 }V\right)U_x|\psi|^b\psi_xdxds
 -\int_0^t\!\int\widetilde R_1|\psi|^b\psi\,dxds
 \equiv\frac1{b+2}\int|\psi_0|^{b+2}dx+\sum_{i=1}^4J_i.
 \label{5.5}
\end{align}
By \eqref{basic} and \eqref{cedu},
\begin{align}
|J_1|&\leq C\int_0^t\||\psi|^b\psi_x\|_2
 \left(\int\phi^2dx+\int_{\{\theta\leq3\Theta\}}\zeta^2dx\right)^{1/2}ds
 \leq C\int_0^t\||\psi|^b\psi_x\|_2ds.
 \label{J11}
\end{align}
For $0\leq\beta\leq1/2$, $2\leq b\leq4$ and $b\beta=b-2$. Thus
\begin{align}
\||\psi|^b\psi_x\|_2
&\leq\|\psi\|_\infty^2\||\psi|^{b\beta}\psi_x\|_2\notag\\
&\leq C\|\psi_x\|_2
 \left(\int|\psi|^{2b-4}\psi_x^2dx\right)^{1/2}
 \leq C\|\psi_x\|_2
 \left(\int\psi_x^2dx+\int|\psi|^b\psi_x^2dx\right)^{1/2}\notag\\
&\leq C_\varepsilon\int\psi_x^2dx
                         +\varepsilon\int|\psi|^b\psi_x^2dx,
 \label{J12}
\end{align}
where
\begin{equation}\label{J13}
\|\psi\|_\infty^2
 \leq2\int|\psi\psi_x|dx\leq2\|\psi\|_2\|\psi_x\|_2.
\end{equation}
For $1/2<\beta<1$, $b>4$ and
\begin{align}
\||\psi|^b\psi_x\|_2
&\leq\|\psi\|_\infty^{b/2}
                      \left(\int|\psi|^b\psi_x^2dx\right)^{1/2}\notag\\
&\leq\varepsilon\int|\psi|^b\psi_x^2dx+C_\varepsilon\|\psi\|_\infty^b
 \leq2\varepsilon\int|\psi|^b\psi_x^2dx+C_\varepsilon\int\psi_x^2dx.
 \label{J14}
\end{align}
Indeed,
\begin{align*}
\|\psi\|_\infty^b
&\leq b\int|\psi|^{b-1}|\psi_x|dx
 \leq b\left(\int|\psi|^{b-4}\psi_x^2dx\right)^{1/2}
             \left(\int|\psi|^{b+2}dx\right)^{1/2}\\
&\leq C\|\psi\|_\infty^{b/2}
                      \left(\int|\psi|^{b-4}\psi_x^2dx\right)^{1/2}.
\end{align*}
For any $\eta>0$, it follows that
\begin{align*}
\|\psi\|_\infty^b
&\leq C\int|\psi|^{b-4}\psi_x^2dx
 \leq\eta\int|\psi|^b\psi_x^2dx+C_\eta\int\psi_x^2dx.
\end{align*}
Consequently, for $0\leq\beta<1$,
\begin{equation}
|J_1|\leq C_\varepsilon\int_0^t\!\int\psi_x^2dxds
                    +C\varepsilon\int_0^t\!\int|\psi|^b\psi_x^2dxds.
 \label{J1-C}
\end{equation}
On $\Omega_3$, $\theta\leq C\zeta$ and $2b/(b+2)\leq\beta+1$; hence
\begin{align}
|J_2|
&\leq\varepsilon\int_0^t\!\int|\psi|^b\psi_x^2dxds
                 +C_\varepsilon\int_0^t\!\int_{\Omega_3}\zeta^2|\psi|^bdxds\notag\\
&\leq\varepsilon\int_0^t\!\int|\psi|^b\psi_x^2dxds
 +C_\varepsilon\int_0^t\sup_x
  \frac{(\zeta-\Theta/2)_+^{\beta+2}}{\zeta}
                         \int(\zeta^2+|\psi|^{b+2})dxds.
 \label{J21}
\end{align}
The quotient is understood to be zero when $\zeta\leq\Theta/2$. Direct calculation gives
\begin{align}
\frac{(\zeta-\Theta/2)_+^{\beta+2}}{\zeta}(x,t)
&=\int_{-\infty}^x\left\{
 \frac{(\zeta-\Theta/2)_+^{\beta+1}[(\beta+1)\zeta+\Theta/2]}{\zeta^2}\zeta_y
 -\frac{\beta+2}{2}\frac{(\zeta-\Theta/2)_+^{\beta+1}}\zeta\Theta_y\right\}dy\notag\\
&\leq C\left(\int\theta^{\beta-2}\zeta_x^2dx
                       +C(M)\int\zeta^2\Theta_x^2dx\right)^{1/2}
 \left(\sup_y\frac{(\zeta-\Theta/2)_+^{\beta+2}}\zeta
                                  \int_{\Omega_{3/2}}\zeta\,dy\right)^{1/2}.\notag
\end{align}
Taking the supremum in $x$ and using \eqref{omega-bound}, we obtain
\begin{equation}\label{J22}
\sup_x\frac{(\zeta-\Theta/2)_+^{\beta+2}}\zeta
\leq C\int\theta^{\beta-2}\zeta_x^2dx+C(M)\int\zeta^2\Theta_x^2dx.
\end{equation}
The decay of $U_x$, the Gaussian estimate \eqref{important}, and the
residual estimate for $\widetilde R_1$ yield
\begin{align}
|J_3|+|J_4|
&\leq\varepsilon\int_0^t\!\int|\psi|^b\psi_x^2dxds
 +C_\varepsilon(M)\int_0^t\!\int(\phi^2+\psi^2+\zeta^2)U_x^2dxds\notag\\
&\quad+\sup_s\|\psi(s)\|_\infty^b\sup_s\|\psi(s)\|_2
                              \int_0^t\|\widetilde R_1\|_2ds
 \leq\varepsilon\int_0^t\!\int|\psi|^b\psi_x^2dxds+\delta^r C_\varepsilon(M).
 \label{j3}
\end{align}
Consequently,
\begin{align}
&\int|\psi|^{b+2}dx+c\int_0^t\!\int|\psi|^b\psi_x^2dxds\notag\\
&\quad\leq C+C_\varepsilon\int_0^t\!\int\psi_x^2dxds
 +C\int_0^t\left(\int\theta^{\beta-2}\zeta_x^2dx\right)
                         \left(\int\zeta^2+|\psi|^{b+2}dx\right)ds
 +\delta^r C(M).
 \label{s2}
\end{align}
For every $\beta\geq0$ and $\chi\geq-1$, \eqref{omega-bound} gives
\begin{align}
 (\theta-3\Theta/2)_+^{\chi+3}(x,t)
&=(\chi+3)\int_{-\infty}^x
 (\theta-3\Theta/2)_+^{\chi+2}(\zeta_y-\Theta_y/2)dy\notag\\
&\leq C\left(\int_{\Omega_{3/2}}
          \theta^\chi(\zeta_x^2+\Theta_x^2)dx\right)^{1/2}
 \left(\int_{\Omega_{3/2}}
     (\theta-3\Theta/2)_+^{2\chi+4}\theta^{-\chi}dx\right)^{1/2}\notag\\
&\leq C\left(\int_{\Omega_{3/2}}
          \theta^\chi(\zeta_x^2+\Theta_x^2)dx\right)^{1/2}
             \|(\theta-3\Theta/2)_+\|_\infty^{(\chi+3)/2}
             \left(\int_{\Omega_{3/2}}\theta\,dx\right)^{1/2}.\notag
\end{align}
Using \eqref{omega-bound} once more gives
\begin{equation}\label{psx22}
\|(\theta-3\Theta/2)_+\|_\infty^{\chi+3}
\leq C\int_{\Omega_{3/2}}\theta^\chi\zeta_x^2dx
                         +C(M)\int\zeta^2\Theta_x^2dx.
\end{equation}
Taking $\chi=-1$ and using
$\theta^{-1}\leq\varepsilon\theta^\beta+C_\varepsilon\theta^{\beta-2}$ gives
\begin{equation}\label{cutoff-max-integral}
\int_0^t\left\|\left(\zeta-\frac12\Theta\right)_+\right\|_\infty^2ds
 \leq\varepsilon\int_0^t\!\int\theta^\beta\zeta_x^2dxds
                               +C_\varepsilon+\delta^r C_\varepsilon(M).
\end{equation}
Adding a sufficiently large fixed multiple of \eqref{s2} to \eqref{5.3},
and taking $\varepsilon$ small, yields
\begin{align}
&\int\bigl[(\zeta-\Theta)_+^2+\psi^2(\zeta-\Theta)_++|\psi|^{b+2}\bigr]dx
 +c\int_0^t\!\int\bigl[(\theta+|\psi|^b)\psi_x^2+\theta^\beta\zeta_x^2\bigr]dxds\notag\\
&\quad\leq C+\delta^r C(M)
 +C\int_0^t\left(\int\theta^{\beta-2}\zeta_x^2dx\right)
       \left(1+\int\bigl[(\zeta-\Theta)_+^2+|\psi|^{b+2}\bigr]dx\right)ds.
 \label{5.7}
\end{align}
Gronwall's inequality and \eqref{basic} give
\begin{align*}
&\sup_{0\leq s\leq t}\int(\zeta^2+|\psi|^{b+2})dx
 +\int_0^t\!\int\bigl[(1+\theta+|\psi|^b)\psi_x^2
                                  +\theta^\beta\zeta_x^2\bigr]dxds
  \leq C+\delta^r C(M).
\end{align*}
In particular,
\[
\|\psi\|_4^4
 \leq\|\psi\|_2^{2(b-2)/b}\|\psi\|_{b+2}^{2(b+2)/b}
 \leq C+\delta^rC(M).
\]
Since $\theta^{-1}\leq\theta^{\beta-2}+\theta^\beta$ and
$\psi^2\leq1+|\psi|^b$, \eqref{cutoff-max-integral} completes
the proof of \eqref{new-nergy}.
\end{proof}

\begin{lemma}\label{beta>1}
For $\beta\geq1$, one has
\begin{equation}\label{beta-new}
\int_0^T\!\int(\psi_x^2+\theta^{-1}\zeta_x^2)dxdt
 +\int_0^T\|(\theta-3\Theta/2)_+\|_\infty^2dt
 \leq C+\delta^r C(M).
\end{equation}
\end{lemma}
\begin{proof}
The temperature equation reads
\begin{equation}\label{ener}
c_\nu\theta_t+\frac{R\theta}{v}u_x
 =\tilde\kappa\left(\frac{\theta^\beta\theta_x}v\right)_x
                         +\frac{\mu }v u_x^2.
\end{equation}
For $q\geq\beta+4$, multiplication by $-(\theta^{-q}-4\Theta^{-q})_+$ gives
\begin{align}
&c_\nu\int\!\int_{\theta}^{4^{-1/q}\Theta}
                 (y^{-q}-4\Theta^{-q})_+dy\,dx
 +q\tilde\kappa\int_0^t\!\int_{\{\theta<4^{-1/q}\Theta\}}
                 \frac{\theta^{\beta-q-1}\theta_x^2}v\,dxds\notag\\
&\quad+\int_0^t\!\int\frac{\mu }v
                       u_x^2(\theta^{-q}-4\Theta^{-q})_+dxds\notag\\
&=c_\nu\int\!\int_{\theta_0}^{4^{-1/q}\Theta(\cdot,0)}
                 (y^{-q}-4\Theta(\cdot,0)^{-q})_+dy\,dx
 +R\int_0^t\!\int\frac{\theta^{1-q}}v u_x
                        (1-4(\theta/\Theta)^q)_+dxds\notag\\
&\quad+4q\tilde\kappa\int_0^t\!\int_{\{\theta<4^{-1/q}\Theta\}}
                    \frac{\theta^\beta\theta_x}v\Theta^{-q-1}\Theta_xdxds
 +4qc_\nu\int_0^t\!\int_{\{\theta<4^{-1/q}\Theta\}}
             (4^{-1/q}\Theta-\theta)\Theta^{-q-1}\Theta_tdxds.
 \label{beta-1}
\end{align}
By \eqref{basic} and \eqref{cedu},
\begin{equation}\label{beta-2}
\int_{\{\theta<4^{-1/q}\Theta\}}\theta^{1-q}dx
 \leq(q-1)\int\!\int_{\theta}^{4^{-1/q}\Theta}
                          (y^{-q}-4\Theta^{-q})_+dy\,dx+C(q).
\end{equation}
The velocity term satisfies
\begin{align}
&\int_0^t\!\int\theta^{-q}u_x^2dxds
 \leq C(q)\int_0^t\!\int\frac{\mu }v
                        u_x^2(\theta^{-q}-4\Theta^{-q})_+dxds\notag\\
&\hspace{32mm}+C(q)\int_0^t\!\int\theta^{-1}\psi_x^2dxds
                   +C(q,M)\int_0^t\!\int U_x^2dxds\notag\\
&\hspace{15mm}\leq C(q)\int_0^t\!\int\frac{\mu }v
                        u_x^2(\theta^{-q}-4\Theta^{-q})_+dxds
                   +C(q)+\delta^r C(q,M).\label{beta-3}
\end{align}
For the terms containing $\Theta_x$ or $\Theta_t$,
Lemma~\ref{decay} gives
\begin{align}
&\int_0^t\!\int_{\{\theta<4^{-1/q}\Theta\}}
 \left[\theta^{\beta-q-1}\Theta_x^2
          +(4^{-1/q}\Theta-\theta)\Theta^{-q-1}|\Theta_t|\right]dxds\notag\\
&\hspace{15mm}\leq C(q,M)\int_0^t\!\int\zeta^2(\Theta_x^2+|\Theta_t|)dxds
 \leq\delta^r C(q,M),\label{beta-4}\\
&\left|\int_0^t\!\int_{\{\theta<4^{-1/q}\Theta\}}
                \frac{\theta^\beta\theta_x}v\Theta^{-q-1}\Theta_xdxds\right|\notag\\
&\hspace{15mm}\leq C(q)\int_0^t\!\int\theta^{\beta-2}\zeta_x^2dxds
      +C(q,M)\int_0^t\!\int_{\{\theta<4^{-1/q}\Theta\}}\Theta_x^2dxds
 \leq C(q)+\delta^r C(q,M).
 \label{beta-6}
\end{align}
Furthermore,
\begin{align}
\|(1-4(\theta/\Theta)^q)_+\|_\infty^2
&\leq16q^2\left(\int_{\{\theta<4^{-1/q}\Theta\}}
 (\theta/\Theta)^{q-1}
          \frac{|\Theta\zeta_x-\zeta\Theta_x|}{\Theta^2}dx\right)^2\notag\\
&\leq C(q)\int\theta^{\beta-2}\zeta_x^2dx
                       +C(q)\int\zeta^2\Theta_x^2dx.\label{beta-7}
\end{align}
Young's inequality yields
\begin{align*}
\left|R\int\frac{\theta^{1-q}}v u_x
                        (1-4(\theta/\Theta)^q)_+dx\right|
&\leq\varepsilon\int\theta^{-q}u_x^2dx
 +C(q,\varepsilon)\|(1-4(\theta/\Theta)^q)_+\|_\infty^2
            \int_{\{\theta<4^{-1/q}\Theta\}}\theta^{1-q}dx.
\end{align*}
Choosing $\varepsilon$ small in \eqref{beta-1}, and applying Gronwall's
inequality, we obtain
\begin{align}
&\sup_{0\leq s\leq t}\int\!\int_{\theta}^{4^{-1/q}\Theta}
                 (y^{-q}-4\Theta^{-q})_+dy\,dx
 +\int_0^t\!\int\bigl(\theta^{\beta-q-1}\zeta_x^2
                                      +\theta^{-q}\psi_x^2\bigr)dxds\notag\\
&\quad\leq C(q)+\delta^r C(q,M).
 \label{theta^beta}
\end{align}
Here the integrals over $\{\theta\geq4^{-1/q}\Theta\}$ are bounded by
\eqref{basic}. For $1\leq q\leq\beta+4$, using
$\theta^{\beta-q-1}\leq\theta^{\beta-2}+\theta^{-5}$ and
$\theta^{-q}\leq\theta^{-1}+\theta^{-\beta-4}$, we obtain
\begin{align*}
\int_0^t\!\int
 (\theta^{\beta-q-1}\zeta_x^2+\theta^{-q}\psi_x^2)dxds
&\leq\int_0^t\!\int
 [(\theta^{\beta-2}+\theta^{-5})\zeta_x^2
             +(\theta^{-1}+\theta^{-\beta-4})\psi_x^2]dxds
 \leq C(q)+\delta^rC(q,M).
\end{align*}
Thus \eqref{theta^beta} gives the derivative bounds for every $q\geq1$.
Taking $q=\beta$ and then $\chi=-1$ in \eqref{psx22}, we obtain
\begin{equation}\label{beta-gradient}
\int_0^t\!\int\theta^{-1}\zeta_x^2dxds
 +\int_0^t\|(\theta-3\Theta/2)_+\|_\infty^2ds
 \leq C+\delta^r C(M).
\end{equation}
Multiplication of \eqref{ener} by $(\theta-2\Theta)_+/\theta$ gives
\begin{align}
&\int_0^t\!\int\frac{\mu }v u_x^2
                               \frac{(\theta-2\Theta)_+}\theta\,dxds\notag\\
&=c_\nu\int\left[\int_{2\Theta(x,t)}^{\theta(x,t)}
                       \frac{(y-2\Theta(x,t))_+}{y}dy
          -\int_{2\Theta(x,0)}^{\theta_0(x)}
                       \frac{(y-2\Theta(x,0))_+}{y}dy\right]dx\notag\\
&\quad+2c_\nu\int_0^t\!\int_{\Omega_2}\log\frac{\theta}{2\Theta}\Theta_tdxds
 +2\tilde\kappa\int_0^t\!\int_{\Omega_2}
       \left(\frac{\Theta\theta^{\beta-2}\theta_x^2}v
                       -\frac{\theta^{\beta-1}\theta_x\Theta_x}v\right)dxds\notag\\
&\quad+R\int_0^t\!\int\frac{(\theta-2\Theta)_+}v(\psi_x+U_x)dxds\notag\\
&\leq C+C(M)\int_0^t\!\int\zeta^2(|\Theta_t|+|U_x|+\Theta_x^2)dxds
 +C\int_0^t\!\int\theta^{\beta-2}\zeta_x^2dxds\notag\\
&\quad+\varepsilon\int_0^t\!\int\psi_x^2dxds
       +C_\varepsilon\int_0^t\!\int_{\Omega_2}(\theta-2\Theta)_+^2dxds\notag\\
&\leq C+\delta^r C_\varepsilon(M)
 +\varepsilon\int_0^t\!\int\psi_x^2dxds
 +C_\varepsilon\int_0^t\|(\theta-3\Theta/2)_+\|_\infty^{\beta+1}ds
                         \sup_s\int_{\Omega_2(s)}\theta\,dx\notag\\
&\leq C+\delta^r C_\varepsilon(M)
 +\varepsilon\int_0^t\!\int\psi_x^2dxds
 +C_\varepsilon\int_0^t\|(\theta-3\Theta/2)_+\|_\infty^{\beta+1}ds.
 \label{psx21}
\end{align}
Taking $\chi=\beta-2\geq-1$ in \eqref{psx22} gives
\begin{align}
\int_0^t\|(\theta-3\Theta/2)_+\|_\infty^{\beta+1}ds
&\leq C\int_0^t\!\int_{\Omega_{3/2}}
                   \theta^{\beta-2}\zeta_x^2dxds
       +C(M)\int_0^t\!\int\zeta^2\Theta_x^2dxds\notag\\
&\leq C+\delta^rC(M).
 \label{psx23}
\end{align}
Finally,
\begin{align}
\int_0^t\!\int\psi_x^2dxds
&\leq C\int_0^t\!\int_{\{\theta\leq3\Theta\}}
                                      \frac{\psi_x^2}\theta\,dxds
 +C\int_0^t\!\int_{\Omega_3}(u_x^2+U_x^2)dxds\notag\\
&\leq C+C\int_0^t\!\int\frac{\mu }v u_x^2
                                   \frac{(\theta-2\Theta)_+}\theta\,dxds
                                   +\delta^r C(M)\notag\\
&\leq C+\delta^r C_\varepsilon(M)
                           +C\varepsilon\int_0^t\!\int\psi_x^2dxds.
 \label{psx24}
\end{align}
Taking $\varepsilon$ small completes the proof.
\end{proof}

\begin{lemma}\label{high-order}
Under the a priori hypothesis \eqref{space}, if
$\alpha C(M)+\delta^rC(M)$ is sufficiently small, then
\begin{equation}\label{high-deri}
\sup_{0\leq t\leq T}\int(\phi_x^2+\psi_x^2)\,dx
+\int_0^T\!\int\bigl\{\zeta_x^2+(1+\theta)\phi_x^2
+\psi_{xx}^2+\psi_t^2\bigr\}\,dxdt\leq C_0.
\end{equation}
\end{lemma}
\begin{proof}
By \eqref{phi-x3}, Lemmas~\ref{lemma-ll-1} and
\ref{beta>1}, and $\alpha\log M\leq\log2$,
\[
\sup_{0\leq s\leq t}\int\phi_x^2\,dx
+\int_0^t\!\int\theta\phi_x^2\,dxds
\leq C_0+\alpha C(M)+\delta^rC(M).
\]
Here we used $\tilde\mu/2\leq\mu\leq2\tilde\mu$ and
$(\log(v/V))_x=\phi_x/v-\phi V_x/(vV)$. Thus
\begin{align}
&\sup_{0\leq s\leq t}\int\phi_x^2\,dx
+c\int_0^t\!\int(1+\theta)\phi_x^2\,dxds\notag\\
&\quad\leq C_0+\alpha C(M)+\delta^rC(M)
+C\int_0^t\!\int(\Theta-\theta)_+\phi_x^2\,dxds\notag\\
&\quad\leq C_0+\alpha C(M)+\delta^rC(M)
+\frac c2\int_0^t\!\int\phi_x^2\,dxds
+C\int_0^t\|(\Theta-\theta)_+\|_\infty^4
\int\phi_x^2\,dxds.\label{phix-p}
\end{align}
Moreover,
\begin{align*}
\int_0^t\|(\Theta-\theta)_+\|_\infty^4ds
&\leq C\int_0^t\left\{
\int(\Theta-\theta)_+\theta^{-1/2}|\zeta_x|\,dx
+\int(\Theta-\theta)_+|\zeta\Theta_x|\,dx\right\}^2ds\\
&\leq C\sup_s\int_{\theta<\Theta}\zeta^2\,dx
\int_0^t\!\int\theta^{-1}\zeta_x^2\,dxds
+C(M)\int_0^t\!\int\zeta^2\Theta_x^2\,dxds\\
&\leq C_0+\delta^rC(M).
\end{align*}
Gronwall's inequality gives
\begin{equation}\label{phix}
\sup_{0\leq t\leq T}\int\phi_x^2\,dx
+\int_0^T\!\int(1+\theta)\phi_x^2\,dxdt\leq C_0.
\end{equation}

Multiplying \eqref{perturb}$_2$ by $-\psi_{xx}$ and integrating, we obtain
\begin{align}
&\frac12\int\psi_x^2\,dx
+\int_0^t\!\int\frac{\mu }v\psi_{xx}^2\,dxds\notag\\
&=\frac12\int\psi_{0x}^2\,dx
+\int_0^t\!\int\left(\frac{R\zeta_x}{v}
-\frac{R\theta\phi_x}{v^2}
-\frac{R\zeta-p_+\phi}{v^2}V_x\right)\psi_{xx}\,dxds\notag\\
&\quad+\int_0^t\!\int\frac{\mu v_x}{v^2}
\psi_x\psi_{xx}\,dxds
-\int_0^t\!\int\frac{\mu_\theta(\theta)\theta_x}{v}
\psi_x\psi_{xx}\,dxds\notag\\
&\quad+\int_0^t\!\int
\left[\left(\frac{\mu_1 }V-\frac{\mu }v\right)U_x\right]_x
\psi_{xx}\,dxds
+\int_0^t\!\int\widetilde R_1\psi_{xx}\,dxds.
\label{6.1}
\end{align}
Combining \eqref{phix} with the $\zeta_x$- and high-temperature estimates
in Lemmas~\ref{lemma-ll-1} and \ref{beta>1}, we obtain
\begin{align}
&\left|\int_0^t\!\int\left(\frac{R\zeta_x}{v}
-\frac{R\theta\phi_x}{v^2}
-\frac{R\zeta-p_+\phi}{v^2}V_x\right)\psi_{xx}\,dxds\right|\notag\\
&\quad\leq\frac18\int_0^t\!\int\frac{\mu }v\psi_{xx}^2\,dxds
+C\int_0^t\!\int\bigl\{\zeta_x^2+\theta^2\phi_x^2
+(\phi^2+\zeta^2)V_x^2\bigr\}\,dxds\notag\\
&\quad\leq\frac18\int_0^t\!\int\frac{\mu }v\psi_{xx}^2\,dxds
+C\int_0^t\!\int\zeta_x^2\,dxds+C_0+\delta^rC(M)\notag\\
&\qquad+C\sup_{0\leq s\leq t}\|\phi_x(s)\|_2^2
\int_0^t\left\|\left(\theta-\frac32\Theta\right)_+\right\|_\infty^2ds\notag\\
&\quad\leq\frac18\int_0^t\!\int\frac{\mu }v\psi_{xx}^2\,dxds
+C\int_0^t\!\int\zeta_x^2\,dxds+C_0+\delta^rC(M).
\label{psi-xx-1}
\end{align}
The two terms containing derivatives of $\mu/v$ satisfy
\begin{align}
&\left|\int_0^t\!\int\frac{\mu v_x}{v^2}
\psi_x\psi_{xx}\,dxds\right|\notag\\
&\quad\leq\frac1{16}\int_0^t\!\int\frac{\mu }v\psi_{xx}^2\,dxds
+C\int_0^t\!\int(\phi_x^2+V_x^2)\psi_x^2\,dxds\notag\\
&\quad\leq\frac1{16}\int_0^t\!\int\frac{\mu }v\psi_{xx}^2\,dxds
+C\sup_s\|\phi_x(s)\|_2^2
\int_0^t\|\psi_x\|_2\|\psi_{xx}\|_2\,ds+\delta^rC(M)\notag\\
&\quad\leq\frac18\int_0^t\!\int\frac{\mu }v\psi_{xx}^2\,dxds
+C\int_0^t\|\psi_x\|_2^2ds+\delta^rC(M),\label{psi-xx-2}\\
&\left|\int_0^t\!\int\frac{\mu_\theta(\theta)\theta_x}{v}
\psi_x\psi_{xx}\,dxds\right|\notag\\
&\quad\leq\alpha C(M)\sup_s\|\theta_x(s)\|_2
\left(\int_0^t\|\psi_x\|_\infty^2ds\right)^{1/2}
\left(\int_0^t\|\psi_{xx}\|_2^2ds\right)^{1/2}\notag\\
&\quad\leq\alpha C(M).
\label{psi-xx-mu}
\end{align}
For $\partial_x\!\left[\left(\frac{\mu_1}{V}-\frac{\mu}{v}\right)U_x\right]$,
we first write
\begin{align}
&\left[\left(\frac{\mu_1 }V-\frac{\mu }v\right)U_x\right]_x
\notag\\
&\quad=\mu \left\{\frac{\phi U_{xx}}{vV}
-\frac{\phi(\phi+2V)U_xV_x}{v^2V^2}
+\frac{U_x\phi_x}{v^2}\right\}
+\frac{\mu_\theta(\theta)\theta_x\phi U_x}{vV}\notag\\
&\qquad+\bigl(\mu_1 -\mu \bigr)
\left(\frac{U_x}{V}\right)_x
+\bigl(\mu_\theta(\Theta)\Theta_x
-\mu_\theta(\theta)\theta_x\bigr)\frac{U_x}{V}.\notag
\end{align}
Hence
\begin{align}
&\left|\int_0^t\!\int
\left[\left(\frac{\mu_1 }V-\frac{\mu }v\right)U_x\right]_x
\psi_{xx}\,dxds\right|\notag\\
&\quad\leq\frac18\int_0^t\!\int\frac{\mu }v\psi_{xx}^2\,dxds
+C(M)\int_0^t\!\int\bigl\{\phi^2U_{xx}^2
+(\phi^2+\phi^4)U_x^2V_x^2+U_x^2\phi_x^2\bigr\}\,dxds
+\alpha C(M)\notag\\
&\quad\leq\frac18\int_0^t\!\int\frac{\mu }v\psi_{xx}^2\,dxds
+\alpha C(M)+\delta^rC(M).\label{psi-xx-3}
\end{align}
Also,
\begin{align}
&\left|\int_0^t\!\int\widetilde R_1\psi_{xx}\,dxds\right|
\leq\frac18\int_0^t\!\int\frac{\mu }v\psi_{xx}^2\,dxds
+C\int_0^t\|\widetilde R_1\|_2^2ds\notag\\
&\qquad\leq\frac18\int_0^t\!\int\frac{\mu }v\psi_{xx}^2\,dxds
+\delta^rC(M).\label{psi-xx-4}
\end{align}
Substitution into \eqref{6.1} gives
\begin{equation}\label{psi-xx}
\|\psi_x(t)\|_2^2+c\int_0^t\|\psi_{xx}\|_2^2ds
\leq C_0+C\int_0^t\!\int\zeta_x^2\,dxds
+\alpha C(M)+\delta^rC(M).
\end{equation}

For $\beta\geq2$, taking $q=\beta-1$ in \eqref{theta^beta} gives
\[
\int_0^T\!\int\zeta_x^2\,dxdt\leq C_0.
\]
For $0\leq\beta<2$, multiply \eqref{ener} by
$(\theta-2\Theta)_+\theta^{-\beta/2}$. Then
\begin{align}
&c_{\nu}\int\!\int_{2\Theta(x,t)}^{\theta(x,t)}
(y-2\Theta(x,t))_+y^{-\beta/2}\,dy\,dx\notag\\
&\quad+\tilde\kappa\int_0^t\!\int_{\Omega_2}
\frac{\theta^{\beta/2}}v
\left(1-\frac\beta2+\frac{\beta\Theta}{\theta}\right)\theta_x^2\,dxds\notag\\
&=c_{\nu}\int\!\int_{2\Theta(x,0)}^{\theta_0(x)}
(y-2\Theta(x,0))_+y^{-\beta/2}\,dy\,dx\notag\\
&\quad+2\tilde\kappa\int_0^t\!\int_{\Omega_2}
\frac{\theta^{\beta/2}\theta_x\Theta_x}{v}\,dxds
-2c_{\nu}\int_0^t\!\int_{\Omega_2}\Theta_t
\int_{2\Theta}^{\theta}y^{-\beta/2}\,dy\,dxds\notag\\
&\quad-R\int_0^t\!\int\frac{\theta^{1-\beta/2}(\theta-2\Theta)_+}{v}
(\psi_x+U_x)\,dxds\notag\\
&\quad+\int_0^t\!\int\frac{\mu }v
(\theta-2\Theta)_+\theta^{-\beta/2}
(\psi_x^2+2\psi_xU_x+U_x^2)\,dxds.\label{psil2l4}
\end{align}
The contributions containing only $\psi_x$ and $\psi_x^2$ in
\eqref{psil2l4} satisfy
\begin{align*}
\int\theta^{1-\beta/2}(\theta-2\Theta)_+|\psi_x|\,dx
&\leq C\int(\theta-2\Theta)_+\theta^{2-\beta/2}\,dx
+C\int(\theta-2\Theta)_+\theta^{-\beta/2}\psi_x^2\,dx\\
&\leq C\left\|\left(\theta-\frac32\Theta\right)_+\right\|_\infty^2
\int_{\Omega_2}\theta^{1-\beta/2}\,dx
+C\|(\theta-2\Theta)_+\|_\infty\|\psi_x\|_2^2\\
&\leq C\left\|\left(\theta-\frac32\Theta\right)_+\right\|_\infty^2
+C\|\psi_x\|_2^4,
\end{align*}
and
\begin{align*}
\int\frac{\mu }v(\theta-2\Theta)_+\theta^{-\beta/2}\psi_x^2\,dx
&\leq C\left\|\left(\theta-\frac32\Theta\right)_+\right\|_\infty^2
+C\|\psi_x\|_2^4.
\end{align*}
The contributions in \eqref{psil2l4} containing $\Theta_x$, $\Theta_t$,
or $U_x$ are bounded by
\begin{align*}
&\frac{(2-\beta)\tilde\kappa}{4}
\int_0^t\!\int_{\Omega_2}\frac{\theta^{\beta/2}\theta_x^2}{v}\,dxds\\
&\quad+C(M)\int_0^t\!\int
\bigl\{\zeta^2(\Theta_x^2+|\Theta_t|+|U_x|+U_x^2)
+U_x^2\psi_x^2\bigr\}\,dxds\\
&\leq\frac{(2-\beta)\tilde\kappa}{4}
\int_0^t\!\int_{\Omega_2}\frac{\theta^{\beta/2}\theta_x^2}{v}\,dxds
+\delta^rC(M).
\end{align*}
Consequently,
\begin{align}
\int_0^t\!\int\zeta_x^2\,dxds
&\leq C\int_0^t\!\int_{\theta\leq2\Theta}\theta^{\beta-2}\zeta_x^2\,dxds
+C\int_0^t\!\int_{\Omega_2}\theta^{\beta/2}\theta_x^2\,dxds
+C\int_0^t\!\int_{\Omega_2}\Theta_x^2\,dxds\notag\\
&\leq C_0+C\int_0^t\|\psi_x\|_2^4ds+\delta^rC(M).
\label{beta0-2}
\end{align}
Together with \eqref{psi-xx},
\begin{align*}
\|\psi_x(t)\|_2^2+c\int_0^t\|\psi_{xx}\|_2^2ds
&\leq C_0+C\int_0^t\|\psi_x\|_2^2\|\psi_x\|_2^2ds.
\end{align*}
Gronwall's inequality gives
\begin{align*}
\sup_{0\leq s\leq t}\|\psi_x(s)\|_2^2
&\leq C_0\exp\left(C\int_0^t\|\psi_x\|_2^2ds\right)\leq C_0.
\end{align*}
Thus, for every $\beta\geq0$,
\begin{equation}\label{zetax-f}
\sup_{0\leq t\leq T}\|\psi_x(t)\|_2^2
+\int_0^T\!\int(\psi_{xx}^2+\psi_x^2+\zeta_x^2)\,dxdt\leq C_0.
\end{equation}
Solving \eqref{perturb}$_2$ for $\psi_t$ and using \eqref{phix},
\eqref{psi-xx}, and \eqref{beta0-2} yields
\begin{align*}
\int_0^T\!\int\psi_t^2\,dxdt
&\leq C\int_0^T\!\int\left\{
\zeta_x^2+\theta^2\phi_x^2+\psi_{xx}^2
+\phi_x^2\psi_x^2+(\phi^2+\zeta^2)V_x^2
+\widetilde R_1^2\right\}\,dxdt\\
&\quad+\alpha C(M)+\delta^rC(M)\leq C_0.
\end{align*}
Combining the displayed $\psi_t$ estimate with \eqref{phix} and
\eqref{zetax-f} establishes \eqref{high-deri}.
\end{proof}

\begin{lemma}\label{theta-supbel}
Under the assumptions of Lemma~\ref{high-order}, there exists $C_0>0$,
independent of $T$ and $M$, such that
\begin{equation}\label{theta-fina}
C_0^{-1}\leq\theta(x,t)\leq C_0,
\qquad (x,t)\in\mathbb R\times[0,T].
\end{equation}
\end{lemma}
\begin{proof}
For $q>\beta+1$ and $q\geq2$, multiplying \eqref{ener} by
$(\theta-2\Theta)_+^{q-1}$ gives
\begin{align}
&\frac{c_{\nu}}q\int(\theta-2\Theta)_+^q\,dx
 +(q-1)\tilde\kappa\int_0^t\!\int_{\Omega_2}
 \frac{\theta^\beta(\theta-2\Theta)_+^{q-2}\zeta_x^2}{v}\,dxds\notag\\
&=\frac{c_{\nu}}q\int(\theta_0-2\Theta(x,0))_+^q\,dx
 +(q-1)\tilde\kappa\int_0^t\!\int_{\Omega_2}
 \frac{\theta^\beta(\theta-2\Theta)_+^{q-2}\Theta_x^2}{v}\,dxds\notag\\
&\quad-2c_{\nu}\int_0^t\!\int\Theta_t(\theta-2\Theta)_+^{q-1}\,dxds
 -R\int_0^t\!\int\frac{\theta(\theta-2\Theta)_+^{q-1}}v
 (\psi_x+U_x)\,dxds\notag\\
&\quad+\int_0^t\!\int\frac{\mu }v
 (\theta-2\Theta)_+^{q-1}
 (\psi_x^2+2\psi_xU_x+U_x^2)\,dxds.
\label{int1}
\end{align}
By \eqref{psx22} and \eqref{zetax-f},
\begin{align*}
\int(\theta-2\Theta)_+^{q-1}\psi_x^2\,dx
&\leq\|(\theta-2\Theta)_+\|_\infty^{q-1}\|\psi_x\|_2^2\\
&\leq\varepsilon_2\left\|\left(\theta-\frac32\Theta\right)_+\right\|_\infty^{\beta+q+1}
 +C(q,\varepsilon_2)\|\psi_x\|_2^{2(\beta+q+1)/(\beta+2)}.
\end{align*}
Consequently,
\begin{align*}
\int\theta(\theta-2\Theta)_+^{q-1}|\psi_x|\,dx
&\leq C(\varepsilon_1)\int(\theta-2\Theta)_+^{q-1}\psi_x^2\,dx
 +\varepsilon_1\int\theta^2(\theta-2\Theta)_+^{q-1}\,dx\\
&\leq C(\varepsilon_1)\int(\theta-2\Theta)_+^{q-1}\psi_x^2\,dx
 +C\varepsilon_1\left\|\left(\theta-\frac32\Theta\right)_+\right\|_\infty^q
 \int_{\Omega_2}\theta\,dx\\
&\leq\bigl(C(\varepsilon_1)\varepsilon_2+C\varepsilon_1\bigr)
 \left\|\left(\theta-\frac32\Theta\right)_+\right\|_\infty^{\beta+q+1}
 +C(q,\varepsilon_1,\varepsilon_2)\|\psi_x\|_2^{2(\beta+q+1)/(\beta+2)}\\
&\quad+C\varepsilon_1\left\|\left(\theta-\frac32\Theta\right)_+\right\|_\infty^2.
\end{align*}
To estimate the first term on the right, we use
\begin{align*}
\left\|\left(\theta-\frac32\Theta\right)_+\right\|_\infty^{\beta+q+1}
&\leq C(q)\int_{\theta>3\Theta/2}\theta^{\beta+q-2}\zeta_x^2\,dx
 +C(q,M)\int\zeta^2\Theta_x^2\,dx\\
&\leq C(q)\int_{\Omega_2}\frac{\theta^\beta(\theta-2\Theta)_+^{q-2}\zeta_x^2}{v}\,dx
 +C(q)\int\theta^{\beta-2}\zeta_x^2\,dx
 +C(q,M)\int\zeta^2\Theta_x^2\,dx.
\end{align*}
The velocity term is bounded by
\begin{align*}
\int_0^T\|\psi_x\|_2^{2(\beta+q+1)/(\beta+2)}dt
&\leq\left(\sup_t\|\psi_x(t)\|_2^2\right)^{(q-1)/(\beta+2)}
 \int_0^T\|\psi_x\|_2^2dt\leq C(q).
\end{align*}
The terms in \eqref{int1} containing $\Theta_x$, $\Theta_t$, or $U_x$
have absolute value at most
\[
C(q,M)\int_0^t\!\int\bigl\{
\zeta^2(\Theta_x^2+|\Theta_t|+|U_x|+U_x^2)
+U_x^2\psi_x^2\bigr\}\,dxds\leq\delta^rC(q,M).
\]
Choosing $\varepsilon_1>0$ and then $\varepsilon_2>0$ sufficiently small yields
\begin{align}
&\sup_{0\leq t\leq T}\int(\theta-2\Theta)_+^q\,dx
 +\int_0^T\!\int\theta^{\beta+q-2}\zeta_x^2\,dxdt
 \leq C(q)+\delta^rC(q,M),
 \qquad q>\beta+1,\quad q\geq2.
\label{int-t+}
\end{align}
Indeed,
\begin{align*}
\int\theta^{\beta+q-2}\zeta_x^2\,dx
&\leq C(q)\int_{\Omega_2}
 \frac{\theta^\beta(\theta-2\Theta)_+^{q-2}\zeta_x^2}{v}\,dx
 +C(q)\int\theta^{\beta-2}\zeta_x^2\,dx.
\end{align*}
The initial integral satisfies
\begin{align*}
\int(\theta_0-2\Theta(x,0))_+^q\,dx
&\leq\|(\theta_0-2\Theta(x,0))_+\|_\infty^{q-2}
 \int(\theta_0-2\Theta(x,0))_+^2\,dx\leq C(q).
\end{align*}
Taking $q=\beta+2$ and $q=\beta+4$ in \eqref{int-t+}, and choosing
$\delta^r[C(\beta+2,M)+C(\beta+4,M)]\leq1$, gives
\begin{equation}\label{int-t+-choices}
\int_0^T\!\int\bigl(\theta^{2\beta}+\theta^{2\beta+2}\bigr)
\zeta_x^2\,dxdt\leq C_0.
\end{equation}

Multiplying \eqref{perturb}$_3$ by $\theta^\beta\zeta_t$, we have
\begin{align}
c_{\nu}\int\theta^\beta\zeta_t^2\,dx
&=\tilde\kappa\int\left(\frac{\theta^\beta\zeta_x}{v}\right)_x
\theta^\beta\zeta_t\,dx
+\tilde\kappa\int\left[\left(\frac{\theta^\beta}{v}
-\frac{\Theta^\beta}{V}\right)\Theta_x\right]_x\theta^\beta\zeta_t\,dx
\notag\\
&\quad+\int\left(\frac{\mu u_x^2}{v}
-\frac{\mu_1 U_x^2}{V}\right)\theta^\beta\zeta_t\,dx
-\int(pu_x-p_+U_x)\theta^\beta\zeta_t\,dx\notag\\
&\quad-\int\widetilde R_2\theta^\beta\zeta_t\,dx
=\sum_{i=1}^{5}K_i.
\label{supthet}
\end{align}
For $K_1$,
\begin{align}
K_1
&=-\tilde\kappa\int\frac{\theta^\beta\zeta_x}{v}
\left\{(\theta^\beta\zeta_x)_t
+\beta\theta^{\beta-1}(\Theta_x\zeta_t-\Theta_t\zeta_x)\right\}\,dx
\notag\\
&=-\frac{\tilde\kappa}{2}\frac d{dt}
\int\frac{\theta^{2\beta}\zeta_x^2}{v}\,dx
-\frac{\tilde\kappa}{2}\int\frac{(\psi_x+U_x)\theta^{2\beta}\zeta_x^2}{v^2}\,dx
\notag\\
&\quad-\beta\tilde\kappa\int\frac{\theta^{2\beta-1}\zeta_x}{v}
(\Theta_x\zeta_t-\Theta_t\zeta_x)\,dx.
\label{k11}
\end{align}
The last two terms satisfy
\begin{align}
&\left|\frac{\tilde\kappa}{2}\int
\frac{(\psi_x+U_x)\theta^{2\beta}\zeta_x^2}{v^2}\,dx\right|
\leq C\|\psi_x\|_\infty\int\frac{\theta^{2\beta}\zeta_x^2}{v}\,dx
+C(M)\|U_x\|_\infty\|\zeta_x\|_2^2,\notag\\
&\beta\tilde\kappa\left|\int\frac{\theta^{2\beta-1}\zeta_x}{v}
(\Theta_x\zeta_t-\Theta_t\zeta_x)\,dx\right|\notag\\
&\qquad\leq\frac{c_{\nu}}{16}\int\theta^\beta\zeta_t^2\,dx
+C(M)\int(\Theta_x^2+|\Theta_t|)\zeta_x^2\,dx.
\label{k12}
\end{align}
For $K_2$,
\begin{align}
&\left[\left(\frac{\theta^\beta}{v}-\frac{\Theta^\beta}{V}\right)\Theta_x\right]_x
\notag\\
&\quad=\frac{\beta\theta^{\beta-1}}v\zeta_x\Theta_x
-\frac{\theta^\beta}{v^2}\phi_x\Theta_x
+\beta\left(\frac{\theta^{\beta-1}}v-\frac{\Theta^{\beta-1}}V\right)\Theta_x^2
\notag\\
&\qquad-\left(\frac{\theta^\beta}{v^2}-\frac{\Theta^\beta}{V^2}\right)V_x\Theta_x
+\left(\frac{\theta^\beta}{v}-\frac{\Theta^\beta}{V}\right)\Theta_{xx},\notag
\end{align}
and therefore
\begin{align}
|K_2|&\leq\frac{c_{\nu}}{16}\int\theta^\beta\zeta_t^2\,dx
+C(M)\int\left\{(\phi_x^2+\zeta_x^2)\Theta_x^2
+(\phi^2+\zeta^2)(\Theta_{xx}^2+\Theta_x^4+V_x^2\Theta_x^2)\right\}\,dx.
\label{k2}
\end{align}
Since $\mu \leq2\tilde\mu$,
\begin{align}
|K_3|&=\left|\int\left\{\frac{\mu }v(\psi_x^2+2\psi_xU_x)
+\left(\frac{\mu }v-\frac{\mu_1 }V\right)U_x^2\right\}
\theta^\beta\zeta_t\,dx\right|\notag\\
&\leq\frac{c_{\nu}}8\int\theta^\beta\zeta_t^2\,dx
+C\int\theta^\beta\psi_x^4\,dx
+C(M)\int\left\{\psi_x^2U_x^2+(\phi^2+\zeta^2)U_x^4\right\}\,dx.
\label{k3}
\end{align}
For $\beta>0$, by \eqref{cedu} and \eqref{zetax-f},
\begin{align*}
\int\theta^\beta\psi_x^4\,dx
&\leq\left[C+\|(\theta^\beta-2\Theta^\beta)_+\|_\infty\right]
                  \|\psi_x\|_\infty^2\|\psi_x\|_2^2\\
&\leq C\left[1+
 \int_{\{\theta>2^{1/\beta}\Theta\}}
 \bigl|\beta\theta^{\beta-1}\zeta_x
       +\beta(\theta^{\beta-1}-2\Theta^{\beta-1})\Theta_x\bigr|\,dx\right]
                  \|\psi_x\|_\infty^2\|\psi_x\|_2^2\\
&\leq C\left[1+
 \left(\int_{\{\theta>2^{1/\beta}\Theta\}}
                     \theta^{2\beta-2}\zeta_x^2\,dx\right)^{1/2}
 +C(M)\left(\int\zeta^2\Theta_x^2\,dx\right)^{1/2}\right]
                  \|\psi_x\|_\infty^2\|\psi_x\|_2^2\\
&\leq C\int(\theta^{-1}+\theta^{2\beta-1})\zeta_x^2\,dx
 +C\|\psi_x\|_2^2\|\psi_{xx}\|_2^2
 +C\|\psi_x\|_2\|\psi_{xx}\|_2
 +C(M)\int\zeta^2\Theta_x^2\,dx\\
&\leq C\int(\theta^{-1}+\theta^{2\beta})\zeta_x^2\,dx
 +C\bigl(\|\psi_x\|_2^2+\|\psi_{xx}\|_2^2\bigr)
 +C(M)\int\zeta^2\Theta_x^2\,dx.
\end{align*}
For $\beta=0$, directly,
\[
\int\psi_x^4\,dx
\leq\|\psi_x\|_\infty^2\|\psi_x\|_2^2
\leq2\|\psi_x\|_2^3\|\psi_{xx}\|_2
\leq C_0\bigl(\|\psi_x\|_2^2+\|\psi_{xx}\|_2^2\bigr).
\]
Similarly,
\begin{align}
|K_4|&\leq\left|\int\frac{R\theta^{\beta+1}}v\psi_x\zeta_t\,dx\right|
+\left|\int(p-p_+)U_x\theta^\beta\zeta_t\,dx\right|\notag\\
&\leq\frac{c_{\nu}}8\int\theta^\beta\zeta_t^2\,dx
+C\int\theta^{\beta+2}\psi_x^2\,dx
+C(M)\int(\phi^2+\zeta^2)U_x^2\,dx,
\label{k4}
\end{align}
where
\begin{align*}
\int\theta^{\beta+2}\psi_x^2\,dx
&\leq\left[C+\|(\theta^{\beta+2}-2\Theta^{\beta+2})_+\|_\infty\right]
                        \|\psi_x\|_2^2\\
&\leq C\left[1+\int_{\{\theta>2^{1/(\beta+2)}\Theta\}}
 \left|\theta^{\beta+1}\zeta_x
 +(\theta^{\beta+1}-2\Theta^{\beta+1})\Theta_x\right|\,dx\right]
                        \|\psi_x\|_2^2\\
&\leq C\left[1+
 \left(\int\theta^{2\beta+2}\zeta_x^2\,dx\right)^{1/2}
 +C(M)\int|\Theta_x|\,dx\right]\|\psi_x\|_2^2\\
&\leq C\left[1+\int\theta^{2\beta+2}\zeta_x^2\,dx\right]
                        \|\psi_x\|_2^2
 +\delta C(M)\|\psi_x\|_2^2.
\end{align*}
Finally,
\begin{equation}\label{k5}
|K_5|\leq\frac{c_{\nu}}{16}\int\theta^\beta\zeta_t^2\,dx
+C(M)\int\widetilde R_2^2\,dx.
\end{equation}
It follows from \eqref{supthet}--\eqref{k5} that
\begin{align}
&\frac d{dt}\int\frac{\theta^{2\beta}\zeta_x^2}{v}\,dx
+c\int\theta^\beta\zeta_t^2\,dx\notag\\
&\quad\leq C\|\psi_x\|_\infty\int\frac{\theta^{2\beta}\zeta_x^2}{v}\,dx
+C\int\bigl(\theta^{\beta+2}\psi_x^2+\theta^\beta\psi_x^4\bigr)\,dx
+C(M)\int|\Theta_t|\zeta_x^2\,dx\notag\\
&\qquad+C(M)\|U_x\|_\infty\|\zeta_x\|_2^2
+C(M)\int\left\{\widetilde R_2^2
+(\phi_x^2+\zeta_x^2)\Theta_x^2+\psi_x^2U_x^2\right.\notag\\
&\hspace{40mm}\left.+(\phi^2+\zeta^2)
(U_x^2+U_x^4+\Theta_{xx}^2+\Theta_x^4+V_x^2\Theta_x^2)\right\}\,dx.
\label{thermal-gradient-differential}
\end{align}
By \eqref{profile:first-supremum} and \eqref{zetax-f},
\begin{align*}
C(M)\int_0^T\!\int(\Theta_x^2+|\Theta_t|)\zeta_x^2\,dxdt
&\leq C(M)\left(\sup_t\|\Theta_x(t)\|_\infty^2
+\sup_t\|\Theta_t(t)\|_\infty\right)
\int_0^T\|\zeta_x(t)\|_2^2dt\\
&\leq C(M)(\delta^2+\delta)\leq\delta^rC(M),
\qquad0\leq\delta\leq1,\quad0<r\leq1.
\end{align*}
Combining \eqref{k3}--\eqref{k4}, \eqref{zetax-f},
\eqref{int-t+-choices}, and \eqref{profile:first-supremum}, we obtain
\begin{align*}
&\int_0^T\!\int\bigl(\theta^\beta\psi_x^4
                         +\theta^{\beta+2}\psi_x^2\bigr)\,dxdt\\
&\quad\leq C\int_0^T\!\int(\theta^{-1}+\theta^{2\beta})\zeta_x^2\,dxdt
 +C\int_0^T(\|\psi_x\|_2^2+\|\psi_{xx}\|_2^2)dt\\
&\qquad+C\sup_t\|\psi_x(t)\|_2^2
           \int_0^T\!\int\theta^{2\beta+2}\zeta_x^2\,dxdt
 +\delta^rC(M)\leq C_0.
\end{align*}
Moreover,
\[
\int_0^T\|\psi_x\|_\infty^2dt
\leq2\int_0^T\|\psi_x\|_2\|\psi_{xx}\|_2dt\leq C_0.
\]
Using
\[
\|\psi_x\|_\infty\int\frac{\theta^{2\beta}\zeta_x^2}{v}\,dx
\leq C\left(\int\frac{\theta^{2\beta}\zeta_x^2}{v}\,dx\right)^2
+C\|\psi_x\|_\infty^2
\]
in \eqref{thermal-gradient-differential}, we obtain by Gronwall's inequality
\begin{align}
&\sup_{0\leq s\leq t}\int\frac{\theta^{2\beta}\zeta_x^2}{v}(x,s)\,dx
+c\int_0^t\!\int\theta^\beta\zeta_t^2\,dxds\notag\\
&\quad\leq\left[C_0+C\int\frac{\theta_0^{2\beta}\zeta_{0x}^2}{v_0}\,dx
+\alpha C(M)+\delta^rC(M)\right]
\exp\left(C\int_0^t\!\int\frac{\theta^{2\beta}\zeta_x^2}{v}\,dxds\right)
\leq C_0.
\label{sup-tx}
\end{align}
For $\theta(x,t)>2\Theta(x,t)$,
\begin{align}
(\theta(x,t)-2\Theta(x,t))_+
&\leq\int_{\Omega_2}(|\zeta_x|+|\Theta_x|)\,dx\notag\\
&\leq\left(\int\theta^{2\beta}\zeta_x^2\,dx\right)^{1/2}
\left(\int_{\Omega_2}\theta^{-2\beta}\,dx\right)^{1/2}
+\int|\Theta_x|\,dx\notag\\
&\leq C\left(\int\theta^{2\beta}\zeta_x^2\,dx\right)^{1/2}
\left(\int_{\Omega_2}\theta\,dx\right)^{1/2}+C\delta
\leq C_0.
\label{the-up}
\end{align}
Hence
\begin{equation}\label{theta-up}
\theta(x,t)\leq C_0.
\end{equation}

By \eqref{basic}, \eqref{high-deri} and \eqref{theta-up},
\begin{align*}
\int\zeta^6\,dx
&\leq C\left(\int\zeta^2\,dx\right)^2\int\zeta_x^2\,dx
 \leq C_0\int\zeta_x^2\,dx.
\end{align*}
Integrating in time gives
\begin{equation}
\int_0^T\!\int\zeta^6\,dxdt\leq C_0.
\label{dt-int}
\end{equation}
Sobolev's inequality and \eqref{sup-tx} give
\begin{align*}
\|\theta^{\beta+1}-\Theta^{\beta+1}\|_\infty^2
&\leq C\left(\int|\theta^{\beta+1}-\Theta^{\beta+1}|^6\,dx\right)^{1/4}
 \left(\int|(\theta^{\beta+1}-\Theta^{\beta+1})_x|^2\,dx\right)^{1/4}\\
&\leq C\left(\int\zeta^6\,dx\right)^{1/4}
 \left[\int\theta^{2\beta}\zeta_x^2\,dx
      +\int\Theta_x^2|\theta^\beta-\Theta^\beta|^2\,dx\right]^{1/4}\\
&\leq C_0\left(\int\zeta^6\,dx\right)^{1/4}.
\end{align*}
Choose $T_0\geq1$ so large that
\[
C_0(C_0/T_0)^{1/4}
\leq\frac14\min\{\theta_-,\theta_+\}^{2\beta+2}.
\]
For every $T_0\leq t\leq T$, some $s\in[t-T_0,t]$ satisfies
\[
\int\zeta^6(x,s)\,dx
\leq\frac1{T_0}\int_{t-T_0}^t\!\int\zeta^6\,dxdr
 \leq\frac{C_0}{T_0}.
\]
For $s\in[t-T_0,t]$ satisfying
$\int\zeta^6(x,s)\,dx\leq C_0/T_0$, the Sobolev estimate for
$\theta^{\beta+1}-\Theta^{\beta+1}$ and the choice of $T_0$ yield
\begin{align*}
\theta^{\beta+1}(x,s)
&\geq\Theta^{\beta+1}(x,s)
 -\|\theta^{\beta+1}(s)-\Theta^{\beta+1}(s)\|_\infty
  \geq\frac12\min\{\theta_-,\theta_+\}^{\beta+1}.
\end{align*}
Thus $\inf_x\theta(x,s)\geq2^{-1/(\beta+1)}
\min\{\theta_-,\theta_+\}\geq c>0$.
By \eqref{ener},
\begin{align*}
c_{\nu}(\theta^{-1})_t
 -\left(\frac{\kappa }v(\theta^{-1})_x\right)_x
&=-\frac{2\kappa \theta_x^2}{v\theta^3}
 +\frac{Ru_x}{v\theta}-\frac{\mu u_x^2}{v\theta^2}\\
&\leq\frac{R^2}{4\mu v}\leq C_0.
\end{align*}
By \eqref{basic},
\[
\int_{\theta<c/2}dx
\leq C_0\int\Phi(\theta/\Theta)\,dx\leq C_0.
\]
For $p>2$, multiplication by $(\theta^{-1}-2/c)_+^p$ gives
\begin{align*}
&\frac{c_{\nu}}{p+1}\frac d{dt}
 \int(\theta^{-1}-2/c)_+^{p+1}\,dx
 +p\int\frac{\kappa }v
 (\theta^{-1}-2/c)_+^{p-1}|(\theta^{-1})_x|^2\,dx\\
&\quad\leq C_0\int(\theta^{-1}-2/c)_+^p\,dx
  \leq C_0C_0^{1/(p+1)}\|(\theta^{-1}-2/c)_+\|_{p+1}^p.
\end{align*}
Since $\|(\theta^{-1}(s)-2/c)_+\|_{p+1}=0$, integration yields
\[
\|(\theta^{-1}(t)-2/c)_+\|_{p+1}
\leq C_0C_0^{1/(p+1)}(t-s).
\]
Letting $p\to\infty$, we obtain
\[
\|(\theta^{-1}(t)-2/c)_+\|_\infty
\leq\liminf_{p\to\infty}\|(\theta^{-1}(t)-2/c)_+\|_{p+1}
\leq C_0(t-s)\leq C_0T_0.
\]
The $L^\infty$ bound for $(\theta^{-1}-2/c)_+$ yields
\begin{equation}\label{t0-infty}
\theta(x,t)\geq(2/c+C_0T_0)^{-1},
\qquad T_0\leq t\leq T.
\end{equation}
For $0\leq t\leq\min\{T,T_0\}$, multiply the differential inequality for
$\theta^{-1}$ by $[\theta^{-1}-2/\min\{\underline\theta_0,c\}]_+^p$ and
repeat the calculation for $(\theta^{-1}-2/c)_+^p$ with $2/c$ replaced by
$2/\min\{\underline\theta_0,c\}$.
Since its value at $t=0$ is zero, letting $p\to\infty$ yields
\begin{equation}\label{0-t0}
\theta(x,t)\geq
\left(\frac2{\min\{\underline\theta_0,c\}}+C_0T_0\right)^{-1},
\qquad0\leq t\leq\min\{T,T_0\}.
\end{equation}
Combining \eqref{theta-up}, \eqref{t0-infty}, and \eqref{0-t0}
establishes \eqref{theta-fina}.
\end{proof}

\begin{lemma}\label{txxtt}
Under the assumptions of Lemma~\ref{high-order},
\begin{equation}\label{zetag}
\sup_{0\leq t\leq T}\int\zeta_x^2\,dx
+\int_0^T\!\int(\zeta_t^2+\zeta_{xx}^2)\,dxdt\leq C_0.
\end{equation}
\end{lemma}
\begin{proof}
By \eqref{theta-fina} and \eqref{sup-tx},
\begin{equation}\label{t-linfl2}
\sup_{0\leq t\leq T}\int\zeta_x^2\,dx
+\int_0^T\!\int\zeta_t^2\,dxdt\leq C_0.
\end{equation}
Equation \eqref{perturb}$_3$ gives
\begin{align*}
\tilde\kappa(\theta^\beta\zeta_x)_x
&=c_{\nu}v\zeta_t+\tilde\kappa\frac{\theta^\beta\zeta_xv_x}{v}
-\tilde\kappa v\left[\left(\frac{\theta^\beta}{v}
-\frac{\Theta^\beta}{V}\right)\Theta_x\right]_x\\
&\quad+vp\psi_x+v(p-p_+)U_x
-\mu u_x^2+\frac{v\mu_1 }V U_x^2+v\widetilde R_2.
\end{align*}
By \eqref{high-deri}, \eqref{theta-fina} and \eqref{t-linfl2},
\begin{align}
\int_0^T\|(\theta^\beta\zeta_x)_x\|_2^2dt
&\leq C_0+C\int_0^T\!\int\theta^{2\beta}\zeta_x^2v_x^2\,dxdt\notag\\
&\leq C_0+C\sup_t\|v_x(t)\|_2^2
\int_0^T\|\theta^\beta\zeta_x\|_\infty^2dt\notag\\
&\leq C_0+C\int_0^T\|\theta^\beta\zeta_x\|_2
\|(\theta^\beta\zeta_x)_x\|_2dt\notag\\
&\leq\frac12\int_0^T\|(\theta^\beta\zeta_x)_x\|_2^2dt
+C\int_0^T\!\int\theta^{2\beta}\zeta_x^2\,dxdt+C_0.
\label{L3.911}
\end{align}
Absorbing the first term on the right-hand side of \eqref{L3.911} into
its left-hand side yields
\begin{equation}\label{L3.92}
\int_0^T\|(\theta^\beta\zeta_x)_x\|_2^2dt
+\int_0^T\|\zeta_x\|_\infty^2dt\leq C_0.
\end{equation}
Finally,
\begin{align*}
\int_0^T\!\int\zeta_{xx}^2\,dxdt
&=\int_0^T\!\int\left\{\theta^{-\beta}(\theta^\beta\zeta_x)_x
-\beta\theta^{-1}\theta_x\zeta_x\right\}^2\,dxdt\\
&\leq C_0+C\int_0^T\!\int\theta_x^2\zeta_x^2\,dxdt\\
&\leq C_0+C\sup_t\|\zeta_x(t)\|_2^2\int_0^T\|\zeta_x\|_\infty^2dt
+C\int_0^T\|\Theta_x\|_\infty^2\|\zeta_x\|_2^2dt\leq C_0.
\end{align*}
The displayed $\zeta_{xx}$ estimate together with \eqref{t-linfl2}
establishes \eqref{zetag}.
\end{proof}

\begin{proof}[Proof of Lemma~\ref{lemma5}]
For $T<\infty$,
\[
\int_0^T\!\int(\phi^2+\psi^2+\zeta^2)w^2\,dxdt
\le\sup_{0\le t\le T}\|(\phi,\psi,\zeta)(t)\|_2^2
                  \int_0^T\frac{dt}{1+t}
\le C_0\log(1+T).
\]
For $g$ in \eqref{w}, $g_x=w$, $g_t=2g_{xx}/c_1$ and
\[
\partial_x\left[\frac{g(\sqrt2x,t)}{\sqrt{2(1+t)}}\right]=w^2.
\]
Moreover,
\[
\left\|\frac{g(\sqrt2\,\cdot,t)}{\sqrt{2(1+t)}}\right\|_\infty
\le C(1+t)^{-1/2},\qquad
\left\|\partial_t\left[\frac{g(\sqrt2\,\cdot,t)}{\sqrt{2(1+t)}}\right]\right\|_\infty
\le C(1+t)^{-3/2}.
\]
Since
\[
R\log(v/V)+c_\nu\log(\theta/\Theta)
 =(R+c_\nu)\log(v/V)+c_\nu\log(p/p_+),
\]
we have
\begin{align*}
\phi
 &=V\left[\exp\left(
 \frac{R\log(v/V)+c_\nu\log(\theta/\Theta)-c_\nu\log(p/p_+)}{R+c_\nu}
                          \right)-1\right],\\
\zeta&=\Theta\left[\frac pP\left(1+\frac\phi V\right)-1\right].
\end{align*}
Thus
\[
|\phi|+|\zeta|
 \le C\left(|R\log(v/V)+c_\nu\log(\theta/\Theta)|+|p-p_+|\right).
\]
Differentiation gives
\begin{align*}
\partial_x[R\log(v/V)+c_\nu\log(\theta/\Theta)]
 &=R\left(\frac{\phi_x}v-\frac{\phi V_x}{vV}\right)
       +c_\nu\left(\frac{\zeta_x}\theta-\frac{\zeta\Theta_x}{\theta\Theta}\right).
\end{align*}
Subtracting the entropy equation satisfied by $(V,U,\Theta)$ from the
entropy equation for $(v,u,\theta)$ gives
\begin{equation}\label{relative-entropy-mode}
\begin{aligned}
\partial_t\left[R\log(v/V)+c_{\nu}\log(\theta/\Theta)\right]
={}&\partial_x\left[
\frac{\kappa\theta_x}{v\theta}
-\frac{\kappa_1\Theta_x}{V\Theta}\right]\\
&+\frac{\kappa\theta_x^2}{v\theta^2}
-\frac{\kappa_1\Theta_x^2}{V\Theta^2}
+\frac{\mu u_x^2}{v\theta}
-\frac{\mu_1U_x^2}{V\Theta}
-\frac{\widetilde R_2}{\Theta}.
\end{aligned}
\end{equation}
By the upper and lower bounds of $v$ and $\theta$,
\begin{align*}
\frac{\kappa\theta_x}{v\theta}
 -\frac{\kappa_1\Theta_x}{V\Theta}
&=\frac{\kappa}{v\theta}\zeta_x
 +\left(\frac{\kappa}{v\theta}
                  -\frac{\kappa_1}{V\Theta}\right)\Theta_x.
\end{align*}
Hence
\begin{align*}
\iint\left|\frac{\kappa\theta_x}{v\theta}
 -\frac{\kappa_1\Theta_x}{V\Theta}\right|^2
&\le C\iint\zeta_x^2+C\iint(\phi^2+\zeta^2)\Theta_x^2\\
&\le C+C\delta^2\iint g_x^2(\phi^2+\zeta^2).
\end{align*}
The identities for $\partial_x\log(v/V)$ and
$\partial_x\log(\theta/\Theta)$, together with the positive upper and
lower bounds for $v$, $\theta$, $V$, and $\Theta$, yield
\begin{align*}
\iint\left|\partial_x[R\log(v/V)+c_\nu\log(\theta/\Theta)]\right|^2
&\le C\iint(\phi_x^2+\zeta_x^2)
       +C\iint(\phi^2+\zeta^2)(V_x^2+\Theta_x^2)\\
&\le C+C\delta^2\iint g_x^2(\phi^2+\zeta^2).
\end{align*}
Multiplying \eqref{relative-entropy-mode} by
$[R\log(v/V)+c_\nu\log(\theta/\Theta)]g^2$ and integrating gives
\begin{align*}
&\int_0^T\left\langle
 \partial_t[R\log(v/V)+c_\nu\log(\theta/\Theta)],
 [R\log(v/V)+c_\nu\log(\theta/\Theta)]g^2\right\rangle dt\\
&=-\iint\left[\frac{\kappa\theta_x}{v\theta}
                    -\frac{\kappa_1\Theta_x}{V\Theta}\right]
       g^2\partial_x[R\log(v/V)+c_\nu\log(\theta/\Theta)]\\
&\quad-2\iint\left[\frac{\kappa\theta_x}{v\theta}
                    -\frac{\kappa_1\Theta_x}{V\Theta}\right]
       gg_x[R\log(v/V)+c_\nu\log(\theta/\Theta)]\\
&\quad+\iint\left\{\frac{\kappa}{v\theta^2}
                                    (\zeta_x^2+2\zeta_x\Theta_x)
 +\left[\frac{\kappa}{v\theta^2}
                    -\frac{\kappa_1}{V\Theta^2}\right]\Theta_x^2\right\}
             [R\log(v/V)+c_\nu\log(\theta/\Theta)]g^2\\
&\quad+\iint\left\{\frac{\mu}{v\theta}
                                    (\psi_x^2+2\psi_xU_x)
 +\left[\frac{\mu}{v\theta}
                    -\frac{\mu_1}{V\Theta}\right]U_x^2\right\}
             [R\log(v/V)+c_\nu\log(\theta/\Theta)]g^2\\
&\quad-\iint\frac{\widetilde R_2}{\Theta}
             [R\log(v/V)+c_\nu\log(\theta/\Theta)]g^2\\
&\le\eta\iint[R\log(v/V)+c_\nu\log(\theta/\Theta)]^2g_x^2
       +C_\eta\iint(\phi_x^2+\psi_x^2+\zeta_x^2)\\
&\quad+C_\eta\iint(\phi^2+\zeta^2)(V_x^2+U_x^2+\Theta_x^2)\\
&\quad+C\left(\int_0^T\|\widetilde R_2\|_1^{4/3}dt\right)^{3/4}
       \left(\int_0^T\|R\log(v/V)+c_\nu\log(\theta/\Theta)\|_\infty^4dt\right)^{1/4}.
\end{align*}
Furthermore,
\begin{align*}
&\int_0^T\|R\log(v/V)+c_\nu\log(\theta/\Theta)\|_\infty^4dt\\
&\quad\le C\sup_t\|R\log(v/V)+c_\nu\log(\theta/\Theta)\|_2^2
 \iint\left|\partial_x[R\log(v/V)+c_\nu\log(\theta/\Theta)]\right|^2\\
&\quad\le C+C\delta^2\iint g_x^2(\phi^2+\zeta^2).
\end{align*}
Taking $\eta$ small in \eqref{scalar-gaussian} yields
\begin{equation}\label{entropy-gaussian}
\int_0^T\!\int
[R\log(v/V)+c_{\nu}\log(\theta/\Theta)]^2g_x^2
\leq C+C\delta^2\int_0^T\!\int
g_x^2(\phi^2+\psi^2+\zeta^2).
\end{equation}

The mass and temperature equations give
\begin{equation}\label{pressure-difference}
\begin{aligned}
(p-p_+)_t+\frac{\gamma p}{v}\psi_x
={}&\partial_x\left\{\frac{\gamma-1}{v}
\left[\frac{\kappa\theta_x}{v}
-\frac{\kappa_1\Theta_x}{V}\right]\right\}\\
&+\frac{\gamma-1}{v^2}v_x
\left[\frac{\kappa\theta_x}{v}
-\frac{\kappa_1\Theta_x}{V}\right]
 +(\gamma-1)(v^{-1}-V^{-1})
 \left(\frac{\kappa_1\Theta_x}{V}\right)_x\\
&+(\gamma-1)\left(\frac{\mu u_x^2}{v^2}
 -\frac{\mu_1U_x^2}{V^2}\right)
 -\gamma\left(\frac p v-\frac {p_+} V\right)U_x
 -\frac{\gamma-1}{V}\widetilde R_2.
\end{aligned}
\end{equation}
The momentum equation gives
\begin{equation}\label{momentum-difference}
\psi_t+(p-p_+)_x
=\partial_x\left[\frac{\mu u_x}{v}
-\frac{\mu_1U_x}{V}\right]-\widetilde R_1.
\end{equation}
Multiply \eqref{pressure-difference} by $\psi\frac{g(\sqrt2x,t)}{\sqrt{2(1+t)}}$ and
\eqref{momentum-difference} by $(p-p_+)\frac{g(\sqrt2x,t)}{\sqrt{2(1+t)}}$. Then
\begin{align*}
&\frac12\iint\left[(p-p_+)^2+\frac{\gamma p}{v}\psi^2\right]g_x^2\\
&=\left[\int(p-p_+)\psi\frac{g(\sqrt2x,t)}{\sqrt{2(1+t)}}\right]_0^T
 -\iint(p-p_+)\psi\partial_t\left[\frac{g(\sqrt2x,t)}{\sqrt{2(1+t)}}\right]
 -\frac12\iint\left(\frac{\gamma p}{v}\right)_x\psi^2\frac{g(\sqrt2x,t)}{\sqrt{2(1+t)}}\\
&\quad+\iint\frac{\gamma-1}{v}
 \left[\frac{\kappa\theta_x}v
                       -\frac{\kappa_1\Theta_x}V\right]
                                  (\psi_x\frac{g(\sqrt2x,t)}{\sqrt{2(1+t)}}+\psi g_x^2)\\
&\quad+\iint\left[\frac{\mu u_x}v
                       -\frac{\mu_1U_x}V\right]
                         [(p-p_+)_x\frac{g(\sqrt2x,t)}{\sqrt{2(1+t)}}+(p-p_+)g_x^2]\\
&\quad-(\gamma-1)\iint\frac{v_x}{v^2}
 \left[\frac{\kappa\theta_x}v
                       -\frac{\kappa_1\Theta_x}V\right]\psi\frac{g(\sqrt2x,t)}{\sqrt{2(1+t)}}\\
&\quad-(\gamma-1)\iint\left(\frac1v-\frac1V\right)
                   \left(\frac{\kappa_1\Theta_x}V\right)_x\psi\frac{g(\sqrt2x,t)}{\sqrt{2(1+t)}}\\
&\quad-(\gamma-1)\iint\left[\frac{\mu}{v^2}
                               (\psi_x^2+2\psi_xU_x)
       +\left(\frac{\mu}{v^2}-\frac{\mu_1}{V^2}\right)U_x^2\right]
                                                      \psi\frac{g(\sqrt2x,t)}{\sqrt{2(1+t)}}\\
&\quad+\gamma\iint\left(\frac p v-\frac {p_+} V\right)U_x\psi\frac{g(\sqrt2x,t)}{\sqrt{2(1+t)}}
       +(\gamma-1)\iint\frac{\widetilde R_2}V\psi\frac{g(\sqrt2x,t)}{\sqrt{2(1+t)}}
       +\iint\widetilde R_1(p-p_+)\frac{g(\sqrt2x,t)}{\sqrt{2(1+t)}}.
\end{align*}
By the hypotheses of Lemma~\ref{lemma5},
\begin{align*}
&\iint\left\{\left|\frac{\kappa\theta_x}v
                    -\frac{\kappa_1\Theta_x}V\right|^2
       +\left|\frac{\mu u_x}v
                    -\frac{\mu_1U_x}V\right|^2+|(p-p_+)_x|^2\right\}\\
&\quad\le C\iint(\phi_x^2+\psi_x^2+\zeta_x^2)
       +C\iint(\phi^2+\zeta^2)(V_x^2+U_x^2+\Theta_x^2)\\
&\quad\le C+C\delta^2\iint g_x^2(\phi^2+\zeta^2).
\end{align*}
Therefore,
\begin{align*}
&\left|\iint\frac{\gamma-1}{v}
 \left[\frac{\kappa\theta_x}v
                    -\frac{\kappa_1\Theta_x}V\right]
                                  (\psi_x\frac{g(\sqrt2x,t)}{\sqrt{2(1+t)}}+\psi g_x^2)\right|\\
&\quad+\left|\iint\left[\frac{\mu u_x}v
                    -\frac{\mu_1U_x}V\right]
                         [(p-p_+)_x\frac{g(\sqrt2x,t)}{\sqrt{2(1+t)}}+(p-p_+)g_x^2]\right|\\
&\quad\le\eta\iint[(p-p_+)^2+\psi^2]g_x^2
                 +C_\eta+C_\eta\delta^2\iint g_x^2(\phi^2+\zeta^2).
\end{align*}
For the term containing $v_x$,
\begin{align*}
&\left|\iint\frac{v_x}{v^2}
 \left[\frac{\kappa\theta_x}v
                    -\frac{\kappa_1\Theta_x}V\right]\psi\frac{g(\sqrt2x,t)}{\sqrt{2(1+t)}}\right|\\
&\quad\le C\iint\left|\frac{\kappa\theta_x}v
                    -\frac{\kappa_1\Theta_x}V\right|^2
        +C\sup_t\|\psi\|_\infty^2\iint\phi_x^2
        +C\iint V_x^2\psi^2\frac{g(\sqrt2x,t)^2}{2(1+t)}\\
&\quad\le C+C\delta^2\iint g_x^2(\phi^2+\psi^2+\zeta^2).
\end{align*}
Also,
\begin{align*}
&\left|\left[\int(p-p_+)\psi\frac{g(\sqrt2x,t)}{\sqrt{2(1+t)}}\right]_0^T\right|
                   +\left|\iint(p-p_+)\psi\partial_t\left[\frac{g(\sqrt2x,t)}{\sqrt{2(1+t)}}\right]\right|\\
&\quad\le C\sup_t\|(\phi,\psi,\zeta)\|_2^2
          \left(1+\int_0^T(1+t)^{-3/2}dt\right)\le C.
\end{align*}
The time-decay estimates for $U_x$ and
$\left(\kappa_1\Theta_x/V\right)_x$ from Lemma~\ref{decay} and
Lemma~\ref{lem:remaining-profile-integrals} give
\begin{align*}
&\iint\left|\left(\frac1v-\frac1V\right)
        \left(\frac{\kappa_1\Theta_x}V\right)_x\psi\frac{g(\sqrt2x,t)}{\sqrt{2(1+t)}}\right|
 +\iint\left|\left(\frac p v-\frac {p_+} V\right)U_x\psi\frac{g(\sqrt2x,t)}{\sqrt{2(1+t)}}\right|\\
&\quad\le C\sup_t\|(\phi,\psi,\zeta)\|_2^2
 \int_0^T\|\frac{g(\sqrt2x,t)}{\sqrt{2(1+t)}}\|_\infty
   \left[\left\|\left(\frac{\kappa_1\Theta_x}V\right)_x\right\|_\infty
                            +\|U_x\|_\infty\right]dt\\
&\quad\le C\int_0^\infty(1+t)^{-3/2}dt\le C.
\end{align*}
For the viscous terms,
\begin{align*}
&\iint\left|\left[\frac{\mu}{v^2}
                      (\psi_x^2+2\psi_xU_x)
       +\left(\frac{\mu}{v^2}-\frac{\mu_1}{V^2}\right)U_x^2\right]
                                                          \psi\frac{g(\sqrt2x,t)}{\sqrt{2(1+t)}}\right|\\
&\quad\le C\sup_t\|\psi\|_\infty\iint\psi_x^2
       +C\iint\psi^2U_x^2
       +C\sup_t\|(\phi,\psi,\zeta)\|_2^2\int_0^T\|U_x\|_\infty^2dt\\
&\quad\le C+C\delta^2\iint g_x^2\psi^2.
\end{align*}
By \eqref{profile:residual-integral},
\begin{align*}
&\iint(|\widetilde R_2\psi|+|\widetilde R_1(p-p_+)|)\frac{g(\sqrt2x,t)}{\sqrt{2(1+t)}}
 \le C\sum_{j=1}^2\left(\int_0^T\|\widetilde R_j\|_1^{4/3}dt\right)^{3/4}
                   \left(\int_0^\infty(1+t)^{-2}dt\right)^{1/4}\le C.
\end{align*}
Finally, since
\[
\left|\left(\frac{\gamma p}{v}\right)_x\right|
\le C(|\phi_x|+|\zeta_x|)
       +C\delta(1+t)^{-1/2}e^{-c x^2/(1+t)},
\]
we obtain
\begin{align*}
&\iint\left|\left(\frac{\gamma p}{v}\right)_x\right|\psi^2\frac{g(\sqrt2x,t)}{\sqrt{2(1+t)}}\\
&\quad\le C\int_0^T\|\frac{g(\sqrt2x,t)}{\sqrt{2(1+t)}}\|_\infty
                     (\|\phi_x\|_2+\|\zeta_x\|_2)\|\psi\|_4^2dt
       +C\delta\iint\psi^2g_x^2+C\\
&\le C\sup_t\|\psi\|_2^{3/2}
       \left[\int_0^T(\|\phi_x\|_2^2+\|\zeta_x\|_2^2)dt\right]^{1/2}
       \left(\int_0^T\|\psi_x\|_2^2dt\right)^{1/4}
       \left(\int_0^\infty(1+t)^{-2}dt\right)^{1/4}\\
&\quad+C\delta\iint\psi^2g_x^2+C
 \le C+C\delta\iint\psi^2g_x^2.
\end{align*}
Taking $\eta$ small in the multiplier identity gives
\[
\iint g_x^2[(p-p_+)^2+\psi^2]
 \le C+C(\delta+\delta^2)\iint g_x^2(\phi^2+\psi^2+\zeta^2).
\]
Since $C^{-1}\le v,\theta\le C$,
\[
c(\phi^2+\zeta^2)
 \le[R\log(v/V)+c_\nu\log(\theta/\Theta)]^2+(p-p_+)^2
 \le C(\phi^2+\zeta^2).
\]
Combining \eqref{entropy-gaussian} with the multiplier estimate for
$g_x^2[(p-p_+)^2+\psi^2]$ and the displayed equivalence between
$(\phi,\zeta)$ and the entropy--pressure variables yields
\[
\iint g_x^2(\phi^2+\psi^2+\zeta^2)
 \le C+C(\delta+\delta^2)\iint g_x^2(\phi^2+\psi^2+\zeta^2).
\]
Choose $\delta$ with $C(\delta+\delta^2)\le1/2$. Then
\[
\frac12\iint g_x^2(\phi^2+\psi^2+\zeta^2)\le C.
\]
\end{proof}

\begin{proof}[Proof of Proposition~\ref{prop}]
Lemmas~\ref{basic-lemma}--\ref{txxtt} and \ref{lemma5} give
$C^{-1}\le v(x,t),\theta(x,t)\le C$ and
\begin{align*}
&\sup_{0\le t\le T}\|(\phi,\psi,\zeta)(t)\|_{H^1}^2
+\int_0^T\left[\|\phi_x\|_2^2+\|(\psi_x,\zeta_x)\|_{H^1}^2
+\|(\psi_t,\zeta_t)\|_2^2\right]dt
+\int_0^T\!\int(\phi^2+\psi^2+\zeta^2)w^2\,dxdt\le C,
\end{align*}
provided $C(M)(\alpha+\delta^r)$ is sufficiently small.
Enlarge $C$ to include the initial bounds, and choose
\[
M\ge4C,\qquad 0<\varepsilon_0,\delta_0\le1,\qquad
\varepsilon_0\log M\le\log2,\qquad
C(M)(\varepsilon_0+\delta_0^r)\le c.
\]
Then
\[
\frac1M<\frac1C\le v,\theta\le C<\frac M2,
\]
and the left-hand side of \eqref{space}, apart from
its pointwise bounds, is at most $M/2$. Thus \eqref{p-1} improves
the a priori hypothesis \eqref{space}.
\end{proof}

The continuation criterion following \eqref{local0:class} and the
improvement of \eqref{space} in \eqref{p-1} extend the local solution to
all $t\geq0$. Letting $T\to\infty$ in \eqref{p-1} gives
\eqref{main-estimate}. Set
\[
f(t)=\|\phi_x(t)\|_2^2+\|\psi_x(t)\|_2^2+\|\zeta_x(t)\|_2^2.
\]
The perturbation equations yield
\[
f'(t)=2\int\phi_x\psi_{xx}-2\int\psi_{xx}\psi_t
      -2\int\zeta_{xx}\zeta_t.
\]
The formula for $f'(t)$ and \eqref{main-estimate} give
\begin{align*}
\int_0^\infty|f'(t)|dt
&\le\int_0^\infty
 \left(\|\phi_x\|_2^2+2\|\psi_{xx}\|_2^2+\|\psi_t\|_2^2
                         +\|\zeta_{xx}\|_2^2+\|\zeta_t\|_2^2\right)dt
\le C.
\end{align*}
The inequalities \eqref{main-estimate} and
$\int_0^\infty|f'(t)|dt\le C$ imply
$f,f'\in L^1(0,\infty)$, and hence
\[
\lim_{t\to\infty}f(t)=0.
\]
The Sobolev inequality and the uniform $L^2$ bound in
\eqref{main-estimate} therefore yield
\begin{align*}
&\|(\phi,\psi,\zeta)(t)\|_\infty^2
 \le C\|(\phi,\psi,\zeta)(t)\|_2 f(t)^{1/2}
 \longrightarrow0.
\end{align*}
This completes the proof of Theorem~\ref{theorem1}.

\section{Proof of Theorem~\ref{theorem2}}
\setcounter{equation}{0}

For the composite wave \eqref{ansatz}, the perturbation equations are
\begin{equation}\label{perturb2}
\left\{
\begin{aligned}
&\phi_t-\psi_x=0,\\
&\psi_t+(p-P)_x
 =\left(\frac{\mu u_x}{v}
                   -\frac{\mu_1 U_x}{V}\right)_x+F,\\
&c_\nu\zeta_t+pu_x-PU_x
 =\left(\frac{\kappa \theta_x}{v}
                   -\frac{\kappa_1 \Theta_x}{V}\right)_x
   +\frac{\mu u_x^2}{v}
                   -\frac{\mu_1 U_x^2}{V}+G,\\
&(\phi,\psi,\zeta)(x,0)=(\phi_0,\psi_0,\zeta_0)(x),\qquad
 (\phi,\psi,\zeta)(\pm\infty,t)=0.
\end{aligned}
\right.
\end{equation}
Here $P_\pm=R\Theta_\pm^r/V_\pm^r$ and
$p^m=R\theta_-^m/v_-^m=R\theta_+^m/v_+^m$. Moreover,
\begin{align*}
F&=(P_-+P_+-P)_x
       +\left(\frac{\mu_1 U_x}{V}\right)_x-U_t^{cd}
       =-\widetilde R_1,
\end{align*}
and
\begin{align*}
G&=(p^m-P)U_x^{cd}+(P_--P)(U_-^r)_x
                         +(P_+-P)(U_+^r)_x\\
 &\quad+\frac{\mu_1 U_x^2}{V}
 +\tilde\kappa\left(\frac{\Theta^\beta\Theta_x}{V}
             -\frac{(\Theta^{cd})^\beta\Theta_x^{cd}}{V^{cd}}\right)_x
 =-\widetilde R_2.
\end{align*}
\begin{lemma}\label{basic2}
Under \eqref{space}, for sufficiently small $\alpha,\delta$,
\begin{equation}\label{basic2-estimate}
\begin{aligned}
&\sup_{0\le t\le T}\int
 \left[\frac{\psi^2}{2}+R\Theta\Phi\left(\frac vV\right)
                 +c_\nu\Theta\Phi\left(\frac\theta\Theta\right)\right]dx\\
&\quad+\int_0^T\!\int
 \left(\frac{\mu \Theta}{v\theta}\psi_x^2
       +\frac{\kappa \Theta}{v\theta^2}\zeta_x^2\right)dxdt\\
&\quad+\int_0^T\!\int
 P\left[\Phi\left(\frac{\theta V}{\Theta v}\right)
                  +\gamma\Phi\left(\frac vV\right)\right]
          \bigl((U_-^r)_x+(U_+^r)_x\bigr)\,dxdt
 \le C.
\end{aligned}
\end{equation}
The constant $C$ depends only on the fixed data in
Theorem~\ref{theorem2}.
\end{lemma}
\begin{proof}
Multiplying \eqref{perturb2}$_2$ by $\psi$ gives
\begin{equation}\label{4.17}
\begin{aligned}
&\left(\frac{\psi^2}{2}\right)_t
 +\left[(p-P)\psi
       -\left(\frac{\mu u_x}{v}
                    -\frac{\mu_1 U_x}{V}\right)\psi\right]_x
 -\frac{R\zeta}{v}\psi_x\\
&\quad-R\Theta\left(\frac1v-\frac1V\right)\phi_t
 +\frac{\mu }v\psi_x^2
 +\left(\frac{\mu }v-\frac{\mu_1 }V\right)U_x\psi_x
 =F\psi.
\end{aligned}
\end{equation}
Multiplying \eqref{perturb2}$_3$ by $\zeta/\theta$ yields
\begin{equation}\label{4.18}
\begin{aligned}
&c_\nu\frac{\zeta\zeta_t}{\theta}
 -\left[\left(\frac{\kappa \theta_x}{v}
                    -\frac{\kappa_1 \Theta_x}{V}\right)
                     \frac\zeta\theta\right]_x
 +\frac{R\zeta}{v}\psi_x+(p-P)U_x\frac\zeta\theta\\
&\quad+\frac{\kappa \Theta}{v\theta^2}\zeta_x^2
 -\frac{\kappa \zeta\Theta_x\zeta_x}{v\theta^2}
 +\left(\frac{\kappa }v-\frac{\kappa_1 }V\right)
                    \frac{\Theta\Theta_x\zeta_x}{\theta^2}\\
&\quad-\left(\frac{\kappa }v-\frac{\kappa_1 }V\right)
                    \frac{\zeta\Theta_x^2}{\theta^2}
 -\frac{\mu \zeta}{v\theta}\psi_x^2
 -\frac{2\mu U_x\zeta\psi_x}{v\theta}\\
&\quad-\left(\frac{\mu }v-\frac{\mu_1 }V\right)
                    \frac{\zeta U_x^2}{\theta}
 =G\frac\zeta\theta.
\end{aligned}
\end{equation}
Since $V_t=U_x$ and $R\Theta_t^{cd}=p^mU_x^{cd}$,
\begin{align}
-R\Theta\left(\frac1v-\frac1V\right)\phi_t
 &=\left[R\Theta\Phi\left(\frac vV\right)\right]_t
   -R\Theta_t\Phi\left(\frac vV\right)
   +\frac{P\phi^2}{vV}U_x.\label{4.19}
\end{align}
Similarly,
\begin{align*}
\left[\Theta\Phi\left(\frac\theta\Theta\right)\right]_t
 &=\frac{\zeta\zeta_t}{\theta}
                         -\Theta_t\Phi\left(\frac\Theta\theta\right).
\end{align*}
Writing $\Theta_t=\Theta_t^{cd}+(\Theta_-^r)_t+(\Theta_+^r)_t$, the
contact identity $R\Theta_t^{cd}=p^mU_x^{cd}$ and the rarefaction
equations yield
\begin{align*}
-R\Theta_t
 &=(\gamma-1)P\bigl((U_-^r)_x+(U_+^r)_x\bigr)-p^mU_x^{cd}
 \\
 &\quad+(\gamma-1)(P_--P)(U_-^r)_x
             +(\gamma-1)(P_+-P)(U_+^r)_x.
\end{align*}
Consequently,
\begin{align}
&-R\Theta_t\Phi\left(\frac vV\right)
 +\frac{P\phi^2}{vV}U_x
 +c_\nu\Theta_t\Phi\left(\frac\Theta\theta\right)
 +(p-P)U_x\frac\zeta\theta\nonumber\\
&\qquad=Q_1\bigl((U_-^r)_x+(U_+^r)_x\bigr)+Q_2,\label{composite-Q}\\
Q_1
 &=P\left[(\gamma-1)\Phi\left(\frac vV\right)
       +\frac{\phi^2}{vV}-\Phi\left(\frac\Theta\theta\right)
       +\left(\frac{\theta V}{\Theta v}-1\right)
                         \left(1-\frac\Theta\theta\right)\right]
 \nonumber\\
 &=P\left[(\gamma-1)\Phi\left(\frac vV\right)
      +\frac vV+\frac{\theta V}{\Theta v}-2
                                      -\log\frac\theta\Theta\right]
 \nonumber\\
 &=P\left[\Phi\left(\frac{\theta V}{\Theta v}\right)
                       +\gamma\Phi\left(\frac vV\right)\right]\ge0,
 \label{composite-Q1}\\
Q_2
 &=U_x^{cd}\left[\frac{P\phi^2}{vV}
       -p^m\Phi\left(\frac vV\right)
       +\frac{p^m}{\gamma-1}\Phi\left(\frac\Theta\theta\right)
       +(p-P)\frac\zeta\theta\right]\nonumber\\
 &\quad+(\gamma-1)(P_--P)(U_-^r)_x
        \left[\Phi\left(\frac vV\right)
                -\frac1{\gamma-1}\Phi\left(\frac\Theta\theta\right)\right]
 \nonumber\\
 &\quad+(\gamma-1)(P_+-P)(U_+^r)_x
        \left[\Phi\left(\frac vV\right)
                -\frac1{\gamma-1}\Phi\left(\frac\Theta\theta\right)\right].
 \label{q2}
\end{align}
Adding \eqref{4.17} and \eqref{4.18}, and using
\eqref{4.19}--\eqref{q2}, we obtain
\begin{align}
&\frac d{dt}\int
 \left[\frac{\psi^2}{2}+R\Theta\Phi\left(\frac vV\right)
                  +c_\nu\Theta\Phi\left(\frac\theta\Theta\right)\right]dx
 \nonumber\\
&\quad+\int\left[\frac{\mu \Theta}{v\theta}\psi_x^2
                +\frac{\kappa \Theta}{v\theta^2}\zeta_x^2\right]dx
 +\int Q_1\bigl((U_-^r)_x+(U_+^r)_x\bigr)dx\nonumber\\
&=-\int Q_2\,dx
 +\int\left[-\left(\frac{\mu }v-\frac{\mu_1 }V\right)
                      +\frac{2\mu \zeta}{v\theta}\right]U_x\psi_x\,dx
 \nonumber\\
&\quad+\int\left(\frac{\mu }v-\frac{\mu_1 }V\right)
                                 \frac{\zeta U_x^2}{\theta}\,dx
 \nonumber\\
&\quad+\int\left[\frac{\kappa \zeta}{v\theta^2}
       -\left(\frac{\kappa }v-\frac{\kappa_1 }V\right)
                               \frac\Theta{\theta^2}\right]\Theta_x\zeta_x\,dx
 \nonumber\\
&\quad+\int\left(\frac{\kappa }v-\frac{\kappa_1 }V\right)
                              \frac{\zeta\Theta_x^2}{\theta^2}\,dx
 +\int\left(F\psi+G\frac\zeta\theta\right)dx.
 \label{2-basic}
\end{align}
The exponential interaction estimates in
$\Omega_c\cup\Omega_-\cup\Omega_+$ yield
\begin{align*}
&|(P_--P)(U_-^r)_x|\\
&\quad\le C\bigl(|\Theta^{cd}-\theta_-^m|
                 +|V^{cd}-v_-^m|
                 +|\Theta_+^r-\theta_+^m|
                 +|V_+^r-v_+^m|\bigr)|(U_-^r)_x|\\
&\quad\le C\delta^2e^{-c(|x|+t)},\\
&|(P_+-P)(U_+^r)_x|\le C\delta^2e^{-c(|x|+t)}.
\end{align*}
Inserting the bounds for $(P_--P)(U_-^r)_x$ and
$(P_+-P)(U_+^r)_x$ into \eqref{q2} and using
\eqref{space}, we obtain
\begin{align*}
&|Q_2|\le C(M)\delta
 \left[w^2
                              +e^{-c(|x|+t)}\right](\phi^2+\zeta^2).
\end{align*}
Since $c(M)(\phi^2+\zeta^2)\le Q_1\le C(M)(\phi^2+\zeta^2)$,
\begin{align*}
&(\phi^2+\zeta^2)(\Theta_x^2+U_x^2)\\
&\quad\le C\delta^2w^2(\phi^2+\zeta^2)
       +C(M)\delta Q_1\bigl((U_-^r)_x+(U_+^r)_x\bigr).
\end{align*}
Young's inequality, the derivative bounds in
Lemma~\ref{lem:remaining-profile-integrals}, and \eqref{composite-Q1} yield
\begin{align}
&\left|\left[-\left(\frac{\mu }v-\frac{\mu_1 }V\right)
             +\frac{2\mu \zeta}{v\theta}\right]U_x\psi_x\right|
 +\left|\left(\frac{\mu }v-\frac{\mu_1 }V\right)
                                    \frac{\zeta U_x^2}{\theta}\right|
 \nonumber\\
&\quad+\left|\left[\frac{\kappa \zeta}{v\theta^2}
       -\left(\frac{\kappa }v-\frac{\kappa_1 }V\right)
                      \frac\Theta{\theta^2}\right]\Theta_x\zeta_x\right|
 +\left|\left(\frac{\kappa }v-\frac{\kappa_1 }V\right)
                              \frac{\zeta\Theta_x^2}{\theta^2}\right|
 \nonumber\\
&\le \frac{\mu \Theta}{4v\theta}\psi_x^2
       +\frac{\kappa \Theta}{4v\theta^2}\zeta_x^2
       +C(M)(\phi^2+\zeta^2)(U_x^2+\Theta_x^2)\nonumber\\
&\le \frac{\mu \Theta}{4v\theta}\psi_x^2
       +\frac{\kappa \Theta}{4v\theta^2}\zeta_x^2
       +C(M)\delta^2(\phi^2+\zeta^2)w^2
       +C(M)\delta Q_1\bigl((U_-^r)_x+(U_+^r)_x\bigr).
 \label{tilde-q}
\end{align}

For the heat-conductivity term in $G$,
\begin{align}
G_3
 &=\tilde\kappa\left(\frac{\Theta^\beta\Theta_x}{V}
                -\frac{(\Theta^{cd})^\beta\Theta_x^{cd}}{V^{cd}}\right)_x
 \nonumber\\
 &=\tilde\kappa\left[
    \frac{(\Theta_-^r)^\beta(\Theta_-^r)_x
                       +(\Theta_+^r)^\beta(\Theta_+^r)_x}{V}\right]_x
 \nonumber\\
 &\quad+\tilde\kappa\left[
 (\Theta^{cd})^\beta\Theta_x^{cd}\left(\frac1V-\frac1{V^{cd}}\right)
 +\frac{\Theta^\beta-(\Theta^{cd})^\beta}{V}\Theta_x^{cd}\right.
 \nonumber\\
 &\hspace{36mm}\left.
 +\frac{\Theta^\beta-(\Theta_-^r)^\beta}{V}(\Theta_-^r)_x
 +\frac{\Theta^\beta-(\Theta_+^r)^\beta}{V}(\Theta_+^r)_x\right]_x
 =G_3^1+G_3^2.\label{G3}
\end{align}
The rarefaction derivative estimates and the exponential interaction
estimates yield
\begin{align}
|G_3^1|
 &\le C\sum_\pm\left[|(\Theta_\pm^r)_{xx}|
                        +|(\Theta_\pm^r)_x|^2\right]
   +C\sum_\pm|(\Theta_\pm^r)_x|
                   \left[|V_x^{cd}|+\sum_\pm|(V_\pm^r)_x|\right],
 \label{G31}\\
\|G_3^1(t)\|_1
 &\le C\min\{\delta,(1+t)^{-1}\}+C\delta^2e^{-ct},
 \label{g31c}\\
|G_3^2|
 &\le C\bigl(|\Theta_{xx}^{cd}|+|\Theta_x^{cd}||V_x^{cd}|
                            +|\Theta_x^{cd}|^2\bigr)
     \sum_\pm\bigl(|V_\pm^r-v_\pm^m|+|\Theta_\pm^r-\theta_\pm^m|\bigr)
 \nonumber\\
 &\quad+C|\Theta_x^{cd}|
          \sum_\pm\bigl(|(V_\pm^r)_x|+|(\Theta_\pm^r)_x|\bigr)
 \nonumber\\
 &\quad+C\sum_\pm\left[|(\Theta_\pm^r)_{xx}|
                    +|(\Theta_\pm^r)_x|
                       \bigl(|V_x|+|(\Theta_\pm^r)_x|\bigr)\right]
                      |\Theta-\Theta_\pm^r|
 \nonumber\\
 &\quad+C\sum_\pm|(\Theta_\pm^r)_x|
                       |\Theta_x-(\Theta_\pm^r)_x|
 \le C\delta^2e^{-c(|x|+t)}.\label{g32}
\end{align}
For the other terms, we use
\[
\left|\left(\frac{\mu_1 U_x}{V}\right)_x\right|
 \le C|U_{xx}|+C|U_x|\bigl(|V_x|+\alpha|\Theta_x|\bigr).
\]
Combining \eqref{G31}--\eqref{g32} with the derivative and interaction
estimates for the remaining components of $F$ and $G$, we obtain
\begin{align*}
&\|F(t)\|_1+\|G(t)\|_1\\
&\quad\le C\delta(1+t)^{-1}
          +C\min\{\delta,(1+t)^{-1}\}+C\delta^2e^{-ct}\\
&\quad\le C\delta^{1/8}(1+t)^{-7/8},
\end{align*}
where we used $\min\{\delta,(1+t)^{-1}\}\le\delta^{1/8}(1+t)^{-7/8}$.
For $\delta=0$, $F=G=0$. Sobolev's inequality gives
\begin{align}
&\left|\int_0^t\!\int\left(F\psi+G\frac\zeta\theta\right)dxds\right|
 \le C(M)\int_0^t(\|F\|_1+\|G\|_1)\|(\psi,\zeta)\|_\infty ds
 \nonumber\\
&\quad\le C(M)\delta^{1/8}\int_0^t(1+s)^{-7/8}
                \|(\psi,\zeta)\|_2^{1/2}
                \|(\psi_x,\zeta_x)\|_2^{1/2}ds\nonumber\\
&\quad\le\frac14\int_0^t\!\int
       \left(\frac{\mu \Theta}{v\theta}\psi_x^2
             +\frac{\kappa \Theta}{v\theta^2}\zeta_x^2\right)dxds
 \nonumber\\
&\qquad+C(M)\delta^{1/6}\int_0^t(1+s)^{-7/6}
       \left[1+\int\left(\psi^2+\Phi\left(\frac\theta\Theta\right)\right)dx\right]ds.
 \label{int-F}
\end{align}
Integrating \eqref{2-basic}, using \eqref{tilde-q} and \eqref{int-F},
and taking $C(M)\delta\le1/4$, we obtain
\begin{align*}
&\int\left[\frac{\psi^2}{2}+R\Theta\Phi\left(\frac vV\right)
                      +c_\nu\Theta\Phi\left(\frac\theta\Theta\right)\right](x,t)dx\\
&\quad+\frac12\int_0^t\!\int
       \left(\frac{\mu \Theta}{v\theta}\psi_x^2
             +\frac{\kappa \Theta}{v\theta^2}\zeta_x^2\right)dxds
       +\frac34\int_0^t\!\int
                   Q_1\bigl((U_-^r)_x+(U_+^r)_x\bigr)dxds\\
&\le C+C(M)\delta\int_0^t\!\int(\phi^2+\zeta^2)w^2\,dxds\\
&\quad+C(M)\int_0^t\left[\delta e^{-cs}
                                  +\delta^{1/6}(1+s)^{-7/6}\right]\\
&\qquad\quad\times\left[1+\int
   \left(\frac{\psi^2}{2}+R\Theta\Phi\left(\frac vV\right)
                   +c_\nu\Theta\Phi\left(\frac\theta\Theta\right)\right)dx\right]ds.
\end{align*}
For $C(M)\delta^{1/6}\le1$, Gronwall's inequality yields
\begin{equation}\label{7.18}
\begin{aligned}
&\sup_{0\le s\le t}\int
 \left[\frac{\psi^2}{2}+R\Theta\Phi\left(\frac vV\right)
                   +c_\nu\Theta\Phi\left(\frac\theta\Theta\right)\right](x,s)dx\\
&\quad+\int_0^t\!\int
       \left(\frac{\mu \Theta}{v\theta}\psi_x^2
             +\frac{\kappa \Theta}{v\theta^2}\zeta_x^2\right)dxds\\
&\quad+\int_0^t\!\int Q_1\bigl((U_-^r)_x+(U_+^r)_x\bigr)dxds\\
&\le C+C(M)\delta^{1/6}
       +C(M)\delta\int_0^t\!\int(\phi^2+\zeta^2)w^2\,dxds.
\end{aligned}
\end{equation}
For the composite wave, \eqref{relative-entropy-mode} and
\eqref{pressure-difference} read
\begin{align*}
&\partial_t\left[R\log\frac vV+c_\nu\log\frac\theta\Theta\right]\\
&\quad=\left[\frac{\kappa \theta_x}{v\theta}
                  -\frac{\kappa_1 \Theta_x}{V\Theta}\right]_x
       +\frac{\kappa \theta_x^2}{v\theta^2}
                  -\frac{\kappa_1 \Theta_x^2}{V\Theta^2}
       +\frac{\mu u_x^2}{v\theta}
                  -\frac{\mu_1 U_x^2}{V\Theta}+\frac G\Theta.
\end{align*}
For the pressure difference,
\begin{align*}
&(p-P)_t+\frac{\gamma p}{v}\psi_x\\
&\quad=\left\{\frac{\gamma-1}{v}
       \left[\frac{\kappa \theta_x}v
                       -\frac{\kappa_1 \Theta_x}V\right]\right\}_x
       +\frac{\gamma-1}{v^2}v_x
       \left[\frac{\kappa \theta_x}v
                       -\frac{\kappa_1 \Theta_x}V\right]\\
&\qquad+(\gamma-1)\left(\frac1v-\frac1V\right)
                    \left(\frac{\kappa_1 \Theta_x}V\right)_x
       +(\gamma-1)\left[\frac{\mu u_x^2}{v^2}
                       -\frac{\mu_1 U_x^2}{V^2}\right]
       -\gamma\left(\frac pv-\frac PV\right)U_x+\frac{\gamma-1}{V}G.
\end{align*}
The pointwise derivative decomposition in
Lemma~\ref{lem:remaining-profile-integrals}, together with the coercivity
in \eqref{composite-Q1}, yields
\begin{align*}
&\int_0^t\!\int(\phi^2+\zeta^2)(V_x^2+\Theta_x^2+U_x^2)\,dxds\\
&\quad\le C(M)\delta^2\int_0^t\!\int(\phi^2+\zeta^2)w^2\,dxds
       +C(M)\delta\int_0^t\!\int
                      Q_1\bigl((U_-^r)_x+(U_+^r)_x\bigr)dxds.
\end{align*}
For the velocity perturbation,
\begin{align*}
&\int_0^t\!\int\psi^2(V_x^2+\Theta_x^2+U_x^2)\,dxds\\
&\quad\le C\delta^2\int_0^t\!\int\psi^2w^2\,dxds
       +C\sup_{0\le s\le t}\|\psi(s)\|_2^2
                         \int_0^t\sum_\pm\|(U_\pm^r)_x\|_\infty^2ds\\
&\quad\le C\delta^2\int_0^t\!\int\psi^2w^2\,dxds+C(M)\delta.
\end{align*}
Consequently,
\begin{align*}
&\int_0^t\!\int
 \left[\left|\frac{\kappa \theta_x}v
                    -\frac{\kappa_1 \Theta_x}V\right|^2
       +\left|\frac{\mu u_x}v
                    -\frac{\mu_1 U_x}V\right|^2+|(p-P)_x|^2\right]dxds\\
&\quad\le C(M)\int_0^t\|(\phi_x,\psi_x,\zeta_x)\|_2^2ds
       +C(M)\delta^2\int_0^t\!\int(\phi^2+\zeta^2)w^2\,dxds\\
&\qquad+C(M)\delta\int_0^t\!\int
                      Q_1\bigl((U_-^r)_x+(U_+^r)_x\bigr)dxds.
\end{align*}
In the pressure and momentum calculation of Lemma~\ref{lemma5},
\begin{align*}
&\int_0^t\!\int\left|\left(\frac{\gamma p}{v}\right)_x\right|
                  \psi^2\frac{g(\sqrt2x,s)}{\sqrt{2(1+s)}}\,dxds\\
&\quad\le C(M)\int_0^t(1+s)^{-1/2}
           (\|\phi_x\|_2+\|\zeta_x\|_2)\|\psi\|_4^2ds
       +C(M)\delta\int_0^t\!\int\psi^2w^2\,dxds\\
&\qquad+C(M)\sup_{0\le s\le t}\|\psi(s)\|_2^2
            \int_0^t(1+s)^{-1/2}\min\{\delta,(1+s)^{-1}\}\,ds\\
&\quad\le C(M)+C(M)\delta\int_0^t\!\int\psi^2w^2\,dxds.
\end{align*}
The terms containing $\left(\kappa_1\Theta_x/V\right)_x$ and $U_x$ are
controlled by the time-integrable derivative bounds in
Lemma~\ref{lem:remaining-profile-integrals}:
\begin{align*}
&\int_0^t\!\int
 \left[\left|\left(\frac1v-\frac1V\right)
                    \left(\frac{\kappa_1 \Theta_x}V\right)_x\psi\right|
      +\left|\left(\frac pv-\frac PV\right)U_x\psi\right|\right]
                     \frac{g(\sqrt2x,s)}{\sqrt{2(1+s)}}\,dxds\\
&\quad\le C(M)\sup_{0\le s\le t}\|(\phi,\psi,\zeta)(s)\|_2^2
 \int_0^t(1+s)^{-1/2}
       \left[\left\|\left(\frac{\kappa_1 \Theta_x}V\right)_x\right\|_\infty
                         +\|U_x\|_\infty\right]ds\\
&\quad\le C(M)\int_0^\infty(1+s)^{-3/2}ds\le C(M).
\end{align*}
The estimate
$\|F(s)\|_1+\|G(s)\|_1\le C\delta^{1/8}(1+s)^{-7/8}$ gives
\begin{align*}
&\int_0^t\!\int(|G\psi|+|F(p-P)|)
                         \frac{g(\sqrt2x,s)}{\sqrt{2(1+s)}}\,dxds\\
&\quad\le C(M)\delta^{1/8}\int_0^\infty(1+s)^{-11/8}ds
 \le C(M)\delta^{1/8}.
\end{align*}
The one-dimensional Sobolev inequality applied to
$R\log(v/V)+c_\nu\log(\theta/\Theta)$ gives
\begin{align*}
&\int_0^t\left\|R\log\frac vV+c_\nu\log\frac\theta\Theta\right\|_\infty^4ds\\
&\quad\le C(M)\int_0^t
       \left\|\partial_x\left[R\log\frac vV
                                +c_\nu\log\frac\theta\Theta\right]\right\|_2^2ds\\
&\quad\le C(M)\left[1+\delta^2\int_0^t\!\int(\phi^2+\zeta^2)w^2\,dxds
           +\delta\int_0^t\!\int
                      Q_1\bigl((U_-^r)_x+(U_+^r)_x\bigr)dxds\right].
\end{align*}
H\"older's inequality, the $L^1$ decay of $G$, and the
$L^4(0,t;L^\infty)$ bound for
$R\log(v/V)+c_\nu\log(\theta/\Theta)$ give
\begin{align*}
&\left|\int_0^t\!\int\frac G\Theta
              \left[R\log\frac vV+c_\nu\log\frac\theta\Theta\right]g^2\,dxds\right|\\
&\quad\le C\left[\int_0^t\|G\|_1^{4/3}ds\right]^{3/4}
           \left[\int_0^t\left\|R\log\frac vV
                        +c_\nu\log\frac\theta\Theta\right\|_\infty^4ds\right]^{1/4}\\
&\quad\le C(M)+C(M)\delta^2\int_0^t\!\int(\phi^2+\zeta^2)w^2\,dxds
           +C(M)\delta\int_0^t\!\int
                      Q_1\bigl((U_-^r)_x+(U_+^r)_x\bigr)dxds.
\end{align*}
Applying the Gaussian multipliers of Lemma~\ref{lemma5} to the
logarithmic-entropy and pressure equations, and using the estimates for
$\frac{\kappa\theta_x}{v}-\frac{\kappa_1\Theta_x}{V}$,
$\frac{\mu u_x}{v}-\frac{\mu_1U_x}{V}$, $(p-P)_x$,
$\left(\frac{\kappa_1\Theta_x}{V}\right)_x$, $U_x$, $F$, and $G$, yields
\begin{align*}
&\int_0^t\!\int
 \left[\left(R\log\frac vV+c_\nu\log\frac\theta\Theta\right)^2
                         +(p-P)^2+\psi^2\right]w^2\,dxds\\
&\quad\le C(M)+C(M)(\delta+\delta^2)
                     \int_0^t\!\int(\phi^2+\psi^2+\zeta^2)w^2\,dxds\\
&\qquad+C(M)\delta\int_0^t\!\int
                      Q_1\bigl((U_-^r)_x+(U_+^r)_x\bigr)dxds.
\end{align*}
The uniform bounds for $v$, $\theta$, $V$, and $\Theta$ imply
\begin{align*}
&c(M)(\phi^2+\zeta^2)
 \le\left(R\log\frac vV+c_\nu\log\frac\theta\Theta\right)^2+(p-P)^2
  \le C(M)(\phi^2+\zeta^2),
\end{align*}
taking $C(M)(\delta+\delta^2)\le1/2$ gives
\begin{equation}\label{7.19}
\begin{aligned}
&\int_0^t\!\int(\phi^2+\psi^2+\zeta^2)w^2\,dxds\\
&\quad\le C(M)+C(M)\int_0^t\|(\phi_x,\psi_x,\zeta_x)\|_2^2ds
          +C(M)\int_0^t\!\int
                       Q_1\bigl((U_-^r)_x+(U_+^r)_x\bigr)dxds\\
&\quad\le C(M)+C(M)\int_0^t\!\int
                       Q_1\bigl((U_-^r)_x+(U_+^r)_x\bigr)dxds.
\end{aligned}
\end{equation}
Substituting \eqref{7.19} into \eqref{7.18} and absorbing the resulting
$Q_1$ term after reducing $\delta$ gives
\begin{align*}
&\sup_{0\le s\le t}\int
 \left[\frac{\psi^2}{2}+R\Theta\Phi\left(\frac vV\right)
                   +c_\nu\Theta\Phi\left(\frac\theta\Theta\right)\right](x,s)dx\\
&\quad+\int_0^t\!\int
       \left(\frac{\mu \Theta}{v\theta}\psi_x^2
             +\frac{\kappa \Theta}{v\theta^2}\zeta_x^2\right)dxds
       +\frac12\int_0^t\!\int
                   Q_1\bigl((U_-^r)_x+(U_+^r)_x\bigr)dxds\\
&\quad\le C+C(M)(\delta^{1/6}+\delta)\le C+1.
\end{align*}
Substituting the Gaussian estimate \eqref{7.19} into the relative-entropy
inequality \eqref{7.18} and absorbing the resulting $Q_1$ term establish
\eqref{basic2-estimate}.
\end{proof}

The additional pressure term in the volume estimate is
\begin{align}
(p-P)_x
 &=-p\frac{\tilde v_x}{\tilde v}+\frac{R\zeta_x}{v}
      -\frac{R\zeta\Theta_x}{v\Theta}
      +\left(\frac pP-1\right)P_x.\nonumber
\end{align}
Hence
\begin{align}
&\left|\int_0^T\!\int
 \frac{(p/P-1)P_x}{\mu }
                   \frac{\tilde v_x}{\tilde v}\,dxdt\right|
 \nonumber\\
&\quad\le\frac18\int_0^T\!\int\frac p{\mu }
                  \left|\frac{\tilde v_x}{\tilde v}\right|^2dxdt
       +C(M)\int_0^T\!\int P_x^2(\phi^2+\zeta^2)\,dxdt
 \nonumber\\
&\quad\le\frac18\int_0^T\!\int\frac p{\mu }
                  \left|\frac{\tilde v_x}{\tilde v}\right|^2dxdt
       +C(M)\delta^2\int_0^T\!\int
         w^2(\phi^2+\zeta^2)\,dxdt
 \nonumber\\
&\qquad+C(M)\delta\int_0^T\!\int
             \bigl((U_-^r)_x+(U_+^r)_x\bigr)(\phi^2+\zeta^2)\,dxdt
 \nonumber\\
&\quad\le\frac18\int_0^T\!\int\frac p{\mu }
                  \left|\frac{\tilde v_x}{\tilde v}\right|^2dxdt
          +C(M)\delta.\label{composite-pressure-estimate}
\end{align}
The derivative estimates for $(V_\pm^r,U_\pm^r,\Theta_\pm^r)$ give
\begin{align}
&\int_0^\infty\sum_\pm
 \left(\|(V_\pm^r)_x\|_\infty^2+\|(U_\pm^r)_x\|_\infty^2
       +\|(\Theta_\pm^r)_x\|_\infty^2
       +\|(\Theta_\pm^r)_t\|_\infty^2
       +\|(\Theta_\pm^r)_{xx}\|_\infty^2\right)dt\nonumber\\
&\quad\le C\int_0^\infty\min\{\delta,(1+t)^{-1}\}^2dt
 \le C\delta.\label{composite-profile-time}
\end{align}
Since $|w_\pm|\ge c>0$,
\begin{align*}
&\int_s^t(w_\pm)_x(x,r)\,dr
 =-\int_s^t\frac{(w_\pm)_t(x,r)}{w_\pm(x,r)}\,dr
 =\log\frac{|w_\pm(x,s)|}{|w_\pm(x,t)|}\le C\delta.
\end{align*}
The relations $(U_\pm^r)_x\le C(w_\pm)_x$ and
$\int_s^t(w_\pm)_x\,dr\le C\delta$ yield
\begin{align*}
&\int_s^t\bigl((U_-^r)_x+(U_+^r)_x\bigr)(x,r)\,dr
 \le C\int_s^t\bigl((w_-)_x+(w_+)_x\bigr)(x,r)\,dr
 \le C\delta.
\end{align*}
Combining $\int_s^t((U_-^r)_x+(U_+^r)_x)(x,r)\,dr\le C\delta$ with
$\int_s^t|U_x^{cd}(x,r)|\,dr\le C\delta$ and $v^{-1}\le M$, we obtain
\begin{align*}
&\left|\int_s^t\!\int_{k+1}^{k+2}\frac{U_x}{v}\,dxdr\right|
 \le M\int_{k+1}^{k+2}\int_s^t
          \left(|U_x^{cd}|+(U_-^r)_x+(U_+^r)_x\right)drdx
 \le C(M)\delta.
\end{align*}
Cauchy--Schwarz in time and \eqref{composite-profile-time} give
\begin{align*}
&\int_s^t\!\int_k^{k+2}
 \left(|u|\sum_\pm|(\Theta_\pm^r)_t|
                    +\sum_\pm|(\Theta_\pm^r)_x|\right)dxdr\\
&\quad\le C(M)\sqrt{t-s}
 \left[\int_s^t\sum_\pm
       \left(\|(\Theta_\pm^r)_t\|_\infty^2
                       +\|(\Theta_\pm^r)_x\|_\infty^2\right)dr\right]^{1/2}
 \le C(M)\delta^{1/2}\sqrt{t-s}.
\end{align*}
The terms containing $\psi_x$ and $\zeta_x$ satisfy
\begin{align*}
&\int_s^t\!\int_k^{k+2}
 \left(|\psi_x|\sum_\pm|(\Theta_\pm^r)_x|
                  +|\zeta_x|\sum_\pm|(U_\pm^r)_x|\right)dxdr\\
&\quad\le C\left[\int_s^t\!\int(\psi_x^2+\zeta_x^2)dxdr\right]^{1/2}
 \left[\int_s^t\sum_\pm
     \left(\|(\Theta_\pm^r)_x\|_\infty^2
                           +\|(U_\pm^r)_x\|_\infty^2\right)dr\right]^{1/2}
 \le C(M)\delta^{1/2}.
\end{align*}
Expanding $\Theta_xU_x$ into its contact and rarefaction contributions
and using their decay estimates gives
\begin{align*}
&\int_0^\infty\|\Theta_xU_x\|_\infty dt\\
&\quad\le C\int_0^\infty
 \left[\delta^2(1+t)^{-3/2}
       +\delta(1+t)^{-1/2}\min\{\delta,(1+t)^{-1}\}
       +\min\{\delta,(1+t)^{-1}\}^2\right]dt
 \le C\delta.
\end{align*}
Inserting the estimates for $U_x/v$, $(\Theta_\pm^r)_t$,
$(\Theta_\pm^r)_x$, $\psi_x$, $\zeta_x$, and $\Theta_xU_x$ into
\eqref{B} yields
\[
\left|\log\frac{B(x,t)}{B(x,s)}\right|
 \le C+\alpha C(M)(1+\sqrt{t-s}).
\]
Since
\begin{align*}
&\int_0^T\!\int\zeta^2\Theta_x^2\,dxdt
 \le C\delta^2\int_0^T\!\int\zeta^2w^2\,dxdt
       +C\delta\int_0^T\!\int
          \bigl((U_-^r)_x+(U_+^r)_x\bigr)\zeta^2\,dxdt
 \le C(M)\delta,
\end{align*}
The estimates for $B$, $U_x/v$, and $\zeta^2\Theta_x^2$ verify the
hypotheses of \eqref{Y-e}--\eqref{v-ka}; hence
\eqref{Y-e}--\eqref{v-ka} hold for \eqref{ansatz}.

The relative entropy in \eqref{basic2-estimate} controls the cutoff sets:
\[
\sup_{0\le t\le T}
 \left(|\Omega_2(t)|+|\{\theta<4^{-1/q}\Theta\}(t)|\right)\le C(q).
\]
Sobolev's inequality gives
\begin{align*}
&\int_0^T\|(\phi,\psi,\zeta)(t)\|_\infty^4dt
 \le C\sup_{0\le t\le T}\|(\phi,\psi,\zeta)(t)\|_2^2
               \int_0^T\|(\phi_x,\psi_x,\zeta_x)(t)\|_2^2dt
 \le C(M).
\end{align*}
Therefore,
\begin{align*}
&\int_0^T\!\int_{\Omega_2}\psi^2|\Theta_t|\,dxdt\\
&\quad\le C\delta\int_0^T\!\int\psi^2w^2\,dxdt\\
&\qquad+C\sup_{0\le t\le T}|\Omega_2(t)|
          \left[\int_0^T\|\psi\|_\infty^4dt\right]^{1/2}
          \left[\int_0^T\sum_\pm\|(\Theta_\pm^r)_t\|_\infty^2dt\right]^{1/2}
 \le C(M)(\delta+\delta^{1/2}).
\end{align*}
On $\Omega_2\cup\{\theta<4^{-1/q}\Theta\}$, splitting
$|U_x|+|\Theta_t|+|\Theta_{xx}|$ into contact and rarefaction
contributions gives
\begin{align*}
&\int_0^T\!\int_{\Omega_2\cup\{\theta<4^{-1/q}\Theta\}}
 (\phi^2+\psi^2+\zeta^2)(|U_x|+|\Theta_t|+|\Theta_{xx}|)\,dxdt\\
&\quad\le C\delta\int_0^T\!\int
                         (\phi^2+\psi^2+\zeta^2)w^2\,dxdt\\
&\qquad+C(q)\left[\int_0^T\|(\phi,\psi,\zeta)\|_\infty^4dt\right]^{1/2}\\*
&\qquad\quad\times\left[\int_0^T\sum_\pm
   \left(\|(U_\pm^r)_x\|_\infty^2
        +\|(\Theta_\pm^r)_t\|_\infty^2
        +\|(\Theta_\pm^r)_{xx}\|_\infty^2\right)dt\right]^{1/2}
 \le C(q,M)\delta^{1/2}.
\end{align*}
Replacing the single-contact coefficient estimates in \eqref{i3},
\eqref{i10}, \eqref{i202}, \eqref{beta-4}, and \eqref{psx21} by the
composite estimates for $\psi^2|\Theta_t|$ on $\Omega_2$ and for
$(\phi^2+\psi^2+\zeta^2)(|U_x|+|\Theta_t|+|\Theta_{xx}|)$ on
$\Omega_2\cup\{\theta<4^{-1/q}\Theta\}$ preserves
\eqref{i3}, \eqref{i10}, \eqref{i202}, \eqref{beta-4}, and \eqref{psx21}
on their respective integration domains.
For the residual terms in \eqref{long} and \eqref{5.5},
\begin{align*}
&\int_0^T\left[\|\psi(\zeta-\Theta)_+\|_\infty^4
       +\|(\zeta-\Theta)_+\|_\infty^4
       +\|\psi^2\mathbf1_{\Omega_2}\|_\infty^4
       +\||\psi|^b\psi\|_\infty^4\right]dt\\
&\quad\le C\left(1+\sup_{0\le t\le T}\|\psi(t)\|_\infty^{4(b+1)}\right)
       \int_0^T(\|\psi\|_\infty^4+\|\zeta\|_\infty^4)dt
       \le C(M).
\end{align*}
H\"older's inequality then yields
\begin{align*}
&\int_0^T\!\int
 \left[|F\psi|(\zeta-\Theta)_++|G|(\zeta-\Theta)_+
       +|G|\psi^2\mathbf1_{\Omega_2}+|F||\psi|^{b+1}\right]dxdt\\
&\quad\le C(M)\left[\int_0^T
                      (\|F\|_1+\|G\|_1)^{4/3}dt\right]^{3/4}\\
&\quad\le C(M)\left[\delta^{1/6}
                        \int_0^T(1+t)^{-7/6}dt\right]^{3/4}
 \le C(M)\delta^{1/8}.
\end{align*}
On $\{\theta<4^{-1/q}\Theta\}$, \eqref{beta-3} becomes
\begin{align*}
&\int_0^T\!\int_{\{\theta<4^{-1/q}\Theta\}}
                                  \theta^{-q}u_x^2\,dxdt\\
&\quad\le C(q)\int_0^T\!\int\frac{\mu }v
                         u_x^2(\theta^{-q}-4\Theta^{-q})_+dxdt
       +C(q)\int_0^T\!\int\theta^{-1}\psi_x^2\,dxdt\\
&\qquad+C(q,M)\int_0^T\!\int_{\{\theta<4^{-1/q}\Theta\}}U_x^2\,dxdt\\
&\quad\le C(q)\int_0^T\!\int\frac{\mu }v
                         u_x^2(\theta^{-q}-4\Theta^{-q})_+dxdt
                        +C(q)+C(q,M)\delta^r.
\end{align*}
Young's inequality gives
\begin{align*}
&\left|R\int_0^T\!\int\frac{\theta^{1-q}}v u_x
                              (1-4(\theta/\Theta)^q)_+dxdt\right|\\
&\quad\le\varepsilon\int_0^T\!\int_{\{\theta<4^{-1/q}\Theta\}}
                                  \theta^{-q}u_x^2\,dxdt\\
&\qquad+C(q,\varepsilon)\int_0^T
           \|(1-4(\theta/\Theta)^q)_+\|_\infty^2
           \int_{\{\theta<4^{-1/q}\Theta\}}\theta^{1-q}\,dxdt.
\end{align*}
Using $\psi_x=u_x-U_x$, we also have
\begin{align*}
&\int_0^T\!\int\theta^{-q}\psi_x^2\,dxdt
 \le 2\int_0^T\!\int_{\{\theta<4^{-1/q}\Theta\}}
                                      \theta^{-q}u_x^2\,dxdt
       +C(q)\int_0^T\!\int\theta^{-1}\psi_x^2\,dxdt
       +C(q,M)\delta^r.
\end{align*}
Thus \eqref{new-nergy} and \eqref{beta-new} hold with the composite state
\eqref{ansatz}.

Apply the arguments of Lemmas~\ref{huang-lemma}--\ref{txxtt} with
\eqref{basic2-estimate} in place of \eqref{basic},
\eqref{composite-pressure-estimate}, and the contact--rarefaction
derivative estimates to obtain
$0<c\le v(x,t),\theta(x,t)\le C$ and
\begin{align}
&\sup_{0\le t\le T}\|(\phi,\psi,\zeta)(t)\|_{H^1}^2
 +\int_0^T\left[\|\phi_x\|_2^2
      +\|(\psi_x,\zeta_x)\|_{H^1}^2
      +\|(\psi_t,\zeta_t)\|_2^2\right]dt\le C.\label{composite-H1}
\end{align}
The terms involving rarefaction derivatives are controlled by the
coercivity of $Q_1$ and \eqref{composite-profile-time}:
\begin{align*}
&\int_0^T\!\int(\phi^2+\zeta^2)
                 \sum_\pm|(U_\pm^r)_x|^2\,dxdt
 \le C\delta\int_0^T\!\int
                 Q_1\bigl((U_-^r)_x+(U_+^r)_x\bigr)dxdt,\\
&\int_0^T\!\int\psi^2\sum_\pm|(U_\pm^r)_x|^2\,dxdt
 \le\sup_{0\le t\le T}\|\psi(t)\|_2^2
       \int_0^T\sum_\pm\|(U_\pm^r)_x\|_\infty^2dt\le C\delta.
\end{align*}
For the residual terms,
\begin{align*}
&\int_0^T\!\int(|F\psi_{xx}|+|G\zeta_t|)\,dxdt\\
&\quad\le\frac18\int_0^T\!\int
       \left(\frac{\mu }v\psi_{xx}^2+c_\nu\zeta_t^2\right)dxdt
       +C\int_0^T\!\int(F^2+G^2)\,dxdt\\
&\quad\le\frac18\int_0^T\!\int
       \left(\frac{\mu }v\psi_{xx}^2+c_\nu\zeta_t^2\right)dxdt
       +C\delta^r.
\end{align*}
For the Gaussian-weighted calculation, decomposing $V_x$, $\Theta_x$,
and $U_x$ into their contact and rarefaction parts gives
\begin{align}
&\int_0^T\!\int(\phi^2+\zeta^2)(V_x^2+\Theta_x^2+U_x^2)\,dxdt
 \nonumber\\
&\quad\le C\delta^2\int_0^T\!\int
        w^2(\phi^2+\zeta^2)\,dxdt
        +C\delta\int_0^T\!\int
                   Q_1\bigl((U_-^r)_x+(U_+^r)_x\bigr)dxdt,
 \nonumber\\
&\int_0^T\!\int\psi^2(V_x^2+\Theta_x^2+U_x^2)\,dxdt
 \nonumber\\
&\quad\le C\delta^2\int_0^T\!\int
                w^2\psi^2\,dxdt+C\delta.
 \label{composite-gaussian-errors}
\end{align}
Insert \eqref{composite-H1} and \eqref{composite-gaussian-errors} into
the Gaussian multiplier inequality used to derive \eqref{7.19} to obtain
\begin{equation}\label{composite-gaussian-final}
\int_0^T\!\int w^2
                    (\phi^2+\psi^2+\zeta^2)\,dxdt\le C.
\end{equation}
Choose $M$ larger than four times the constants in
\eqref{basic2-estimate}, \eqref{composite-H1} and
\eqref{composite-gaussian-final}, and then
choose $0<\varepsilon_0,\delta_0\le1$ so that
$\varepsilon_0\log M\le\log2$ and
$C(M)(\varepsilon_0+\delta_0^r)\le c$.
Since $c(\phi^2+\zeta^2)\le Q_1\le C(\phi^2+\zeta^2)$,
\begin{align*}
&\int_0^T\!\int
           \bigl((U_-^r)_x+(U_+^r)_x\bigr)(\phi^2+\zeta^2)\,dxdt
 \le C\int_0^T\!\int Q_1\bigl((U_-^r)_x+(U_+^r)_x\bigr)dxdt
 \le C.
\end{align*}
The continuation criterion following \eqref{local0:class}, together with
\eqref{basic2-estimate}, \eqref{composite-H1}, and
\eqref{composite-gaussian-final}, extends the local solution to all
$t\geq0$ and yields \eqref{main-estimate}. Letting $T\to\infty$ in
\eqref{basic2-estimate} gives \eqref{main-rarefaction-dissipation}.
For $f(t)=\|\phi_x(t)\|_2^2+\|\psi_x(t)\|_2^2+\|\zeta_x(t)\|_2^2$,
the derivative identity used in the proof of Theorem~\ref{theorem1} and
\eqref{composite-H1} give $f,f'\in L^1(0,\infty)$; the Sobolev inequality
therefore yields \eqref{behavior-1}.
The bounds for $F$, $G$, the Gaussian terms, and the higher-order terms
depend on the wave amplitudes only through
$\delta=\delta^{r_1}+\delta^{cd}+\delta^{r_3}$; hence the argument also
covers the case in which any component strength vanishes.

\end{document}